\documentclass[12pt,twoside,reqno]{amsart}

\usepackage{amsmath,amsthm,amscd}
\usepackage{amsfonts, amsmath, wasysym}

\usepackage{amssymb}

\usepackage[scr]{rsfso}
\usepackage[english]{babel}
\usepackage[mathscr]{euscript}

\usepackage{stmaryrd}

\usepackage[dvipsnames,svgnames,x11names]{xcolor}
\usepackage[hyphens]{url}
\usepackage[colorlinks=true,linkcolor=Maroon,citecolor=blue,urlcolor=blue,hypertexnames=false,linktocpage]{hyperref}
\usepackage{bookmark}
\usepackage{amsmath,thmtools}
\usepackage{bm}
\usepackage{mathtools}
\mathtoolsset{showonlyrefs} 

\usepackage{fancyhdr}
\usepackage{esint}
\usepackage{enumerate}
\usepackage{enumitem}

\usepackage[scr]{rsfso}
\usepackage{mathtools}
\usepackage{pictexwd,dcpic}
\usepackage{graphicx}

\def\avint{\mathop{\mathchoice{\,\rlap{-}\!\!\int}
		{\rlap{\raise.15em{\scriptstyle -}}\kern-.2em\int}
		{\rlap{\raise.09em{\scriptscriptstyle -}}\!\int}
		{\rlap{-}\!\int}}\nolimits}

\usepackage[labelfont=bf, font=small]{caption}

\newcounter{mgncount}

\declaretheorem[name=Theorem,numberwithin=section]{thm}
\declaretheorem[name=Remark,style=remark,sibling=thm]{rmk}
\declaretheorem[name=Lemma,sibling=thm]{lemma}
\declaretheorem[name=Proposition,sibling=thm]{prop}

\declaretheorem[name=Definition,style=definition,sibling=thm]{defn}
\declaretheorem[name=Corollary,sibling=thm]{cor}

\numberwithin{equation}{section}

\DeclareMathOperator{\scal}{scal}

\newcommand{\ov}{\overline}

\newcommand{\Rm}{\operatorname{Rm}}
\newcommand{\Ric}{\operatorname{Ric}}

\newcommand{\vol}{\operatorname{vol}}

\newcommand{\eps}{\varepsilon}

\newcommand{\ds}{\displaystyle}

\DeclareMathOperator{\id}{id}

\DeclareMathOperator{\diam}{diam}

\newcommand{\mc}{\mathcal}

\usepackage{dsfont}
 
\newcommand{\R}{\mathds{R}}

\newtheorem*{claimA}{Claim A}
\newtheorem*{claimB}{Claim B}

\def\restrict#1{\raise-.5ex\hbox{\ensuremath|}_{#1}}

\usepackage{scalerel}[2014/03/10]
\usepackage[usestackEOL]{stackengine}
\def\intavg{\,\ThisStyle{\ensurestackMath{%
			\stackinset{c}{0\LMpt}{c}{0\LMpt}{\SavedStyle-}{\SavedStyle\phantom{\int}}}%
		\setbox0=\hbox{$\SavedStyle\int\,$}\kern-\wd0}\int}

\usepackage{graphicx}

\numberwithin{equation}{section}

\usepackage{xcolor}
\definecolor{color-cites}{HTML}{9a2144}
\definecolor{brickred}{HTML}{e63d12}
\definecolor{royalblue}{HTML}{2054e3}
\definecolor{blavet}{HTML}{3adbcc}
\definecolor{verd}{HTML}{4a6741}

\usepackage{tikz}

\usepackage{cancel}
\usepackage{tcolorbox}

\begin{document}
	
	\title[Smoothing via Ricci flow]{Smoothing polyhedral spaces via Ricci flow}

	\author{Richard H. Bamler}
	\address{\flushleft\parbox{\linewidth}{{\bf Richard H. Bamler} \\Department of Mathematics, \\ University of California Berkeley, \\ CA 94720\\ {\href{mailto:rbamler@berkeley.edu}{rbamler@berkeley.edu}}}}

	\author{Esther Cabezas-Rivas}
	\address{\flushleft\parbox{\linewidth}{{\bf Esther Cabezas-Rivas} \\Universitat de Val\`encia\\ Departament de Matem\`atiques\\ Av. Vicent Andrés Estellés~19\\ 46100 Burjassot\\ Spain\\ {\href{mailto:esther.cabezas-rivas@uv.es}{esther.cabezas-rivas@uv.es}}}}

	\begin{abstract}
		We prove Petrunin's smoothing conjecture in all dimensions: every
		compact Euclidean polyhedral space without boundary and with
		nonnegative Alexandrov curvature is a Gromov--Hausdorff limit of
		smooth Riemannian orbifolds with geometrically nonnegative curvature.
		More strongly, the approximating metrics are positive-time slices of
		a single orbifold Ricci flow whose metric initial condition is the
		given polyhedral space. The proof rests on a new short-time existence and 
		regularization theory for Ricci flow that, under two-sided volume
		bounds and a generalized segment inequality, replaces pointwise
		lower curvature control by small scale-invariant integral control of
		the defect from a preserved curvature cone. Although it allows
		arbitrarily large pointwise violations, this theory yields an
		existence time and positive-time curvature estimates independent of
		the initial upper curvature bound. As a further
		application, it gives rigidity consequences for manifolds with small
		integral curvature defect.
	\end{abstract}
	
	

	\date{\today. The first author is supported by NSF grant DMS-2604326.
		This work has been partially supported by project PID2025-174652NB-I00, funded by MICIU/AEI
		/10.13039/501100011033 and by ERDF, EU. The second author has also been partially supported by project CIAICO/2023/035, funded by the Conselleria d’Educació, Cultura, Universitats i Ocupació. 
	}

	\maketitle

	\section{Introduction and main results}

	A central question in Alexandrov geometry is whether singular
	spaces with a \emph{synthetic} lower curvature bound admit approximations satisfying a comparable \emph{smooth} Riemannian curvature condition. 
	Kapovitch showed that topological obstructions rule out smooth approximations in general \cite{Kapo_restric}. 
	For nonnegatively curved polyhedral spaces, however, Petrunin conjectured that such smoothings are possible if one allows the smoothings to be orbifolds \cite{Petrunin03}.
	We prove Petrunin's smoothing conjecture in all dimensions, using
	a new short-time existence theory for Ricci flow under integral
	curvature bounds.

	\medskip
	
	\begin{thm}[\textbf{Smoothing Conjecture}] \label{thm:smoothing}
		Let $P$ be a compact Euclidean polyhedral space of dimension $n\geq2$
		without boundary and with nonnegative curvature in the sense of Alexandrov. 
		Then $P$ is a Gromov--Hausdorff limit of smooth Riemannian orbifolds with geometrically nonnegative curvature.
	\end{thm}
	
	\medskip
	
	Here geometrically nonnegative curvature, also called nonnegative
	cosectional curvature, means that the curvature tensor at each point
	is a nonnegative linear combination of rotated copies of the curvature tensor
	of $\mathbb S^2\times\mathbb R^{n-2}$. 
	The condition implies nonnegative curvature operator and in dimensions at most four it is equivalent to it \cite[Section~2]{3dcase}; see Section~\ref{GNNCO} for the precise definition.

	In dimension two, the conjecture follows by smoothing cone points
	while preserving nonnegative Gaussian curvature. In dimension three,
	it was proved by Lebedeva, Matveev, Pe\-trunin and Shevchishin
	\cite{3dcase}. Their argument combines explicit convex hypersurface
	smoothings near edges and vertices with short-time Ricci flow and a
	three-dimensional curvature-pinching condition. 
	These explicit smoothing and pinching arguments are not available in higher dimensions. Moreover,
	normal links may themselves be singular polyhedral spaces, so the
	smoothing construction must proceed recursively in lower dimensions.
	The metrics used to start the Ricci flow can have large negative
	curvature, which can only be controlled in an integral sense.

	Our proof shows that the approximating metrics in
	Theorem~\ref{thm:smoothing} may even be chosen as time-slices of a
	single compact orbifold Ricci flow emerging from $P$. We make this
	precise using the following definition.
	
	\begin{defn}\label{def:flow-from-metric-space}
		A Ricci flow $(M,g(t))$, $t\in(0,T)$, on an orbifold $M$
		\textbf{emerges from} a metric space $(X,d_X)$ with dense
		Riemannian regular part $X^{\mathrm{reg}} \subset X$ if it has complete
		time-slices, uniformly bounded curvature on compact subintervals
		of $(0,T)$, and a smooth embedding $\iota:X^{\mathrm{reg}}\to M$
		such that $\iota^*g(t)$ converges locally smoothly to the regular
		metric of $X^{\mathrm{reg}}$ as $t\searrow0$. We also require
		pointed Gromov--Hausdorff convergence
		\[
		(M,d_{g(t)},\iota(x))\longrightarrow(X,d_X,x)
		\qquad\text{as }t\searrow0
		\]
		for every $x\in X^{\mathrm{reg}}$.
	\end{defn}
	
	We prove the following Ricci flow version of
	Theorem~\ref{thm:smoothing}.
	
	\begin{thm}\label{thm:smoothing-flow}
		Let $P$ be a compact Euclidean polyhedral space as in Theorem~\ref{thm:smoothing}. 
		Then there exist $T>0$, a compact orbifold $M$ with
		underlying space homeomorphic to $P$, and a Ricci flow $(M,g(t))$,
		$t\in(0,T)$ with geometrically nonnegative
		curvature emerging from $P$.
	\end{thm}

	Our approach combines two new ideas. 
	The first is a short-time
	existence theorem for Ricci flow under integral lower curvature
	bounds, which is of independent interest
	(Theorem~\ref{Lp-RF-existence_short}). 
	It substantially strengthens our earlier work with Wilking
	\cite{BamlerCabezasRivasWilking} by allowing arbitrarily large
	negative curvature, provided the negative part of the curvature is
	small in a scale-invariant integral sense. Under suitable geometric
	assumptions, we obtain an existence time and positive-time curvature
	estimates independent of pointwise bounds on the initial curvature.
	The second idea is a recursive smoothing construction, in which
	we smooth the normal links and use their Ricci flows to fill in
	the polyhedral strata. A multiscale blow-up analysis shows that
	the resulting initial metrics satisfy the geometric and integral
	hypotheses of our existence theorem.
	
	We will need the following generalization of the segment inequality
	of Cheeger--Colding \cite[Theorem~2.11]{ChColSegment}.

	\begin{defn} \label{segment}
		A Riemannian orbifold $(M,g)$ is said to satisfy the \textbf{generalized $(\alpha, D, \Upsilon)$-segment inequality} if for every \(x,y\in M\) with
		$
		0<r:=d(x,y)\leq1,
		$
		there exist a probability space
		$
		(\Omega_{x,y},\mathcal F_{x,y},\nu_{x,y})
		$
		and a measurable map
		$
		\Phi_{x,y}\colon
		\Omega_{x,y}\times[0,1]\rightarrow M
		$
		such that, for every \(\omega\in\Omega_{x,y}\), the curve
		\[
		\gamma^\omega_{x,y}
		:=
		\Phi_{x,y}(\omega,\cdot)
		\]
		is a rectifiable constant-speed curve satisfying $L_g(\gamma^\omega_{x,y})
		\leq
		Dr,
		$ with
		\begin{equation}\label{eq:A3-endpoints}
		\gamma^\omega_{x,y}(0)\in B(x,\alpha r),
		\qquad
		\gamma^\omega_{x,y}(1)\in B(y,\alpha r),
		\end{equation}
		and
		\begin{equation}\label{eq:A3-pushforward}
		(\Phi_{x,y})_*
		\left(
		\nu_{x,y}\otimes
		\mathcal L^1|_{[0,1]}
		\right)
		\leq
		\Upsilon r^{-n}\mu_g.
		\end{equation}
		Here \(\mathcal L^1|_{[0,1]}\) denotes one-dimensional Lebesgue measure restricted to \([0,1]\), and \eqref{eq:A3-pushforward} is understood as an inequality of Borel measures on \(M\). 
	\end{defn}
	
	Here the curves need not be minimizing, and their endpoints may
	vary within small balls. This
	flexibility is essential for the polyhedral approximations, which do
	not satisfy a useful uniform lower Ricci bound. The bound
	\eqref{eq:A3-pushforward} lets us estimate averages along these curves
	by spatial integrals. Together with the volume and integral curvature
	bounds below, this allows us to control distance expansion along the
	Ricci flow.
	
	We first formulate the existence theorem for the cone of nonnegative
	curvature operators. This version isolates the analytic result from
	the polyhedral construction and may be useful in other smoothing and
	compactness problems.

	\begin{thm}\label{Lp-RF-existence_short} 
		Given $2 \leq n\in\mathbb N$, $v_0,D,\Upsilon>0$, and
		$\beta\in(0,\frac{1}{2})$, there are positive constants
		$\tau=\tau(n,v_0,D)$, $C=C(n,v_0,D)$, and
		$\alpha=\alpha(n,D), \eps_0=\eps_0(n,v_0,D,\beta,\Upsilon)\in(0,1]$
		with the following property.
		
		Let $0<\eps\leq\eps_0$ and let $(M^n,g)$ be a complete
		Riemannian orbifold with bounded curvature satisfying a 
		generalized $(\alpha,D,\Upsilon)$-segment inequality.
		Suppose that, for every
		$x\in M$ and $0<r\le1$, one has
		\[ \textup{\textbf{(A1)}} \quad v_0 r^n \leq \operatorname{vol}_g B_g(x,r) \leq v_0^{-1} r^n  \qquad \text{and} \qquad \textup{\textbf{(A2)}} \quad
		\int_{B_g(x,r)} \ell^{1+\beta} d\mu_g \leq \eps r^{n - 2 (1+\beta)}, \]	
		where $\ell$ denotes the negative part of the lowest eigenvalue of $\Rm_g$. Then the Ricci flow $g(t)$ with initial metric $g$ exists on the time interval 
		$[0,\tau]$ 
		and for every $t \in (0, \tau]$ we have the curvature bounds 
		\[ \Rm_{g(t)} > - C \cdot \frac{\sqrt{\eps}}{t}  \qquad \text{and} \qquad  \big|\Rm\big|_{g(t)} < \frac{C}{t}. \]
	\end{thm}
	
	The key feature of Theorem~\ref{Lp-RF-existence_short} is that its existence time is
	independent of an upper bound for the initial curvature and a lower curvature bound is only required in an integral sense.
	Moreover, the theorem converts a mere \emph{integral} lower curvature bound into a pointwise one, where the bound improves according to the integral bound.
	The theorem is therefore applicable to families whose pointwise curvature
	becomes unbounded on small regions, as happens in the polyhedral approximation.


	Theorem~\ref{Lp-RF-existence_short} also holds for geometrically nonnegative curvature, with $\ell$ replaced by the defect
	\[
	\ell_{\mathrm{geom}}(x)
	=\inf\{b\geq0:\Rm_g(x)+b\mathscr I_g(x)
	\in\mathcal C_{\mathrm{geom}\geq0}\},
	\]
	where $\mathscr I_g$ is the curvature operator of constant sectional curvature $1$ and the lower bound on the curvature operator is replaced with a bound of the form $\ell_{g(t)} < C \sqrt{\eps} t^{-1}$.
	See Section~\ref{GNNCO} for a discussion of the cone
	$\mathcal C_{\mathrm{geom}\geq0}$ and its preservation under Hamilton's
	ODE, and Theorem~\ref{Lp-RF-existence} for a detailed short-time
	existence statement.
	
	The existence theory also gives an integral version of the gap result
	in~\cite[Corollary~3]{BamlerCabezasRivasWilking}.
	\begin{cor}\label{cor:integral-curvature-gap}
		Fix $2 \leq n\in\mathbb N$, $v_0,D,\Upsilon,d_0>0$ and
		$\beta\in(0,\frac{1}{2})$, and let $\alpha$ be as in
		Theorem~\ref{Lp-RF-existence_short}. There is
		$\eps_{\mathrm{gap}}>0$ such that every closed connected Riemannian
		manifold $(M^n,g)$ satisfying the generalized
		$(\alpha,D,\Upsilon)$-segment inequality,
		$\operatorname{diam}_gM\leq d_0$, and \textup{\textbf{(A1)--(A2)}} with
		$\eps\leq\eps_{\mathrm{gap}}$ admits a smooth metric with nonnegative
		curvature operator. The analogous statement holds for geometrically
		nonnegative curvature if $\ell$ in \textup{\textbf{(A2)}} is replaced by
		$\ell_{\mathrm{geom}}$.
	\end{cor}
	

	\medskip
	
	The use of Ricci flow to smooth singular metrics is motivated by its
	ability to preserve and improve curvature conditions, as established
	in the work of Hamilton, B\"ohm--Wilking, Brendle--Schoen and Nguyen
	\cite{BW2,BrendleSchoen,HamPositiveRicci,HamSurfaces,Nguyen}.
	Ricci flow and Ricci--DeTurck flow have been
	constructed from continuous or rough metrics
	\cite{BamlerGromov,BurkhardtGuim,KochLamm,LammSimon,LeeLiu,SimonC0},
	and from incomplete surfaces and two-dimensional singular spaces
	\cite{GiesenTopping,Richard,YinConical}. Existence and curvature
	estimates independent of an initial upper curvature bound, together
	with preservation or almost-preservation of pointwise curvature
	conditions, were established in
	\cite{BamlerCabezasRivasWilking,CabezasRivasWilking,Lai,
		LeeTam,LeeTopping,SimonToppingLocal}. Such estimates also underlie
	constructions of flows from noncollapsed metric limits under lower
	curvature assumptions
	\cite{McLeodToppingPyramid,McLeodToppingGlobal,SimonNoncollapsed,
		SimonToppingMollification}. Regularity and stability of flows emerging
	from metric spaces were studied in
	\cite{DeruelleSchulzeSimon,DeruelleSchulzeSimonStability}.
	Flows from metric cones and spaces with isolated or edge-type conical
	singularities were constructed in
	\cite{BamlerChen,Deruelle,GianniotisSchulze,Lavoyer,SchulzeSimon}.
	
	Our starting point is the short-time existence theorem of Bamler,
	Cabezas-Rivas and Wilking \cite{BamlerCabezasRivasWilking}.
	Its pointwise almost-nonnegativity hypothesis, however, is not
	satisfied uniformly by the polyhedral approximations constructed here.
	Normal links may
	themselves be singular, so even constructing the initial approximations
	requires lower-dimensional Ricci flows. The successive smoothings
	introduce arbitrarily large negative curvature on smaller and smaller
	scales. This necessitates a new existence theory based on
	integral control of the negative curvature. We must control its
	nonlinear growth as well as distance expansion and collapse, without
	a uniform pointwise lower Ricci bound. The proof couples heat-kernel
	estimates with a new exponential integrability theorem and the
	generalized segment inequality.
	This theory makes the recursive construction possible and allows the
	resulting approximations to be evolved for a uniform time.
	\medskip
	
	We briefly describe the proof of
	Theorem~\ref{Lp-RF-existence_short}. Let $\ell$ denote the negative
	part of the lowest eigenvalue of the curvature operator, or the
	analogous defect for geometrically nonnegative curvature. This
	quantity satisfies the evolution bound
	\[
	(\partial_t-\Delta)\ell
	\le
	\scal\,\ell+C\ell^2
	\]
	in the barrier sense. A key observation, already central to
	\cite{BamlerCabezasRivasWilking}, is that the possible growth from
	$\scal\,\ell$ is exactly offset by the volume change
	$\partial_t d\mu_{g(t)}=-\scal\,d\mu_{g(t)}$ when differentiating integrals
	of $\ell$. Thus the potentially large positive scalar curvature
	drops out of this part of the estimate.
	
	The quadratic term $C\ell^2$ remains, and controlling it requires a
	new argument when $\ell$ is initially small only in an integral sense.
	We compare $\ell$ with Gaussian $L^{1+\beta}$ averages of the initial
	data in a bootstrap argument. The initial Morrey smallness then makes
	the contribution of $C\ell^2$ to the heat-kernel integral small enough
	to improve this comparison and obtain a pointwise lower curvature
	bound at positive times. A
	point-picking and blow-up argument gives the complementary curvature
	estimate
	$
	|\Rm_{g(t)}|\le \tfrac Ct.
	$
	These two estimates provide the curvature component of the
	bootstrap.
	
	The absence of a uniform lower Ricci curvature bound is a major
	obstacle to controlling the geometry along the flow. Even basic
	estimates for the volumes of balls and distances require new
	arguments. For the distance estimates, the main difficulty is to
	control expansion. We address this using the exponential integrability
	estimate of Section~\ref{sec:4}, which provides averaged control of
	length growth. Applying this estimate requires propagating the initial
	integral control of $\ell$ to a spacetime Morrey bound.
	Combined with the generalized segment inequality, this produces a
	curve with small length expansion and, after short
	endpoint corrections, an upper bound for the evolving distance.
	
	The preceding estimates are combined in a bootstrap argument based on
	a priori assumptions controlling the basic geometry of the evolving
	metric: upper curvature bounds, noncollapsing and control of distance
	expansion, together with pointwise and spacetime integral bounds for
	$\ell$. These assumptions are propagated by strictly improving their
	bounds in Propositions~\ref{prop:upper-curvature-improvement}--\ref{improved_Gaussian}.
	Distance control preserves noncollapsing at the parabolic
	scale; noncollapsing and curvature control yield the required
	heat-kernel bounds; the Gaussian estimates improve the pointwise bound
	for $\ell$; and the spacetime Morrey estimate feeds back into the
	exponential integrability argument. The auxiliary constants can be
	ordered so that the requirements on each depend only on previously
	fixed constants. All improvement requirements can therefore be
	satisfied simultaneously, and the bootstrap yields a uniform existence
	time independent of the initial upper curvature bound. Moreover, if
	the constants in the initial Morrey bounds for $\ell$ tend to zero
	along a sequence, then, at
	every fixed positive time, the corresponding curvature tensors
	approach the preserved cone uniformly. This integral recovery
	principle is the central analytic advance of the paper and the main
	mechanism behind the proof of
	Theorem~\ref{thm:smoothing}.
	
	It remains to produce geometric approximations satisfying the
	hypotheses of this theorem. We construct smooth orbifold metrics
	converging to the polyhedral space by resolving its singular strata
	inductively through their normal links. The convergence theorems of
	Hamilton and B\"ohm--Wilking \cite{BW2,HamSurfaces} provide the round
	limits of the normalized link flows needed for the fillings.
	The successive filling scales
	are chosen sufficiently far apart, and a multiscale compactness
	argument gives uniform volume and segment bounds and establishes the
	Morrey bound \textup{\textbf{(A2)}} with $\eps\to0$ along the approximating
	sequence. Applying the analytic theorem
	on a common time interval and passing to a positive-time limit
	produces a flow with geometrically nonnegative curvature; the distance
	estimates finally identify the original polyhedral space as its metric initial
	condition.
	\medskip
	
	Beyond resolving the Smoothing Conjecture, the construction developed here opens seve\-ral directions for further research. Our smoothing flow
	depends a priori on the auxiliary choices involved
	in resolving the singular strata. A first question is whether it is
	nevertheless canonical, up to time-preserving isometry, and whether
	it depends continuously on the initial polyhedral metric in the
	Gromov--Hausdorff topology. Both properties are known for compact
	Alexandrov surfaces by work of Richard \cite{Richard}. In view of the
	regularity and stability theory developed by
	Deruelle--Schulze--Simon, it is also natural to ask whether the
	positive-time flow can be characterized intrinsically among flows
	with the same metric initial condition
	\cite{DeruelleSchulzeSimon,DeruelleSchulzeSimonStability}.
	
	Such conclusions cannot be expected from curvature decay and
	Gromov--Hausdorff convergence alone. Topping exhibited sequences of
	Ricci flows with uniform $C/t$ curvature decay for which the limiting
	flow does not retain the expected initial metric
	\cite{ToppingLoss}. Thus the distance estimates obtained from the
	generalized segment inequality are not merely a technical
	ingredient. They supply the additional control needed to identify
	the metric initial condition.

	The integral existence theory also raises analytic questions of independent interest:  whe\-ther the  exponent $1+\beta$ in
	Theorem~\ref{Lp-RF-existence_short}  can be replaced by the critical
	exponent $1$, possibly with an additional smallness
	condition across scales, and whether the generalized segment inequality follows from simpler intrinsic geometric or metric-measure hypotheses.
	Progress on either question could
	extend the method beyond the polyhedral category and contribute to a general existence, uniqueness
	and stability theory for Ricci flows emerging from singular spaces.
	\medskip

	The paper is organized as follows. Section~\ref{sec:2} recalls the polyhedral
	geometry and heat kernel estimates used later. Section~\ref{GNNCO} studies geometrically nonnegative
	curvature, proves its invariance under Hamilton's ODE and derives the
	evolution inequality for its defect. Section~\ref{sec:4} establishes the
	exponential integrability estimate underlying the distance control. 
	Section~\ref{sec:integral-almost-nonnegative-curvature} proves the short-time existence theorem, its
	curvature-operator variant and its extension to effective orbifolds.
	Section~\ref{sec:approximators} constructs the polyhedral approximations, establishes the
	multiscale compactness and integral-defect estimates, and completes
	the proof of Theorem~\ref{thm:smoothing}.
	
	\medskip
	\noindent\textbf{Statement on the use of AI.}
A draft from 2017 already contained a complete, though unpolished,
proof of Theorem~\ref{Lp-RF-existence_short} and a clear outline
of the construction in Section~\ref{sec:approximators}.
The remaining work in the construction was largely a matter of
writing out its rather tedious technical details.
In preparing the present version, AI assistance was used to identify
possible improvements to the exposition, particularly in the
technical description of this construction, but not to develop
new mathematical ideas. All revisions were independently verified
by the authors, who take full responsibility for the manuscript.

	\section{Background material} \label{sec:2}
	
	We first recall the local structure of polyhedral spaces. Their links
	describe the geometry transverse to the faces and will enter the
	smoothing construction. We then recall the heat kernel estimate used
	in the Ricci flow argument.
	
	\subsection{Polyhedral spaces}
	
	\begin{defn} Let $(X, d)$ be a metric space.
		\begin{itemize}
			\item[(a)] $(X, d)$ is called a {\it length space} if the distance between any two points
			coincides with the infimum of lengths of curves joining these points.
			\item[(b)] A complete length space is called a $\kappa$-{\it polyhedral space} if it admits a locally finite triangulation such that each $m$-simplex is (globally) isometric to a simplex in the simply connected $m$-manifold of constant curvature $\kappa$. For $\kappa =-1, 0$ or  $1$ we talk about a hyperbolic, Euclidean or spherical polyhedral space, respectively.
			\item[(c)] If furthermore the space is a topological manifold, we refer to it as a {\it polyhedral manifold}.
		\end{itemize}
		
	\end{defn}
	
	Unless otherwise stated, a polyhedral space will mean a Euclidean
	polyhedral space. We next describe its singular strata and their links.
	\begin{defn} Let $(X, d)$ be a polyhedral space.
		\begin{itemize}
			\item[(a)] The {\it dimension} of $X$ is the maximal dimension of a simplex in any triangulation. A triangulation is called {\it pure} if every simplex is a face of a simplex of maximal dimension.
			\item[(b)] A point $x$ in an $n$-dimensional polyhedral space is called a {\it regular point} if it has a neighborhood isometric to an open subset of $\R^n$. Otherwise, it is called a {\it singular point}.
			\item[(c)] The singular locus of a polyhedral metric is naturally stratified. A singular point is said to have {\it codimension
				$k$} if its tangent cone is isometric to a direct product $\R^{n-k} \times C$, where $C$ is a $k$-dimensional
			polyhedral cone, yet there is no such product for $\R^{n-k+1}$ and a $(k-1)$-dimensional
			polyhedral cone.
			\item[(d)] Let $\Delta$ be a simplex and $p$ a point in its relative
			interior. The {\it link} $\operatorname{Link}\Delta$ is obtained by
			taking the unit tangent vectors at $p$ perpendicular to $\Delta$ in
			each incident simplex and gluing these spherical normal faces using
			the face identifications near $p$.
			\item[(e)] For a pure $n$-dimensional triangulation, the {\it boundary}
			is the subcomplex formed by the $(n-1)$-simplices incident to exactly
			one $n$-simplex, together with all their faces.
			
		\end{itemize}
	\end{defn}
	
	Links carry the spherical angle metrics of their normal faces and
	are independent, up to isometry, of the chosen interior point.
	If $\Delta$ is a $k$-simplex with nonempty link and $p$ lies in its
	relative interior, then
	\[
	T_pX\cong\mathbb R^k\times C(\operatorname{Link}\Delta).
	\]
	Thus the link describes the normal cone along the simplex.
	
	A Euclidean polyhedral space of dimension $n\geq2$ has nonnegative
	Alexandrov curvature if and only if its triangulation is pure, each
	$(n-1)$-simplex belongs to one or two $n$-simplices, the link of every
	simplex of codimension at least two is connected, and each
	codimension-two link is a circle of length at most $2\pi$ or a closed
	interval of length at most $\pi$; see~\cite[Theorem~4.2.14]{AlexanderKapovitchPetrunin}.
	In this case, the boundary defined above coincides with the boundary
	in the sense of Alexandrov spaces. Our main objects will be compact
	spaces without boundary, so only the circle case occurs. Under the
	stated combinatorial conditions, the bound $2\pi$ on the cone angles
	is therefore both necessary and sufficient for nonnegative
	Alexandrov curvature.
	
	There is a topological restriction on smoothing by manifolds. If a
	polyhedral space is a noncollapsed limit of Riemannian manifolds with
	a uniform lower sectional curvature bound, then the links of its
	simplices of positive codimension must be homeomorphic to
	spheres~\cite[Theorem~1.3 and Corollary~1.4]{Kap}.
	This restriction concerns manifold approximations; the smoothing
	statement in this paper allows orbifolds.
	
	The conical local models also characterize polyhedral metrics without
	reference to a triangulation. We recall the following result of
	Lebedeva--Petrunin~\cite[Theorem~1.1]{LP}.
	\begin{thm}
		A compact length space $(X,d)$ is a Euclidean polyhedral space if
		and only if every point $p\in X$ has a neighborhood admitting an
		open isometric embedding into a Euclidean cone that maps $p$ to
		the vertex.
	\end{thm}
	
	\subsection{Heat kernel estimates for Ricci flows} \label{sec_hk}
	We now fix the heat kernel convention used below.
	Let $(M^n, g(t))$, $t \in [0, T]$, be a complete Ricci flow.
	Hereafter we denote by $G(x, t; y, s)$, with $x, y \in M$, $0 \leq s < t \leq T$, the heat kernel
	corresponding to the backwards heat equation coupled with the Ricci flow.
	This means that for any fixed $(x, t) \in M \times [0, T]$ we have
	\begin{equation} \label{heat_ker_def}
	\big(\partial_s + \Delta_{y,s}\big) G(x,t; \,\cdot\,, \,\cdot\,) = 0 \qquad \text{and }
	\qquad \lim_{s \nearrow t}  G(x, t; \, \cdot, s) = \delta_{x}.
	\end{equation}
	Unlike the more commonly used conjugate heat equation,
	\eqref{heat_ker_def} contains no scalar curvature term.
	Then for any fixed $(y, s) \in M \times [0, T]$ one can compute that $G(\,\cdot\,, \,\cdot\,;
	y, s)$ is the heat kernel associated to the conjugate forward equation
	\begin{equation} \label{G_conj}
	\big(\partial_t- \Delta_{x, t} - \scal_{g(t)} \big)G(\,\cdot\,, \,\cdot\,; y, s) = 0 \quad \text{and }
	\quad \lim_{t \searrow s} G(\, \cdot \,,t, y, s) = \delta_{y}.
	\end{equation}
	Hereafter $d_t$ and $d\mu_t$ will denote the Riemannian distance and the volume element, respectively, for the metric $g(t)$.

	We will use the following Gaussian upper bound. Its constants depend
	only on the curvature decay and noncollapsing bounds, so it applies
	uniformly to the approximating flows.
	
	\begin{prop} \label{thm:heat}\cite[Proposition 3.1]{BamlerCabezasRivasWilking}
		For any $A > 0$, there is a constant $C = C(n, A) < \infty$ such that the following holds: Let $(M^n, g(t))$, $t \in [0, T]$, be a complete Ricci flow  satisfying 
		\begin{equation} \label{RF curv assum A/t}
		|{\rm Rm}_{g(t)}| < \frac{A}{t} \quad  \text{and} \quad \vol_{g(t)}\big(B_{g(t)}(x,\sqrt{t})\big) > \frac{t^{n/2}}{A}
		\end{equation}
		for all $(x,t) \in M \times (0, T]$. Then
		\[G(x,t;y,s) < \frac{C }{(t-s)^{n/2} } \exp \bigg({ - \frac{d^2_s(x,y)}{C (t -s)} }\bigg)
		\qquad \text{for all} \qquad  0 \leq 2s \leq t \leq T.\]
	\end{prop}

	\section{Geometrically nonnegative curvature operators} \label{GNNCO}
	
	\subsection{Definition and convex-geometric characterization}
	Hereafter set \(n\geq 2\). Let \(V=\mathbb{R}^{n}\) be endowed with its standard
	Euclidean inner product. A two-form \(\alpha\in\Lambda^{2}V\) is called \emph{simple} (or \emph{decomposable}) if it can be written as $\alpha=u\wedge v$ for some \(u,v\in V\). 
	We identify \(\Lambda^{2}V\) with \(\mathfrak{so}(n)\) by associating
	to \(u\wedge v\) the skew-symmetric endomorphism
	$
	w\mapsto \langle v,w\rangle u-\langle u,w\rangle v.
	$
	Under this identification, the nonzero simple two-forms correspond precisely to the skew-symmetric endomorphisms of rank \(2\). We equip \(\mathfrak{so}(n)\) with the \(O(n)\)-invariant inner product
	$
	\langle A,B\rangle
	:=
	-\frac{1}{2}\operatorname{tr}(AB).
	$
	
	Let
	$
	\mathcal{A}_{n}
	=
	S^{2}_{B}\bigl(\mathfrak{so}(n)\bigr)
	$
	denote the vector space of algebraic curvature operators on \(V\).
	Thus, an element \(\mathscr R\in\mathcal{A}_{n}\) may be regarded either as a
	symmetric bilinear form on \(\mathfrak{so}(n)\), or as a self-adjoint
	endomorphism
	$
	\mathscr R\colon\mathfrak{so}(V)\rightarrow\mathfrak{so}(V),
	$
	satisfying the first Bianchi identity. We write
	$
	\mathscr R(\alpha,\beta)
	=
	\langle \mathscr R\alpha,\beta\rangle.
	$ We equip \(\mathcal A_n\), with the Hilbert--Schmidt inner product \[ \langle\mathscr R,\mathscr S\rangle := \operatorname{tr}(\mathscr R \mathscr S). \]
	
	Two algebraic curvature operators \(\mathscr R_{1}, \mathscr R_{2}\in\mathcal{A}_{n}\)
	are said to be \emph{isometric} if there exists \(A\in O(n)\) such
	that
	\[
	\mathscr R_{2}(\alpha,\beta)
	=
	\mathscr R_{1}(A^{-1}\alpha,A^{-1}\beta)
	\]
	for all \(\alpha,\beta\in\mathfrak{so}(n)\), where the action of
	\(A\) on \(\mathfrak{so}(n)\cong\Lambda^{2}V\) is induced by
	$
	A(u\wedge v)=Au\wedge Av.
	$
	
	Let \(\mathscr R_{0}\in\mathcal{A}_{n}\) denote the curvature operator of
	$
	\mathbb S^{2}\times\mathbb{R}^{n-2},
	$
	where \(\mathbb S^{2}\) is the two-dimensional sphere of constant
	sectional curvature \(1\).
	
	\begin{defn}[Geometrically nonnegative curvature]\label{def:geometrically-nonnegative}
		An algebraic curvature operator \(\mathscr R\in\mathcal{A}_{n}\) is called
		\emph{geometrically nonnegative} if it can be expressed as a finite
		linear combination, with nonnegative coefficients, of algebraic
		curvature operators that are isometric to \(\mathscr R_{0}\). 
		The set of geometrically nonnegative algebraic curvature operators
		is denoted by
		$
		\mathcal{C}_{\mathrm{geom}\geq0}.
		$
	\end{defn}
	
	We next give an equivalent algebraic description of the generators
	of \(\mathcal C_{\mathrm{geom}\geq0}\). For
	\(\alpha,\beta\in\mathfrak{so}(n)\), define the rank-one endomorphism
	$
	\alpha\otimes\beta
	\colon
	\mathfrak{so}(n)\rightarrow\mathfrak{so}(n)
	$
	by
	$
	(\alpha\otimes\beta)(\xi)
	=
	\langle\beta,\xi\rangle\alpha.
	$
	The associated bilinear form is
	\[
	(\alpha\otimes\beta)(\xi,\eta)
	:=
	\left\langle
	(\alpha\otimes\beta)(\xi),\eta
	\right\rangle
	=
	\langle\beta,\xi\rangle
	\langle\alpha,\eta\rangle.
	\]
	In particular,
	$
	(\alpha\otimes\alpha)(\xi)
	=
	\langle\alpha,\xi\rangle\alpha
	$
	and
	$
	(\alpha\otimes\alpha)(\xi,\eta)
	=
	\langle\alpha,\xi\rangle
	\langle\alpha,\eta\rangle.
	$
	If
	\(\mathscr S \in \mathcal A_n\), then it holds
	\begin{equation}\label{eq:rank-one-pairing}
	\langle\alpha\otimes\beta,\mathscr S\rangle
	=
	\langle\mathscr S\alpha,\beta\rangle
	=
	\mathscr S(\alpha,\beta).
	\end{equation}
	
	\begin{prop}[Description by simple generators] \label{prop:geom-generators}
		An algebraic curvature operator \(\mathscr R\in\mathcal{A}_{n}\) is
		geometrically nonnegative if and only if it can be written as
		\[
		\mathscr R
		=
		\sum_{i=1}^{m}
		\lambda_{i}\,
		\alpha_{i}\otimes\alpha_{i},
		\]
		where \(\lambda_{i}\geq0\) and each
		\(\alpha_{i}\in\mathfrak{so}(n)\) has rank \(2\) (equivalently, under \(\mathfrak{so}(n)\cong\Lambda^{2}\mathbb{R}^{n}\), it is a nonzero simple two-form).
	\end{prop}
	
	\begin{proof}
		Fix an orthonormal basis
		\(\{e_{1},\ldots,e_{n}\}\) of \(\mathbb{R}^{n}\), and set
		$
		\alpha_{0}:=e_{1}\wedge e_{2}\in\mathfrak{so}(n).
		$
		With our normalization, $
		\mathscr R_{0}=\alpha_{0}\otimes\alpha_{0}.$
		
		We first observe that the curvature operators isometric to \(\mathscr R_{0}\)
		are of the form
		$
		\alpha\otimes\alpha,
		$
		where \(\alpha\in\mathfrak{so}(n)\) has rank \(2\) and norm \(1\). Indeed, if \(A\in O(n)\), then
		\[
		A\cdot \mathscr R_{0}
		=
		(A\alpha_{0})\otimes(A\alpha_{0}).
		\]
		Since \(A\) preserves both the inner product and the rank,
		\(A\alpha_{0}\) has rank \(2\) and norm \(1\). Conversely, let \(\alpha\in\mathfrak{so}(n)\) have rank \(2\) and
		norm \(1\), hence there is an
		orthonormal pair \(u,v\in\mathbb{R}^{n}\) such that
		$
		\alpha=u\wedge v.
		$
		Choose \(A\in O(n)\) satisfying
		$
		Ae_{1}=u,
		\,
		Ae_{2}=v.
		$
		Then
		$
		\alpha\otimes\alpha
		=
		A\cdot \mathscr R_{0}$.
		
		Suppose now that \(\mathcal R\) is geometrically nonnegative. By
		definition and the above remark, there exist coefficients \(\mu_{i}\geq0\) and   unit
		rank-two elements \(\alpha_{i}\in\mathfrak{so}(n)\) such that
		$
		\mathscr R
		=
		\sum_{i=1}^{m}
		\mu_{i}\,\alpha_{i}\otimes\alpha_{i}.
		$
		Conversely, suppose that
		$
		\mathscr R
		=
		\sum_{i=1}^{m}
		\lambda_{i}\,\alpha_{i}\otimes\alpha_{i},
		$
		where \(\lambda_{i}\geq0\) and every
		\(\alpha_{i}\in\mathfrak{so}(n)\) has rank \(2\). 
		Then the
		normalization $
		\widehat{\alpha}_{i}
		:=
		\frac{\alpha_{i}}{|\alpha_{i}|}
		$ is well defined. Moreover,
		we can write
		\[
		\mathscr R
		=
		\sum_{i=1}^{m}
		\lambda_{i}|\alpha_{i}|^{2}
		\widehat{\alpha}_{i}\otimes\widehat{\alpha}_{i}.
		\]
		As
		$
		\lambda_{i}|\alpha_{i}|^{2}\geq0
		$ and 	\(\widehat{\alpha}_{i}\otimes\widehat{\alpha}_{i}\) is isometric to
		\(\mathscr R_{0}\),
		the statement follows.
	\end{proof}

	We denote by
	$
	\mathcal{C}_{\mathrm{sec}\geq0}
	$ and $\mathcal{C}_{\mathrm{op}\geq0}$
	the cone of algebraic curvature operators with nonnegative sectional
	curvature and of positive semidefinite curvature operators, respectively.

	\begin{lemma}[Closedness and supporting half-spaces]\label{lem:geom-cone-properties}
		The set \(\mathcal C_{\mathrm{geom}\geq0}\) is a closed, convex, and
		\(O(n)\)-invariant cone. Moreover, the homogeneous closed half-spaces containing
		\(\mathcal C_{\mathrm{geom}\geq0}\) are 
		\[
		H_{\mathscr S}
		:=
		\left\{
		\mathscr R\in\mathcal A_n:
		\langle\mathscr R,\mathscr S\rangle\geq0
		\right\},
		\]
		where \(\mathscr S\in\mathcal A_n\) has nonnegative sectional
		curvature. Consequently,
		\[
		\mathcal C_{\mathrm{geom}\geq0}
		=
		\bigcap_{\mathscr S\in\mathcal C_{\mathrm{sec}\geq0}}
		\left\{
		\mathscr R\in\mathcal A_n:
		\langle\mathscr R,\mathscr S\rangle\geq0
		\right\}.
		\]
	\end{lemma}
	
	\begin{proof}
		By Proposition~\ref{prop:geom-generators}, it follows
		immediately that
		\(\mathcal C_{\mathrm{geom}\geq0}\) is a convex cone.	The cone is \(O(n)\)-invariant because, for every \(A\in O(n)\), it holds
		$
		A\cdot(\alpha\otimes\alpha)
		=
		(A\alpha)\otimes(A\alpha),
		$
		and \(A\alpha\) has rank \(2\) whenever \(\alpha\) has rank \(2\).

		To prove that the cone is closed, let
		$
		X
		:=
		\left\{
		\alpha\otimes\alpha:
		\operatorname{rank}(\alpha)=2,\ |\alpha|=1
		\right\}
		$
		be the set of algebraic curvature operators
		isometric to the curvature operator of
		\(\mathbb S^2\times\mathbb R^{n-2}\). The set of unit rank-two elements of \(\mathfrak{so}(n)\) is compact,
		and the map
		$
		\alpha\mapsto\alpha\otimes\alpha
		$
		is continuous. Hence \(X\) is compact. Since \(\mathcal A_n\) is finite-dimensional,
		the convex hull $K:=\operatorname{conv}(X)$ is compact as well. 
		
		For every \(\alpha\in\mathfrak{so}(n)\), one has
		$
		\operatorname{tr}(\alpha\otimes\alpha)=|\alpha|^2;
		$
		in particular, every element of \(X\) has trace \(1\). Since the trace
		is linear, every \(\mathscr P\in K\) also satisfies
		$
		\operatorname{tr}(\mathscr P)=1.
		$
		
		Moreover, we have
		\[
		\mathcal C_{\mathrm{geom}\geq0}
		=
		\left\{
		t\mathscr P:
		t\geq0,\ \mathscr P\in K
		\right\}.
		\]
		Indeed, the zero operator corresponds to \(t=0\), and for every nonzero $
		\mathscr R$, absorbing the squared norms of the generators into the
		coefficients, we may assume that
		$\mathscr R_i=\alpha_i\otimes\alpha_i$ with $|\alpha_i|=1$. We then set $t:=\sum_{i=1}^{m}\lambda_i>0$, and 
		$
		\mathscr P
		:=
		\sum_{i=1}^{m}\frac{\lambda_i}{t}\mathscr R_i\in K$, where we used the notation from Proposition~\ref{prop:geom-generators}.
		
		Now let
		\[
		\mathscr R_j=t_j\mathscr P_j
		\longrightarrow\mathscr R,
		\qquad
		t_j\geq0,\quad \mathscr P_j\in K.
		\]
		As \(\operatorname{tr}(\mathscr P_j)=1\), we get
		$
		t_j
		=
		\operatorname{tr}(\mathscr R_j)
		\rightarrow
		\operatorname{tr}(\mathscr R)=:t.
		$
		In particular, \(t\geq0\). By compactness of \(K\), after passing to
		a subsequence we may assume that
		$
		\mathscr P_j\rightarrow\mathscr P
		$
		for some \(\mathscr P\in K\). Therefore,
		\[
		\mathscr R
		=
		\lim_{j\to\infty}t_j\mathscr P_j
		=
		t\mathscr P
		\in\mathcal C_{\mathrm{geom}\geq0}.
		\]

		We now characterize the closed half-spaces containing the cone. Let
		\[
		H_{\mathscr S,a}
		:=
		\left\{
		\mathscr R\in\mathcal A_n:
		\langle\mathscr R,\mathscr S\rangle\geq a
		\right\}
		\]
		be a closed half-space containing
		\(\mathcal C_{\mathrm{geom}\geq0}\). Since
		\(0\in\mathcal C_{\mathrm{geom}\geq0}\), necessarily \(a\leq0\). We claim that
		$
		\mathcal C_{\mathrm{geom}\geq0}
		\subseteq
		H_{\mathscr S,0}.
		$
		Suppose, to the contrary, that some
		\(\mathscr R\in\mathcal C_{\mathrm{geom}\geq0}\) satisfies
		\[
		\langle\mathscr R,\mathscr S\rangle<0.
		\]
		As \(\mathcal C_{\mathrm{geom}\geq0}\) is a cone,
		$
		t\mathscr R\in\mathcal C_{\mathrm{geom}\geq0}
		$
		for every \(t\geq0\). On the other hand,
		\[
		\langle t\mathscr R, \mathscr S\rangle
		=
		t\langle\mathcal R,S\rangle
		\longrightarrow-\infty
		\qquad\text{as }t\to\infty.
		\]
		Hence, for sufficiently large \(t\),
		$
		\langle t\mathscr R,\mathscr S\rangle<a,$
		contradicting
		\(\mathcal C_{\mathrm{geom}\geq0}\subseteq H_{\mathscr S,a}\).
		
		It therefore suffices to characterize the homogeneous half-spaces
		\(H_{\mathscr S}\) containing the cone. By the description of its
		generators,
		$
		\mathcal C_{\mathrm{geom}\geq0}\subseteq H_{\mathscr S}
		$
		if and only if
		$
		\langle\alpha\otimes\alpha,\mathscr S\rangle\geq0
		$
		for every rank-two \(\alpha\in\mathfrak{so}(n)\). Using
		\[
		\langle\alpha\otimes\alpha,\mathscr S\rangle
		=
		\langle\mathscr S\alpha,\alpha\rangle
		=
		\mathscr S(\alpha,\alpha),
		\]
		this condition is equivalent to nonnegative sectional
		curvature. Finally, as \(\mathcal C_{\mathrm{geom}\geq0}\) is closed and
		convex, it is the intersection of all closed half-spaces containing
		it. Thus the statement follows.
	\end{proof}

	\subsection{Preservation under Hamilton's ODE}
	
	The aim of this subsection is to show that
	\(\mathcal C_{\mathrm{geom}\geq0}\) is preserved under Hamilton's ODE
	\begin{equation}\label{eq:Hamilton-ODE}
	\frac{d\mathscr R}{dt}=2Q(\mathscr R),
	\end{equation}
	which was claimed without proof in \cite{BW2}. 
	Here 
	$
	Q(\mathscr R)
	$ denotes Hamilton's quadratic map. We denote by
	$Q(\mathscr R,\mathscr S)$ its symmetric polarization, that is,
	\[
	Q(\mathscr R,\mathscr S)
	:=
	\frac12\bigl(Q(\mathscr R+ \mathscr S)-Q(\mathscr R)-Q(\mathscr S)\bigr).
	\]
	Thus $Q(\mathscr R, \mathscr S)=Q(\mathscr S,\mathscr R)$ and $Q(\mathscr R, \mathscr R)=Q(\mathscr R)$.

	Let
	\(\alpha,\beta\in\mathfrak{so}(n)\) have rank \(2\). 	By polarization of Hamilton's quadratic formula, we have
	\begin{equation}\label{eq:Q-rank-two}
	2Q(\alpha\otimes\alpha,\beta\otimes\beta)
	=
	\langle\alpha,\beta\rangle
	\bigl(
	\alpha\otimes\beta+\beta\otimes\alpha
	\bigr)
	+
	[\alpha,\beta]\otimes[\alpha,\beta],
	\end{equation}
	where
	$
	[\alpha,\beta]
	:=
	\alpha\beta-\beta\alpha.
	$
	Notice that the right-hand side of \eqref{eq:Q-rank-two} is
	self-adjoint: the first term is explicitly symmetrized, and the
	second term is of the form \(\gamma\otimes\gamma\).

	\begin{lemma}\label{lem:Q-rank-two-generators}
		Let
		\(\alpha,\beta\in\mathfrak{so}(n)\) have rank \(2\), and let
		$
		\mathscr S\in\mathcal C_{\mathrm{sec}\geq0}
		$
		satisfy
		$
		\mathscr S(\beta,\beta)=0.
		$
		Then
		$
		\left\langle
		Q(\alpha\otimes\alpha,\beta\otimes\beta),
		\mathscr S
		\right\rangle
		\geq0.
		$
	\end{lemma}
	
	\begin{proof}
		As the assertion is homogeneous in \(\alpha\),	without loss of generality, we may assume that
		$
		|\alpha|=1.
		$ 
		If $n<4$, we extend $\alpha$, $\beta$ and $\mathscr S$ by
		zero to $\mathbb R^4$. The extended operator $\mathscr S$ still has
		nonnegative sectional curvature, since the orthogonal projection of
		a simple bivector onto $\Lambda^2\mathbb R^n$ is simple or zero for
		$n=2,3$. We may therefore assume that $n\geq4$.
		
		Since \(\alpha\) and \(\beta\) have rank \(2\), their supporting
		two-planes span a subspace of dimension at most \(4\). We may
		therefore choose an orthonormal basis
		\(\{e_i\}_{i=1}^{n}\) of \(\mathbb R^n\) such that
		$
		\alpha_{ij}=\beta_{ij}=0
		$
		whenever $i>4$ or $j>4$,
		and such that the upper \(4\times4\) blocks of \(\alpha\) and
		\(\beta\) are
		\[
		\hat \alpha
		=
		\begin{pmatrix}
		0&1&0&0\\
		-1&0&0&0\\
		0&0&0&0\\
		0&0&0&0
		\end{pmatrix}
		\qquad \text{
			and }
		\qquad 
		\hat \beta
		=
		\begin{pmatrix}
		0&a&0&b\\
		-a&0&-c&0\\
		0&c&0&d\\
		-b&0&-d&0
		\end{pmatrix}.
		\]
		Then $\langle\alpha,\beta\rangle=a$
		and the upper $4 \times 4$-block of $[\alpha, \beta]$ is of the form $\left[\begin{array}{cc}0 & -B \\  B & 0\end{array}\right]$ with $B= {\rm diag}[c, b]$.

		Consider now the family
		$
		\gamma(t)
		:=
		\beta+t[\alpha,\beta]+at^2\alpha
		\in\mathfrak{so}(n).
		$
		Its upper \(4\times4\) block is
		\[
		\hat{\gamma}(t)
		=
		\begin{pmatrix}
		0&a(1+t^2)&-ct&b\\
		-a(1+t^2)&0&-c&-bt\\
		ct&c&0&d\\
		-b&bt&-d&0
		\end{pmatrix}.
		\]
		
		We claim that \(\gamma(t)\) has rank \(2\) for every
		\(t\in\mathbb R\). In fact, assuming \((a,c)\neq(0,0)\), one computes that
		$\hat\gamma(t)\cdot(-c,ct,a(1+t^2),0)^\top = (1 + t^2)(0, 0, 0, bc - ad)^\top$. As 	${\rm rank}(\beta) = 2$ implies that $ad = bc$, we have found a non-trivial kernel element. It follows that
		$
		\operatorname{rank}\gamma(t)\leq2.
		$ Moreover, \(\gamma(t)\neq0\). Indeed, if
		\(\gamma(t)=0\), then $
		a(1+t^2)=0$, and $c=0$, a contradiction.
		Therefore, $\operatorname{rank}\gamma(t)=2$.

		Suppose now that \(a=c=0\). Then $\hat \gamma(t)$ has rank at most \(2\). Since \(\beta\) has rank \(2\),
		we have
		$
		(b,d)\neq(0,0),
		$
		and hence \(\gamma(t)\neq0\). Thus the claim follows
		also in this case.

		Since
		\(\mathscr S\in\mathcal C_{\mathrm{sec}\geq0}\), the function
		$
		f(t)
		:=
		\mathscr S(\gamma(t),\gamma(t))
		$
		is nonnegative. By hypothesis,
		$
		f(0)
		=
		\mathscr S(\beta,\beta)
		=
		0.
		$

		Since \(\mathscr S\) is symmetric and \(f\) has a minimum at \(t=0\), we get
		\[
		\begin{aligned}
		0 \leq \frac12 f''(0)
		&=
		\mathscr S(\gamma''(0),\gamma(0))
		+
		\mathscr S(\gamma'(0),\gamma'(0))
		=
		2a\,\mathscr S(\alpha,\beta)
		+
		\mathscr S([\alpha,\beta],[\alpha,\beta]).
		\end{aligned}
		\]

		On the other hand, pairing \eqref{eq:Q-rank-two} with
		\(\mathscr S\) and using \eqref{eq:rank-one-pairing}, we obtain
		\[
		\begin{aligned}
		2\left\langle
		Q(\alpha\otimes\alpha,\beta\otimes\beta),
		\mathscr S
		\right\rangle
		&=
		\langle\alpha,\beta\rangle
		\left\langle
		\alpha\otimes\beta+\beta\otimes\alpha,
		\mathscr S
		\right\rangle
		+
		\left\langle
		[\alpha,\beta]\otimes[\alpha,\beta],
		\mathscr S
		\right\rangle
		\\
		&=
		2a\,\mathscr S(\alpha,\beta)
		+
		\mathscr S([\alpha,\beta],[\alpha,\beta]),
		\end{aligned}
		\]
		from where the conclusion follows.
	\end{proof}

	\begin{prop}\label{prop:Hamilton-preservation}
		The cone
		\(\mathcal C_{\mathrm{geom}\geq0}\) of geometrically nonnegative
		curvature operators is preserved by Hamilton's ODE
		\eqref{eq:Hamilton-ODE}.
	\end{prop}
	
	\begin{proof}
		Since \(\mathcal C_{\mathrm{geom}\geq0}\) is a closed convex cone, it
		suffices to show that the vector field \(Q\) points into its tangent
		cone at every boundary point.
		
		Let
		$
		\mathscr R
		\in
		\partial\mathcal C_{\mathrm{geom}\geq0},
		$
		and let
		$
		H_{\mathscr S}
		=
		\left\{
		\mathscr T\in\mathcal A_n:
		\langle\mathscr T,\mathscr S\rangle\geq0
		\right\}
		$
		be a supporting half-space of
		\(\mathcal C_{\mathrm{geom}\geq0}\) at \(\mathscr R\). Thus
		$
		\mathcal C_{\mathrm{geom}\geq0}
		\subseteq H_{\mathscr S}$, $\langle\mathscr R,\mathscr S\rangle=0$ and, by Lemma \ref{lem:geom-cone-properties}, we have
		$
		\mathscr S\in\mathcal C_{\mathrm{sec}\geq0}.
		$
		
		Express
		\[
		\mathscr R
		=
		\sum_{i=1}^{m}
		a_i\,\alpha_i\otimes\alpha_i, \quad \text{	where } \quad a_i>0 \]
		and each
		\(\alpha_i\in\mathfrak{so}(n)\) has rank \(2\). Notice that terms with
		\(a_i=0\) have been omitted. The fact that \(H_{\mathscr S}\) is a supporting half-space at
		\(\mathscr R\) yields
		\[
		\begin{aligned}
		0
		=
		\langle\mathscr R,\mathscr S\rangle
		&=
		\sum_{i=1}^{m}
		a_i
		\langle\alpha_i\otimes\alpha_i,\mathscr S\rangle
		=
		\sum_{i=1}^{m}
		a_i\mathscr S(\alpha_i,\alpha_i).
		\end{aligned}
		\]
		Since \(\mathscr S\) has nonnegative sectional curvature and \(a_i>0\), it follows that $\mathscr S(\alpha_i,\alpha_i)=0$
		for every \(i=1,\ldots,m\). The latter allows us to apply Lemma~\ref{lem:Q-rank-two-generators}, from which we conclude that
		\[
		\begin{aligned}
		\langle Q(\mathscr R),\mathscr S\rangle
		&=
		\sum_{i,j=1}^{m}
		a_i a_j
		\left\langle
		Q(\alpha_i\otimes\alpha_i,
		\alpha_j\otimes\alpha_j),
		\mathscr S
		\right\rangle \geq 0.
		\end{aligned}
		\]
		Thus \(Q(\mathscr R)\) belongs to the tangent cone of
		\(\mathcal C_{\mathrm{geom}\geq0}\) at \(\mathscr R\). The standard
		invariance criterion for closed convex sets now implies that
		\(\mathcal C_{\mathrm{geom}\geq0}\) is preserved under Hamilton's
		ODE.
	\end{proof}

	\subsection{Evolution inequality under the Ricci flow} \label{subsec:evolution-inequality}
	
	Let $\mathscr I\in\mathcal A_n$ 
	denote the constant curvature operator of sectional curvature \(1\).
	Thus,
	$
	\operatorname{Ric}(\mathscr I)=(n-1)\operatorname{id}
	$
	and
	$
	\operatorname{scal}(\mathscr I)=n(n-1).
	$

	If $0\neq \mathscr S\in\mathcal C_{\mathrm{sec}\geq0}$, then
	$\langle\mathscr I, \mathscr S\rangle=\frac12\scal(\mathscr S)>0$. By compactness of
	the unit section of $\mathcal C_{\mathrm{sec}\geq0}$, there is
	$c_n>0$ such that
	$\langle\mathscr I,\mathscr S\rangle\geq c_n|\mathscr S|$. Hence
	$\mathscr I\in\operatorname{Int}\mathcal C_{\mathrm{geom}\geq0}$,
	and we may define
	\begin{equation}\label{eq:def-ell}
	\ell(\mathscr R)
	:=
	\min
	\left\{
	a\geq0:
	\mathscr R+a \mathscr I
	\in
	\mathcal C_{\mathrm{geom}\geq0}
	\right\}.
	\end{equation}
	Notice that  $\ell$ is finite, convex and globally Lipschitz. In fact, by Lemma~\ref{lem:geom-cone-properties},
	\[
	\ell(\mathscr R)=
	\max\left\{
	0,\,
	\max_{\substack{\mathscr S\in\mathcal C_{\mathrm{sec}\geq0}\\ |\mathscr S|=1}}
	\frac{-\langle \mathscr R, \mathscr S\rangle}{\langle\mathscr I,\mathscr S\rangle}
	\right\}.
	\]

	If
	\((M,g(t))_{t\in[0,T]}\) is a solution of the Ricci flow, we set
	$
	\ell(p,t)
	:=
	\ell\bigl(\operatorname{Rm}_{g(t)}(p)\bigr).
	$ The curvature operator evolves under the Ricci flow according to
	\begin{equation}\label{eq:Rm-evolution}
	(\nabla_t-\Delta)\operatorname{Rm}
	=
	2Q(\operatorname{Rm}),
	\end{equation}
	where \(\nabla_t\) denotes the natural space-time extension of the
	Levi--Civita connection that is compatible with the evolving metric.
	
	For symmetric bilinear forms \(A\) and \(B\) on \(\mathbb R^n\), we
	denote by
	$
	A\owedge B
	\in \mathcal A_n$
	their Kulkarni--Nomizu product, defined by
	\begin{equation}\label{eq:Kulkarni-Nomizu}
	(A\owedge B)_{ijkl}
	=
	A_{ik}B_{jl}
	+
	A_{jl}B_{ik}
	-
	A_{il}B_{jk}
	-
	A_{jk}B_{il}.
	\end{equation}
	
	With the conventions fixed above, from \cite[Lemma 2.1]{BW2}, Hamilton's quadratic map satisfies
	\begin{equation}\label{eq:Q-shift}
	Q(\mathscr R+aI)
	=
	Q(\mathscr R)
	+
	a\,\operatorname{Ric}(\mathscr R)\owedge\operatorname{id}
	+
	(n-1)a^2I,
	\end{equation}
	where we write $\Ric \owedge \id$ to denote $\Ric \owedge g$ when $g_{ij} = \delta_{ij}$.
	
	\begin{prop}\label{prop:ell-barrier}
		Along the Ricci flow, the function \(\ell\) satisfies
		\begin{equation}\label{eq:ell-barrier}
		(\partial_t-\Delta)\ell
		\leq
		\operatorname{scal}\,\ell
		+
		(n-1)(n-2)\ell^2
		\end{equation}
		in the lower-barrier sense. More precisely, for every
		\((q,\tau)\in M\times(0,T)\), there exist a spacetime neighborhood
		\(\mathcal U\) of \((q,\tau)\) and a smooth function
		\(\varphi\colon \mathcal U\to\mathbb R\) such that
		$
		\varphi\leq\ell$ on $\mathcal U$, with equality at $(q, \tau)$
		and
		\begin{equation} \label{eq_barrier}
		(\partial_t-\Delta)\varphi
		\leq
		\operatorname{scal}\,\ell
		+
		(n-1)(n-2)\ell^2 \qquad \text{ at } \quad (q,\tau).
		\end{equation}
		
	\end{prop}
	
	\begin{proof}
		Fix an arbitrary
		$
		(q,\tau)\in M\times(0,T),
		$
		and set
		\[
		a:=\ell(q,\tau),
		\qquad
		\overline{\mathscr R}
		:=
		\operatorname{Rm}_{g(\tau)}(q)+aI.
		\]
		By the definition of \(\ell\),
		$
		\overline{\mathscr R}
		\in
		\mathcal C_{\mathrm{geom}\geq0}.
		$ If \(\overline{\mathscr R}\) belongs to the interior of
		\(\mathcal C_{\mathrm{geom}\geq0}\), then the minimality of \(a\)
		implies that \(a=0\). By continuity,
		$
		\operatorname{Rm}_{g(t)}(p)
		\in
		\mathcal C_{\mathrm{geom}\geq0}
		$
		in a spacetime neighborhood of \((q,\tau)\), and hence
		$
		\ell\equiv0
		$
		there. In this case, the function
		$
		\varphi\equiv0
		$
		is a suitable lower barrier.
		
		We may therefore assume that
		$
		\overline{\mathscr R}
		\in
		\partial\mathcal C_{\mathrm{geom}\geq0}.
		$
		By Lemma~\ref{lem:geom-cone-properties}, there exists a nonzero
		operator
		$
		\mathscr S\in\mathcal C_{\mathrm{sec}\geq0}
		$
		such that
		\begin{equation}\label{eq:Rbar-S-zero}
		\langle\overline{\mathscr R},\mathscr S\rangle=0
		\end{equation}
		and	$\langle\mathscr T,\mathscr S\rangle\geq0$
		for every
		$\mathscr T\in\mathcal C_{\mathrm{geom}\geq0}$. Since \(\mathscr S\) has nonnegative sectional curvature and is
		nonzero, its scalar curvature is strictly positive, and hence
		\begin{equation}\label{eq:I-S-pairing}
		\langle I,\mathscr S\rangle
		=
		\frac12\operatorname{scal}(\mathscr S)>0.
		\end{equation}
		
		Extend \(\mathscr S\) spatially at time \(\tau\) by parallel
		transport along radial geodesics issuing from \(q\), and extend it
		in time by means of a moving orthonormal frame. At \((q,\tau)\), the
		resulting local section, still denoted by \(\mathscr S\), satisfies
		\[
		\nabla\mathscr S=0,
		\qquad
		\Delta\mathscr S=0,
		\qquad
		\nabla_t\mathscr S=0.
		\]
		
		Now	consider, in a spacetime neighborhood $\mathcal U$ of \((q,\tau)\), the function
		\begin{equation}\label{eq:def-phi}
		\varphi
		:=
		-
		\frac{
			\langle\operatorname{Rm},\mathscr S\rangle
		}{
			\langle \mathscr I,\mathscr S\rangle
		}
		=
		-
		\frac{
			2\langle\operatorname{Rm},\mathscr S\rangle
		}{
			\operatorname{scal}(\mathscr S)
		},
		\end{equation}
		which is well-defined because of \eqref{eq:I-S-pairing}. We have $	\varphi(q,\tau)=a=\ell(q,\tau)$, due to
		\eqref{eq:Rbar-S-zero}.

		We next show that \(\varphi\leq\ell\) in a sufficiently small
		neighborhood of \((q,\tau)\). At every nearby point,
		$
		\operatorname{Rm}+\ell I
		\in
		\mathcal C_{\mathrm{geom}\geq0} \subset H_{\mathscr S}$, which yields
		$
		\left\langle
		\operatorname{Rm}+\ell I,\mathscr S
		\right\rangle
		\geq0.
		$ Thus
		$
		\varphi\leq\ell,
		$
		and \(\varphi\) is a smooth lower barrier for \(\ell\) at
		\((q,\tau)\).
		
		This formulation also implies the viscosity subsolution inequality:
		if a smooth function $\psi$ touches $\ell$ from above at
		$(q,\tau)$, then $\psi-\varphi$ has a local minimum there, and hence
		\[
		(\partial_t-\Delta)\psi
		\leq
		(\partial_t-\Delta)\varphi
		\qquad\text{at }(q,\tau).
		\]
		
		It remains to show that $\varphi$ satisfies \eqref{eq_barrier} at $(q, \tau)$. With this goal, using \eqref{eq:Rm-evolution} and the chosen extension of
		\(\mathscr S\), we obtain at \((q,\tau)\)
		\begin{equation}\label{eq:phi-evolution-first}
		(\partial_t-\Delta)\varphi
		=
		-
		\frac{
			2\langle Q(\operatorname{Rm}),\mathscr S\rangle
		}{
			\langle I,\mathscr S\rangle
		}.
		\end{equation}
		We aim to estimate the right hand side of this equality. 	
		As
		$
		\overline{\mathscr R}
		=
		\operatorname{Rm}+aI
		$
		belongs to
		\(\partial\mathcal C_{\mathrm{geom}\geq0}\), and \(\mathscr S\)
		defines a supporting half-space at
		\(\overline{\mathscr R}\),
		Proposition~\ref{prop:Hamilton-preservation} gives
		\[
		\begin{aligned}
		0
		&\leq
		\langle Q(\overline{\mathscr R}),\mathscr S\rangle
		=
		\langle Q(\operatorname{Rm}),\mathscr S\rangle
		+
		a\left\langle
		\operatorname{Ric}\owedge\operatorname{id},
		\mathscr S
		\right\rangle
		+
		(n-1)a^2\langle I,\mathscr S\rangle,
		\end{aligned}
		\]
		where we have used \eqref{eq:Q-shift} and $\operatorname{Ric}
		=
		\operatorname{Ric}(\operatorname{Rm})$.
		Substituting this estimate into
		\eqref{eq:phi-evolution-first} gives
		\begin{equation}\label{eq:phi-evolution-second}
		(\partial_t-\Delta)\varphi
		\leq
		\frac{
			2a
			\left\langle
			\operatorname{Ric}\owedge\operatorname{id},
			\mathscr S
			\right\rangle
		}{
			\langle I,\mathscr S\rangle
		}
		+
		2(n-1)a^2.
		\end{equation}
		
		It remains to estimate the first term on the right-hand side. Since
		$
		\overline{\mathscr R}
		\in
		\mathcal C_{\mathrm{geom}\geq0},
		$
		we can write
		\begin{equation}\label{eq:Rbar-decomposition}
		\overline{\mathscr R}
		=
		\sum_{\mu=1}^{N}
		\lambda_\mu\,
		\alpha_\mu\otimes\alpha_\mu,
		\qquad
		\lambda_\mu>0,
		\end{equation}
		where
		$
		\alpha_\mu=x_\mu\wedge y_\mu
		$
		and \(x_\mu,y_\mu\) are orthonormal for each $\mu$. For the generator
		$
		\alpha_\mu
		\otimes
		\alpha_\mu,
		$
		the corresponding Ricci tensor is the orthogonal projection onto
		\(\operatorname{span}\{x_\mu,y_\mu\}\). Therefore,
		\begin{equation}\label{eq:Ric-expansion}
		\operatorname{Ric}
		+
		(n-1)a\,\operatorname{id}
		=
		\sum_{\mu=1}^{N}
		\lambda_\mu
		\left(
		x_\mu^\flat\otimes x_\mu^\flat
		+
		y_\mu^\flat\otimes y_\mu^\flat
		\right).
		\end{equation}
		Taking the trace gives
		\begin{equation}\label{eq:scalar-expansion}
		\operatorname{scal}
		+
		n(n-1)a
		=
		2\sum_{\mu=1}^{N}\lambda_\mu.
		\end{equation}
		
		On the other hand,
		\eqref{eq:Rbar-S-zero} and
		\eqref{eq:Rbar-decomposition} give
		\[
		\begin{aligned}
		0
		&=
		\langle\overline{\mathscr R},\mathscr S\rangle
		=
		\sum_{\mu=1}^{N}
		\lambda_\mu
		\langle
		\alpha_\mu\otimes\alpha_\mu,
		\mathscr S
		\rangle
		=
		\sum_{\mu=1}^{N}
		\lambda_\mu
		\mathscr S(\alpha_\mu,\alpha_\mu).
		\end{aligned}
		\]
		Every term in the last sum is nonnegative, because
		\(\mathscr S \in\mathcal C_{\mathrm{sec}\geq0}\). As
		\(\lambda_\mu>0\), it follows that $\mathscr S(\alpha_\mu,\alpha_\mu)=0$
		for every \(\mu\); equivalently, for the sectional curvatures we get
		$
		K_{\mathscr S}(x_\mu,y_\mu)=0.
		$ Fix \(\mu\), and extend \(x_\mu,y_\mu\) to an orthonormal basis
		$
		x_\mu,y_\mu,z_3,\ldots,z_n.
		$ of $\mathbb R^n$. Then we can estimate
		
		\begin{equation} \label{estim_Ric}
		\operatorname{Ric}_{\mathscr S}(x_\mu,x_\mu)
		+
		\operatorname{Ric}_{\mathscr S}(y_\mu,y_\mu)
		=
		\sum_{j=3}^{n}
		\left(
		K_{\mathscr S}(x_\mu,z_j)
		+
		K_{\mathscr S}(y_\mu,z_j)
		\right)
		\leq
		\frac12\operatorname{scal}(\mathscr S).
		\end{equation}

		With
		the conventions fixed above, one has
		$
		\left\langle
		A\owedge\operatorname{id},
		\mathscr S
		\right\rangle
		=
		\left\langle
		A,\operatorname{Ric}(\mathscr S)
		\right\rangle
		$
		for every symmetric bilinear form \(A\). Taking this into account, by means of
		\eqref{eq:Ric-expansion}, we obtain
		\[
		\begin{aligned}
		\left\langle
		\operatorname{Ric}\owedge\operatorname{id},
		\mathscr S
		\right\rangle
		&=
		\sum_{\mu=1}^{N}
		\lambda_\mu
		\left(
		\operatorname{Ric}_{\mathscr S}(x_\mu,x_\mu)
		+
		\operatorname{Ric}_{\mathscr S}(y_\mu,y_\mu)
		\right)
		-
		(n-1)a\,\operatorname{scal}(\mathscr S).
		\end{aligned}
		\]
		
		Consequently, thanks to \eqref{estim_Ric} and \eqref{eq:scalar-expansion}, we reach
		\[
		\begin{aligned}
		\frac{
			\left\langle
			\operatorname{Ric}\owedge\operatorname{id},
			\mathscr S
			\right\rangle
		}{
			\operatorname{scal}(\mathscr S)
		}
		&\leq
		\frac14
		\left(
		\operatorname{scal}
		+
		n(n-1)a
		\right)
		-
		(n-1)a
		=
		\frac14
		\left(
		\operatorname{scal}
		+
		(n-4)(n-1)a
		\right).
		\end{aligned}
		\]
		
		Plugging this into \eqref{eq:phi-evolution-second}, we conclude that
		\[
		\begin{aligned}
		(\partial_t-\Delta)\varphi
		&\leq 
		\operatorname{scal}\,a
		+
		(n-4)(n-1)a^2
		+
		2(n-1)a^2,
		\end{aligned}
		\]
		which finishes the proof.
	\end{proof}
	
	Since $\ell$ is locally Lipschitz, the standard equivalence between
	viscosity and distributional subsolutions implies that the same
	inequality holds in the distributional sense.
	
	\section{An exponential integrability estimate} \label{sec:4}
	
	The following lemma converts a scale-invariant
	\(L^{1+\beta}\)-bound and a pointwise bound of the form $\ell \leq \eta/t$ into exponential
	integrability along the time variable.
	
	\begin{lemma}\label{lem:exponential-integrability}
		For every $\beta\in(0,1)$ and $\delta,A>0$,
		there exists a constant
		$
		\eta=\eta(\beta,\delta,A)\in(0,1]
		$
		such that the following holds.
		
		Let \((\Omega,\mathcal F,\mu)\) be a probability space, let \(r>0\),
		and consider  a measurable function
		$
		f\colon\Omega\times(0,r^2]\rightarrow[0,\infty)
		$
		that satisfies
		\begin{equation}\label{eq:exp-pointwise-bound}
		f(\omega,t)\leq\frac{\eta}{t} \qquad \text{and} \qquad \int_0^{r^2}\int_\Omega
		f^{1+\beta}(\omega,t)\,d\mu(\omega)\,dt
		\leq
		Ar^{-2\beta}
		\end{equation}
		for almost every
		\((\omega,t)\in\Omega\times(0,r^2]\).
		Then
		\begin{equation}\label{eq:exp-integrability-conclusion}
		\int_\Omega
		\exp\left(
		A\int_0^{r^2}f(\omega,t)\,dt
		\right)
		\,d\mu(\omega)
		\leq
		1+\delta.
		\end{equation}
	\end{lemma}
	
	\begin{proof}
		We first reduce the proof to the case \(r=1\). Define
		$
		\widetilde f(\omega,s)
		:=
		r^2f(\omega,r^2s) \leq
		\frac{\eta}{s},
		\,
		s\in(0,1].
		$
		Then
		$
		\int_0^1\widetilde f(\omega,s)\,ds
		=
		\int_0^{r^2}f(\omega,t)\,dt
		$ and, 
		after the change of variable \(t=r^2s\), we also get
		\[
		\begin{aligned}
		\int_0^1\int_\Omega
		\widetilde f^{1+\beta}(\omega,s)
		\,d\mu(\omega)\,ds
		&= r^{2\beta}
		\int_0^{r^2}\int_\Omega
		f^{1+\beta}(\omega,t)
		\,d\mu(\omega)\,dt
		\leq A.
		\end{aligned}
		\]
		We henceforth drop the tilde and assume $r = 1$.

		Let \(\theta\in(0,1)\), to be chosen below, and set
		$
		b:=\frac{\theta}{1+\beta}.
		$
		Using the pointwise bound on \(f\), we have
		$
		f
		\leq
		f^\theta
		\left(\frac{\eta}{t}\right)^{1-\theta}.
		$
		Thus Hölder's inequality gives
		\[
		\int_0^1f(\omega,t)\,dt
		\leq
		\eta^{1-\theta}
		\int_0^1
		t^{\theta-1}f^\theta(\omega,t)\,dt \leq
		\eta^{1-\theta}
		\left(
		\int_0^1
		t^{-\frac{1-\theta}{1-b}}\,dt
		\right)^{1-b}
		\left(
		\int_0^1
		f^{1+\beta}(\omega,t)\,dt
		\right)^b.
		\]
		Since 
		\(\theta>b\), the first integral is finite and satisfies
		\[
		\begin{aligned}
		\int_0^1
		t^{-\frac{1-\theta}{1-b}}\,dt
		&=
		\left(
		1-\frac{1-\theta}{1-b}
		\right)^{-1}
		=
		\frac{1-b}{\theta-b}
		=
		\frac{1+\beta-\theta}{\beta\theta}
		\leq
		\frac{2}{\beta\theta}.
		\end{aligned}
		\]
		Plugging this back into the previous inequality yields
		\begin{equation}\label{eq:basic-interpolation}
		\int_0^1f(\omega,t)\,dt
		\leq
		\eta^{1-\theta}
		\left(
		\frac{2}{\beta\theta}
		\right)^{1-b}
		\mc	I^b, \qquad \text{with} \quad \mc I:=\int_0^{1}  f^{1+\beta}(\omega,t) dt 
		\end{equation}
		
		We now distinguish two cases.
		
		\medskip

		\emph{Case 1: \(\mc I \leq100 \, \eta^{1/4}\).} By Hölder's inequality,
		\[
		\int_0^1f(\omega,t)\,dt
		\leq
		\mc I^{\frac1{1+\beta}}  \leq 100 \, \eta^{\frac1{4(1+\beta)}}.
		\]
		This, using that 	
		\(\beta\in(0,1)\) and \(0<\eta\leq1\), leads to
		\begin{equation}\label{eq:small-I}
		\exp\left(
		A\int_0^1f(\omega,t)\,dt
		\right)
		\leq
		\exp\left(100A\, \eta^{1/8}\right).
		\end{equation}
		
		\medskip

		\emph{Case 2: \(\mc I >100 \, \eta^{1/4}\).} Set
		$
		\mc J:=\eta^{-1/4}\mc I > 100
		$ and choose
		\begin{equation}\label{eq:choice-theta}
		0 <	\theta
		:=
		\frac{1+\beta}{\log \mc J} <
		\frac12 \qquad \text{so that} \qquad \frac{2}{\beta\theta}>1
		\qquad \text{and} \qquad b=\frac{1}{\log\mc J}.
		\end{equation}
		From here, using \eqref{eq:basic-interpolation}, together with
		\(1-b\leq1\) and $	\mc J^b=e$, we obtain
		\[
		\begin{aligned}
		\int_0^1f(\omega,t)\,dt
		&\leq
		\eta^{1-\theta}
		\frac{2}{\beta\theta}
		\, \mc I^b
		=
		\eta^{1-\theta}
		\frac{2\log \mc J}{\beta(1+\beta)}
		\eta^{b/4} e
		\leq 
		\frac{2e}{\beta}
		\eta^{1/2}
		\log \mc J,
		\end{aligned}
		\]
		where we applied that $\eta \leq 1$, \(1+\beta\geq1\) and $\theta < \frac12$. Now set
		$
		q
		:=
		\frac{\beta}{2eA\eta^{1/2}}.
		$
		After decreasing \(\eta\) if necessary, we may assume that
		$
		q\geq 1.
		$
		Taking this into account, Young's inequality gives
		\[
		\exp\left(
		A\int_0^1 f(\omega,t)\,dt
		\right)
		\leq
		\mathcal J^{1/q} \leq	1+\frac{\mathcal J}{q}
		=
		1+
		\frac{2eA}{\beta}
		\eta^{1/4} \mc I.
		\]
		Combining  the latter and \eqref{eq:small-I}, we obtain, in
		both cases,
		\[
		\exp\left(
		A\int_0^1f(\omega,t)\,dt
		\right)
		\leq
		\exp\left(100A\eta^{1/8}\right)
		+
		\frac{2eA}{\beta}
		\eta^{1/4}\mc I.
		\]
		Integrating over \(\Omega\) and using the integral assumption from the statement, we get
		\begin{equation}\label{eq:final-exp-estimate}
		\int_\Omega
		\exp\left(
		A\int_0^1f(\omega,t)\,dt
		\right)
		\,d\mu(\omega)
		\leq
		\exp\left(100A\eta^{1/8}\right)
		+
		\frac{2eA^2}{\beta}\eta^{1/4}.
		\end{equation}
		Choosing
		$
		\eta
		:=
		\min\left\{
		1,\,
		\left(
		\frac{\log(1+\delta/2)}{100A}
		\right)^8,\,
		\left(
		\frac{\delta\beta}{4eA^2}
		\right)^4,\,
		\left(
		\frac{\beta}{2eA}
		\right)^2
		\right\}
		$
		makes the right-hand side at most \(1+\delta\), which proves
		the result.
	\end{proof}

	\section{Ricci flows with integral almost nonnegative curvature}
	\label{sec:integral-almost-nonnegative-curvature}
	
	In this section, we study Ricci flows starting from complete Riemannian manifolds and orbifolds with small integral cosectional curvature defect. More precisely, we assume uniform two-sided volume bounds, a scale-invariant Morrey bound for the initial defect, and a generalized segment inequality. We shall show that, under suitable smallness assumptions, the corresponding Ricci flows satisfy uniform curvature, noncollapsing, defect, and distance-distortion estimates. For simplicity, the estimates are first proved for manifolds; their extension to orbifolds is discussed in the proof of Theorem~\ref{Lp-RF-existence}.

	\subsection{Assumptions on the initial metric}
	\label{subsec:initial-assumptions}
	
	Let
	$
	v_0,\beta,\varepsilon,\alpha,D,\Upsilon>0
	$
	be constants whose values will be determined in the course of the
	argument. We assume that
	\[
	v_0,\varepsilon,\alpha\in(0,1],
	\qquad
	0<\beta<\frac12,
	\qquad
	D,\Upsilon\geq1.
	\]
	
	Let $(M^n,g)$
	be a complete Riemannian manifold with bounded curvature.  We denote by
	$
	d,\, d\mu,\, B(x,r)
	$
	the distance, the Riemannian measure, and the open metric ball
	determined by \(g\), respectively. We write
	\[
	\ell(x)
	:=
	\inf
	\left\{
	a\geq0:
	\operatorname{Rm}_g(x)+a \, \mathscr I_g(x)
	\in
	\mathcal C_{\mathrm{geom}\geq0}
	\right\}
	\]
	for the cosectional curvature defect of \(g\). Here \(\mathscr I_g(x)\) denotes the curvature operator of constant sectional curvature \(1\) on \((T_xM,g_x)\), that is,  \(\mathscr I_g(x)\) is the identity on \(\Lambda^2T_xM\).
	
	We first give the proof for geometrically nonnegative curvature.
	For the curvature-operator case, the same argument applies with
	$\ell$ equal to the negative part of the lowest eigenvalue of $\Rm$;
	see the proof of Theorem~\ref{Lp-RF-existence}.
	
	We shall consider initial metrics satisfying some or all of the
	following properties from Theorem \ref{Lp-RF-existence_short}: \textup{\textbf{(A1)}}, \textup{\textbf{(A2)}} with $\ell$ as above, and a generalized $(\alpha, D, \Upsilon)$-segment inequality, which will be denoted by \textup{\textbf{(A3)}} hereafter. For the latter, \eqref{eq:A3-pushforward} implies that, for every nonnegative Borel function
	\(f\colon M\to[0,\infty]\),
	\begin{equation}\label{eq:A3-integral}
	\int_{\Omega_{x,y}}
	\int_0^1
	f\bigl(\gamma^\omega_{x,y}(s)\bigr)
	\,ds\,d\nu_{x,y}(\omega)
	\leq
	\Upsilon r^{-n}
	\int_M f\,d\mu.
	\end{equation}
	
	\begin{rmk}\label{rem:initial-assumptions}
		(a)	Arguing as in \cite[Lemma 2.1]{BaZ}, a covering argument based on the uniform two-sided volume
		bounds \textup{\textbf{(A1)}} gives constants
		$
		C=C(n,v_0)<\infty
		$
		such that
		\begin{equation}\label{eq:global-volume-growth}
		\operatorname{vol}_g B(x,R)
		\leq
		C(1+R)^n e^{CR}
		\end{equation}
		for every $x\in M$ and $R>0$, where $C=C(n,v_0)$. In particular,
		every metric ball has finite volume, uniformly in terms of $n$ and
		$v_0$. This growth bound will be used for integrals with Gaussian
		weights.

		(b)	The generalized segment inequality should be regarded as a weak substitute for a lower Ricci curvature bound. The classical Cheeger--Colding segment inequality shows that a lower Ricci curvature bound provides quantitative control of averaged line integrals along minimizing geodesics in terms of ambient volume integrals. Here we retain only the feature needed in the subsequent argument: the existence of a sufficiently rich family of curves of controlled length whose averaged occupation measure is bounded by the Riemannian measure. Unlike in the classical setting, the curves need not be minimizing geodesics and their endpoints may vary within \(B(x,\alpha r)\) and \(B(y,\alpha r)\). This weaker formulation is adapted to the approximations constructed later, where a uniform lower Ricci curvature bound is not assumed.
	\end{rmk}

	\subsection{A priori assumptions}
	\label{subsec:apriori-assumptions}
	
	Let $\tau_0\in(0,1]$ be a fixed upper time bound, to be specified
	below, and let
	$
	(M,g(t))_{t\in[0,\tau]}
	$
	be a complete Ricci flow with $g(0)=g$ and uniformly bounded curvature on compact time intervals, where $0<\tau<\tau_0$. Integrations by parts on $M$ are justified by exhaustion cutoffs.
	On time intervals bounded away from zero, the curvature,
	volume-growth and Gaussian estimates make the cutoff errors vanish.

	We denote by
	$
	d_t,\, d\mu_t,\, B_t(x,r)
	$
	the distance, the Riemannian measure, and the open metric balls
	associated with \(g(t)\). In particular,
	$
	d_0=d,
	\,
	d\mu_0=d\mu,
	\,
	B_0(x,r)=B(x,r).
	$ We will also write $|\cdot|_t$ for $\vol_{g(t)}(\cdot)$. For \(x\in M\) and \(t\in[0,\tau]\), we set
	\[
	\ell(x,t)
	:=
	\inf
	\left\{
	a\geq0:
	\operatorname{Rm}_{g(t)}(x)+a \mathscr I_{g(t)}(x)
	\in
	\mathcal C_{\mathrm{geom}\geq0}
	\right\}.
	\]
	Thus,
	$
	\ell(x,0)=\ell(x),
	$
	where \(\ell(x)\) is the initial cosectional curvature defect
	introduced in Subsection~\ref{subsec:initial-assumptions}. In particular, from here we have
	\begin{equation}
	\operatorname{Ric}_{g(t)}
	\geq
	-(n-1)\ell(\cdot,t)\,g(t).
	\label{eq:Ricci-lower-from-defect}
	\end{equation}

	We now introduce the a priori bounds. For this purpose, fix 
	$
	\mathscr K,v_1, \mathscr A, \mathscr B, \mathscr C>0,
	$ and
	$0<\eta\leq1$.
	These constants and $\tau_0$ will be determined in the course of the
	argument. We shall consider Ricci flows satisfying some or all of
	the following a priori assumptions.
	
	\medskip

	\textnormal{\textbf{(B1) Upper curvature bound.}}
	For every \(x\in M\) and every \(0<t\leq\tau\),
	\begin{equation}\label{eq:B1}
	\left|\operatorname{Rm}_{g(t)}(x)\right|
	\leq
	\frac{\mathscr K}{t}.
	\end{equation}
	
	\medskip

	\textnormal{\textbf{(B2) Pointwise defect bound.}}
	For every \(x\in M\) and every \(0<t\leq\tau\),
	\begin{equation}\label{eq:B2}
	\ell(x,t)
	\leq
	\frac{\eta}{t}.
	\end{equation}
	
	\medskip

	\textnormal{\textbf{(B3) Parabolic noncollapsing.}}
	For every \(x\in M\),  \(0<t\leq\tau\), and 
	\(0<r\leq\sqrt t\),
	\begin{equation}\label{eq:B3}
	|B_t(x,r)|_t
	\geq
	v_1r^n.
	\end{equation}

	\medskip

	\textnormal{\textbf{(B4) Upper distance bound.}}
	For all \(x,y\in M\) and all \(0<r\leq1\) with
	$
	d(x,y)\leq r,
	$
	\begin{equation}\label{eq:B4}
	d_t(x,y)
	\leq
	2Dr \qquad \text{ for every } \qquad 
	0\leq t\leq\min\{r^2,\tau\}.
	\end{equation}

	\medskip

	\textnormal{\textbf{(B5) Spacetime Morrey-type defect bound.}}
	For every \(x\in M\) and every \(0<r\leq1\),
	\begin{equation}\label{eq:B5}
	\int_0^{\min\{r^2,\tau\}}
	\int_{B(x,r)}
	\ell^{1+\beta}(\cdot,t)\,d\mu\,dt
	\leq
	\mathscr Ar^{n-2\beta}.
	\end{equation}
	
	For \(x\in M\) and \(r>0\), set
	\[
	\widehat{\ell}_r^{1+\beta}(x)
	:=
	r^{-n} \int_M
	\exp\left(
	-\frac{d^2(x,\cdot)}{\mathscr C r^2}
	\right)
	\ell^{1+\beta}(\cdot,0)\,d\mu.
	\]
	
	\medskip

	\textnormal{\textbf{(B6) Gaussian bound for the curvature defect.}}
	For every \(x\in M\) and \(0<t\leq\tau\),
	\[
	\ell(x,t)
	\leq
	\mathscr B\,\widehat{\ell}_{\sqrt t}(x).
	\]
	
	\medskip 
	
	By \textup{\textbf{(B1)}}, the curvature remains bounded as $t\nearrow\tau$,
	so the flow extends to $[0,\tau']$ for some $\tau<\tau'<\tau_0$.
	In Propositions~\ref{prop:upper-curvature-improvement}--\ref{improved_Gaussian},
	we will improve the a priori assumptions \textup{\textbf{(B1)--(B6)}} on
	$[0,\tau]$ and show that they remain valid on this extension,
	after decreasing $\tau'$ if necessary, provided the involved
	constants satisfy suitable bounds. The admissible order of choices
	is shown in the following diagram; in particular, $\Upsilon$ can
	be chosen after $\alpha$.
	\begin{equation} \label{order-constants}
	\boxed{
		\begin{tikzpicture}[baseline=(current bounding box.center),
		>=stealth, every node/.style={inner sep=4pt}]
		\node (data) at (0,0) {$(n,v_0,D)$};
		\node (v) at (1.9,0) {$v_1$};
		\node (K) at (3.4,0) {$\mathscr K$};
		\node (CB) at (5.3,0) {$(\mathscr C,\mathscr B)$};
		\node (A) at (7.2,0) {$\mathscr A$};
		\node (eta) at (8.8,0) {$\eta$};
		\node (eps) at (10.3,0) {$\varepsilon$};
		\node (alpha) at (0,-1) {$\alpha$};
		\node (Upsilon) at (1.9,-1) {$\Upsilon$};
		\node (beta) at (8.8,1) {$\beta$};
		\node (tau) at (3.4,1) {$\tau_0$};
		\draw[->] (data) -- (v);
		\draw[->] (v) -- (K);
		\draw[->] (K) -- (CB);
		\draw[->] (CB) -- (A);
		\draw[->] (A) -- (eta);
		\draw[->] (eta) -- (eps);
		\draw[->] (data) -- (alpha);
		\draw[->] (alpha) -- (Upsilon);
		\draw[->] (Upsilon.east) -| (eta.south);
		\draw[->] (beta) -- (eta);
		\draw[->] (K) -- (tau);
		\end{tikzpicture}
	}
	\end{equation}
	More precisely, we require
	$\alpha=\alpha(D)\in(0,1/32]$ and
	$v_1=\frac12\overline v_1(n,v_0,D)$, together with
	$\mathscr K \geq \underline{ \mathscr K}(n,v_1) \geq v_1^{-1}$,
	$\mathscr C \geq \underline{ \mathscr C}(n,v_0,\mathscr K)$, $\mathscr B \geq \underline{ \mathscr B}(n,v_0,\mathscr K)$ and $\mathscr A \geq \underline{ \mathscr A}(n,v_0,\mathscr B, \mathscr C)$. Dependencies on preceding parameters in
	this hierarchy are understood implicitly.
	The remaining bounds are
	$
	\eta\leq
	\ov \eta(n,\beta,v_0,D,\Upsilon, \mathscr A,\mathscr K)
	$,
	$
	\varepsilon\leq
	\varepsilon_0(n,\beta,v_0,\mathscr K,\mathscr B, \mathscr C,\eta)
	$
	and
	\[
	0<\tau_0\leq \overline\tau_0(n,v_0,D,\mathscr K)\leq1.
	\]
	Each proposition uses only the initial assumptions
	\textup{\textbf{(A1)--(A3)}} and the original bounds \textup{\textbf{(B1)--(B6)}}.
	The bounds on the constants are established in the proofs below;
	since each depends only on parameters earlier in
	\eqref{order-constants}, all requirements can subsequently be
	satisfied simultaneously.

	\subsection{Extension of the curvature estimate (B1)}
	We first improve the upper curvature bound \textup{\textbf{(B1)}}. The
	proof is a blow-up argument using \textup{\textbf{(B2)}} and \textup{\textbf{(B3)}}.
	Failure of the bound would produce a nonflat $\kappa$-solution
	with positive asymptotic volume ratio, contradicting Perelman's theorem.
	
	\begin{prop}
		\label{prop:upper-curvature-improvement}
		If 
		$
		{\mathscr K} \geq  \underline{\mathscr K}(n,v_1)
		$ and $0<\tau<\tau_0\leq1$, then $g(t)$ extends to a complete
		Ricci flow with bounded curvature on $[0,\tau']$ for some
		$\tau<\tau'<\tau_0$, and \textup{\textbf{(B1)}} holds on this larger interval.
	\end{prop}
	
	\begin{proof}
		We follow the point-picking and blow-up argument used in the proof of
		\cite[Theorem~1]{BamlerCabezasRivasWilking} to establish the strict version of \textup{\textbf{(B1)}}. We first claim that there exists a constant
		$\underline{\mathscr K} = \underline{\mathscr K}(n,v_1)<\infty$
		such that every flow satisfying \textup{\textbf{(B2)}} and \textup{\textbf{(B3)}} on
		$[0,\tau]$ satisfies
		\begin{equation}
		\sup_{0<t\le\tau}\sup_{x\in M}
		t|\Rm_{g(t)}(x)|
		\le  \frac{\underline{\mathscr K}}{2}.
		\label{eq:universal-curvature-bound}
		\end{equation}
		
		Fix $v_1$ and suppose by contradiction that there exists a sequence
		of complete Ricci flows
		$(M_i,g_i(t))_{t\in[0,\tau_i)}$, with curvature uniformly bounded
		on compact time intervals, satisfying \textup{\textbf{(B2)}} and
		\textup{\textbf{(B3)}} with the same constants, and points
		$(x_i,s_i)\in M_i\times(0,\tau_i)$ such that
		$
		s_i|\Rm_{g_i(s_i)}(x_i)|\longrightarrow\infty.
		$
		Choose $s_i<\tau'_i<\tau_i$ and set
		\[
		Q_i
		:=
		\sup_{0<t\leq\tau'_i}\sup_{x\in M_i}
		t|\Rm_{g_i(t)}(x)|.
		\]
		Then $Q_i<\infty$, since the curvature is bounded on compact time
		intervals, and
		\[
		Q_i
		\geq
		s_i|\Rm_{g_i(s_i)}(x_i)|
		\longrightarrow\infty.
		\]
		Take an approximate maximizer
		$(p_i,t_i)\in M_i\times(0,\tau'_i]$ such that
		$
		t_i|\Rm_{g_i(t_i)}(p_i)|
		\geq
		\frac12Q_i.
		$ Consider the rescaled Ricci flows
		\begin{equation}\label{eq:rescaled-flows}
		\widetilde g_i(s)
		:=
		\frac{Q_i}{t_i}\,
		g_i\left(
		t_i+\frac{t_i}{Q_i}s
		\right),
		\qquad
		-\frac{Q_i}{2}\leq s\leq0,
		\end{equation}
		which satisfy $|
		\operatorname{Rm}_{\widetilde g_i(0)}(p_i)
		|
		\geq
		\frac{1}{2}$ and $| \Rm_{\tilde g_i (s)} | \leq 2$ for every
		$
		s\in[-\frac{Q_i}{2},0].
		$

		Let \(\widetilde\ell_i\) denote the cosectional curvature defect of
		\(\widetilde g_i\). Since it scales in the same way as the
		curvature tensor,  \textup{\textbf{(B2)}} and $\eta\leq1$ imply
		$
		\widetilde\ell_i(x,s) \leq \frac{2}{Q_i}$, and hence
		$
		\operatorname{Rm}_{\widetilde g_i(s)}
		\geq
		-\frac{2}{Q_i}\mathscr I_{\widetilde g_i(s)}
		$ for $s \in (-\frac{Q_i}{2}, 0]$.
		
		Since the parabolic noncollapsing estimate \textup{\textbf{(B3)}} is preserved
		under the rescaling, as $Q_i \to \infty$, Hamilton's compactness theorem gives a subsequence pointed at $(p_i,0)$, which converges to a non-flat ancient solution $g_\infty(s)$ with $\Rm_{g_\infty(s)} \geq 0$ for all times and bounded curvature. Passing \textup{\textbf{(B3)}} to the limit and using volume  comparison, we conclude that the asymptotic volume ratio of $\tilde g_\infty(t)$ is positive.
		This contradicts the vanishing of the volume ratio on a $\kappa$-solution (cf.~\cite[11.4]{P1}).
		
		Accordingly, \eqref{eq:universal-curvature-bound} follows. 
		If $\mathscr K\ge\underline{\mathscr K}$, then
		$
		t |\Rm_{g(t)}|
		\le
		\frac{\mathscr K}{2} 
		\,
		\text{on }M\times(0,\tau].
		$
		The bounded-curvature continuation theorem extends $g(t)$ to
		$[0,\tau']$ for some $\tau<\tau'<\tau_0$. The strict estimate
		above ensures that \textup{\textbf{(B1)}} remains valid on this extension
		after decreasing $\tau'$ if necessary.
	\end{proof}

	\subsection{Extension of the lower curvature control (B2)}
	We next improve the pointwise lower curvature bound \textup{\textbf{(B2)}}.
	The proof estimates the Gaussian averages $\widehat\ell_{\sqrt t}$
	using the initial volume and Morrey bounds \textup{\textbf{(A1)}} and
	\textup{\textbf{(A2)}} and then applies \textup{\textbf{(B6)}}.
	
	\begin{prop}
		\label{prop:lower-curvature-improvement}
		We can find a constant $\mathscr D = \mathscr D(n, v_0, \mathscr C) < \infty$ such that
		\begin{equation} \label{eq_lem_B2_l_hat_bound}
		\widehat\ell_{\sqrt t}(x) \leq \frac{\mathscr D\sqrt{\eps}}{t}.
		\end{equation}
		If, in addition, 
		$ \eps \leq \ov\eps(n, v_{0}, \mathscr B, \mathscr C, \eta)$,
		then the Ricci flow admits an extension beyond $\tau$ on which
		\textup{\textbf{(B2)}} remains valid.
	\end{prop}
	
	\begin{proof}	
		Set $\rho:= \sqrt{t}$. If we cover $M\setminus\{x\}$ by the annuli $A_k^\rho = B(x, (k+1)\rho) \setminus B(x, k \rho)$ with $k=0,1,\ldots$, we get from the definition of $\widehat \ell_\rho$ that
		\[
		\begin{aligned}
		\widehat \ell_\rho^{1+\beta}(x)
		= \rho^{-n} \sum_{k=0}^{\infty} \int_{A_k^\rho}  \exp\bigg({-\frac{d^2(x, \cdot)}{\mathscr C t}}\bigg) \ell^{1 + \beta}(\cdot, 0)\, d\mu.
		\end{aligned}
		\]

		For each $k$, choose a maximal $\frac{\rho}{2}$-separated set
		$
		\{y_1,\ldots,y_{N_k}\}\subset A_k^\rho.
		$
		Then the balls $B(y_i,\frac{\rho}{2})$ cover $A_k^\rho$, while the balls
		$B(y_i,\frac{\rho}{4})$ are pairwise disjoint. Since
		$y_i\in B(x,(k+1)\rho)$, these disjoint balls are contained in
		$B(x,(k+2)\rho)$. Hence, by \textup{\textbf{(A1)}} and the volume-growth estimate
		in Remark~\ref{rem:initial-assumptions}(a),
		\[
		\begin{aligned}
		N_k v_0\left(\frac{\rho}{4}\right)^n
		&\le
		\sum_{i=1}^{N_k}
		\bigl|B(y_i,\frac{\rho}{4})\bigr|
		\le
		\bigl|B(x,(k+2)\rho)\bigr|
		\le
		C(k+2)^n\rho^n
		e^{C (k + 2) \rho}.
		\end{aligned}
		\]

		As $\cup_{i = 1}^{N_k} B(y_i, \frac{\rho}{2})$ covers the annulus $A^\rho_k$ and $d(x, \cdot)|_{A_k^\rho} \geq k \rho$ we deduce
		\[\widehat \ell_\rho^{1+\beta}(x) \leq \rho^{-n} \sum_{k=0}^\infty e^{-\frac{k^2}{\mathscr C}} \sum_{i = 1}^{N_k}\int_{B(y_i, \frac{\rho}{2})}   \ell^{1 + \beta}(\cdot, 0)\, d\mu \leq \frac{C \varepsilon}{t^{1 + \beta}} \sum_{k=0}^\infty N_k \, e^{-\frac{k^2}{\mathscr C}}, \]
		where we have applied assumption \textup{\textbf{(A2)}}. From the upper bound for $N_k$ computed above we get
		\[\widehat \ell_\rho^{1+\beta}(x) \leq \frac{\eps}{t^{1 + \beta}} \, \widetilde  C(v_0) \sum_{k = 0}^\infty (k+2)^n \exp\bigg(C (k+2) - \frac{k^2}{\mathscr C}\bigg) \leq C(n, v_0, \mathscr C)\frac{\eps}{t^{1 + \beta}}.\]
		As
		$
		0<\beta<1
		\,\text{and}\,
		0<\varepsilon\leq1
		$, this finally gives $\widehat \ell_{\sqrt t}(x) \leq \mathscr D  \frac{\eps^{\frac{1}{1+\beta}}}{t} \leq \mathscr D  \frac{\sqrt{\eps}}{t}$
		for $\mathscr D= \mathscr D(n, v_{0}, \mathscr C)$. The conclusion then follows from \eqref{eq_lem_B2_l_hat_bound} using a priori assumption \textup{\textbf{(B6)}} for $\eps \leq  \left(\frac{\eta}{2\mathscr B \mathscr D}\right)^2 =:\ov\eps(n, v_0, \mathscr B, \mathscr C, \eta)$.
	\end{proof}

	\subsection{Extension of the parabolic noncollapsing (B3)} We next improve the parabolic noncollapsing estimate \textup{\textbf{(B3)}}. The
	proof is a packing and compactness argument. The initial volume bounds
	provide many separated points, the distance estimates preserve their
	separation for a short time, and Alexandrov compactness and volume
	convergence give a lower volume bound at the parabolic scale.

	\begin{prop}
		\label{prop:parabolic-noncollapsing-improvement}
		If $\eta \leq \ov\eta (n, v_0, D, \mathscr K)$ and $v_1 \leq \ov v_1 (n, v_0, D)$ and $\tau_0 \leq \ov\tau(n, v_0, D, \mathscr K)$, then the Ricci flow admits an extension beyond $\tau$ on which \textup{\textbf{(B3)}} remains valid.
	\end{prop}
	
	\begin{proof}
		Let $\delta > 0$ and $\Lambda > 1$ be constants whose values we will choose later and assume that $\tau_0 \leq \Lambda^{-2}$.
		Fix $(x,t) \in M \times (0, \tau]$ and set $\rho := \Lambda \sqrt t \leq 1$.
		If $\eta \leq \ov\eta (n, \Lambda)$, then \textup{\textbf{(B2)}} combined with volume comparison gives $|B_t(x,r)|_t r^{-n} \geq \frac12 |B_t(x, \rho)|_t \rho^{-n}$.
		So it is enough to show that 
		\[ |B_t(x, \rho)|_t > 4 v_1 \rho^n. \]

		At $t = 0$, a packing estimate using \textup{\textbf{(A1)}} reveals that there are $N \geq c(v_0) \delta^{-n}$ many points $\{ y_1, \ldots, y_N \} \subset B(x, \rho)$ with $d(y_i, y_j) \geq \delta \rho$ for $i \neq j$.
		By a priori assumption \textup{\textbf{(B1)}} and a distance distortion argument (see \cite[Theorem 17.2]{HamFS} and Editor's note 24 in \cite{Coll}) this yields
		\begin{equation} \label{net_timet}
		d_t (y_i, y_j) \geq d(y_i, y_j) - C\sqrt{\mathscr K t} 
		\geq \delta \rho  - C\sqrt{\mathscr K} \,\frac{\rho}{\Lambda} \geq \frac12 \delta \rho,
		\end{equation}
		as long as $\Lambda \geq 2C \delta^{-1} \sqrt{\mathscr K}$.
		By \textup{\textbf{(B4)}}, we moreover have $\{ y_1, \ldots, y_N \} \subset B_t (x, 2D \rho)$.
		
		\medskip

		Consider the rescaled metric $h := \rho^{-2} g(t)$.
		Then $\{ y_1, \ldots, y_N \}$ forms a $\frac12\delta$-separated subset of $B_h(x, 2D)$ within $(M,h)$, and hence the $\frac12 \delta$-packing number of $B_h (x, 2D)$ is at least $c(v_0) \delta^{-n}$.
		Moreover, if $\eta \leq \overline{\eta} (\Lambda)$, then \textup{\textbf{(B2)}} implies that the sectional curvature of $h$ is bigger or equal than $-1$.
		We now claim that if $v_1 \leq \overline v_1(n, v_0,D)$ and  $\delta \leq \ov\delta (n,v_0,D)$, then this implies $|B_{h}(x,1)| > 4 v_1$.
		If not, then we can find a sequence of pointed Riemannian manifolds $(M_i,h_i,x_i)$ with the same packing estimate (for a $\delta$, which we will choose soon), such that 
		\begin{equation} \label{ball-collapse}
		|B_{h_i}(x_i,1)| \to 0.
		\end{equation}

		By Gromov's compactness theorem, after passing to a subsequence, we may assume that $(M_i,h_i,x_i) \to (X,d_X,x_\infty)$ in the pointed Gromov--Hausdorff topology, where \(X\) is an Alexandrov space with curvature bounded below by \(-1\); see \cite{BGP}. 
		Taking the packing estimate to the limit implies that the $\frac14\delta$ packing number of $B_X(x_\infty, 4D)$ is at least $c(v_0) \delta^{-n}$.
		For any $\delta \leq \ov\delta(n,v_0,D)$ sufficiently small,  this implies by \cite[Theorem 15.13 and Corollary 15.16]{AlexanderKapovitchPetrunin} that $(X,d_X)$ has dimension $n$.
		Then the noncollapsed volume convergence theorem (cf. \cite[Theorem~5.9]{ChCol}) therefore contradicts \eqref{ball-collapse}.
		
		In summary, we have shown the proposition if $\delta \leq \ov\delta(n, v_0, D)$, $v_1 \leq \ov v_1 (n, v_0, D)$, $\Lambda \geq \underline{\Lambda}(\delta, \mathscr K)$ and $\eta \leq \ov\eta(\Lambda)$ and $\tau_0 \leq \Lambda^{-2}$.
		This finishes the proof.
	\end{proof}

	\subsection{Extension of the upper distance bound (B4)}
	We next improve the upper distance bound \textup{\textbf{(B4)}}. The
	exponential integrability estimate of Lemma~\ref{lem:exponential-integrability}
	and the generalized segment inequality provide a curve with controlled length growth. Short
	endpoint corrections and \textup{\textbf{(B2)}} then yield the desired
	distance bound.

	\begin{prop}
		\label{prop:upper-distance-improvement}
		If $\alpha \leq 1/32$ and $\eta \leq \overline\eta (n, \beta, v_0, D, \Upsilon, \mathscr A)$, then the Ricci flow admits an extension beyond $\tau$ on which \textup{\textbf{(B4)}} remains valid.
	\end{prop}

	\begin{proof}
		Fix $x,y\in M$ and set
		$
		s:=d(x,y).
		$ 
		There is nothing to prove if $s=0$, so assume that $0<s\le1$. We
		first consider
		$
		0<t\le\min\{s^2,\tau\}.
		$
		Let $(\Omega_{x,y},\mathcal F_{x,y},\nu_{x,y})$ and
		$
		\Phi_{x,y}\colon\Omega_{x,y}\times[0,1]\rightarrow M
		$ be supplied by Definition~\ref{segment} at the scale
		$s$, and write
		$\gamma_\omega:=\Phi_{x,y}(\omega,\cdot)$.

		\smallskip 
		Set
		$
		t_0:=\alpha^2t \leq(\alpha s)^2 \leq s^2.
		$
		We first control the increase in length of the curves
		\(\gamma_\omega\) between times \(0\) and \(t_0\).
		A basic distance distortion estimate gives
		\begin{equation}\label{eq:length-ratio-controlled-curve}
		\frac{L_{t_0}(\gamma_\omega)}
		{L(\gamma_\omega)}
		\leq
		\int_0^1
		\exp
		\left(
		(n-1)\int_0^{t_0}
		\ell(\gamma_\omega(q),\sigma)\,d\sigma
		\right)dq.
		\end{equation}
		In fact, using \eqref{eq:Ricci-lower-from-defect}, we obtain
		$
		\frac{d}{dt}
		\log|\dot\gamma|_{g(t)}
		\leq
		(n-1)\ell(\gamma(\cdot),t).
		$
		If \(\gamma\) has constant \(g(0)\)-speed, integration from \(0\) to \(t_0\), followed by integration along
		\(\gamma\), proves the above bound.

		Notice that every point of \(\gamma_\omega\) has \(g(0)\)-distance at
		most
		$
		L(\gamma_\omega)\leq D s
		$
		from \(\gamma_\omega(0)\), while
		$
		d\bigl(x,\gamma_\omega(0)\bigr)<\alpha s.
		$
		Thus every curve \(\gamma_\omega\) is contained in
		\begin{equation}\label{eq:controlled-curves-ball}
		B(x,(D+\alpha)s)
		\subseteq
		B(x,(D+1)s) =: U.
		\end{equation}

		In order to apply Lemma \ref{lem:exponential-integrability} on \(U\), let
		$
		d\overline\mu
		:=
		\frac{1}{|U|}\,d\mu|_U
		$
		so that \((U,\overline\mu)\) is a probability space. By the covering consequence of \textup{\textbf{(A1)}} recorded in
		Remark~\ref{rem:initial-assumptions}\textup{(a)}, there exists
		$
		N=N(n,v_0,D)<\infty
		$
		such that \(U\) can be covered by at most \(N\) balls of radius
		\(s\). Applying \textup{\textbf{(B5)}} to these balls and using
		\(t_0\leq s^2\), we obtain
		\begin{equation}\label{eq:normalized-enlarged-ball-integral}
		\int_0^{t_0}
		\int_U
		\ell^{1+\beta}(\cdot,\sigma)\,d\overline{\mu} \,d\sigma
		\leq
		\frac{N \mathscr A}{v_0} \, s^{-2\beta},
		\end{equation}
		where $|U|\geq v_0s^n$ by \textup{\textbf{(A1)}}.

		Extend \(\ell|_{U\times(0,t_0]}\) by zero to
		\(U\times(t_0,s^2]\). This does not change any of the integrals over
		\((0,t_0]\), and the pointwise  estimate
		\textup{\textbf{(B2)}} continues to imply
		$
		\ell(\cdot,\sigma)\leq\frac{\eta}{\sigma}
		$
		on \(U\times(0,s^2]\).
		Then
		\[
		\int_0^{s^2}
		\int_U
		\ell^{1+\beta}(\cdot,\sigma)\,d\overline\mu\,d\sigma
		\leq
		\mathscr A_*s^{-2\beta}, \qquad \text{with} \quad 
		\mathscr A_*
		:=
		\max\left\{
		n-1,\frac{N \mathscr A}{v_0}
		\right\}.
		\]
		Applying Lemma~\ref{lem:exponential-integrability} with the constant
		\(\mathscr A_*\geq n-1\), we conclude that, for any given \(\delta>0\), there exists
		$
		\eta_1
		=
		\eta_1(n,\beta,v_0,D,\mathscr A,\delta)>0
		$
		such that, if \(\eta\leq\eta_1\), then
		\begin{equation}\label{eq:spatial-exponential-average}
		\int_U
		\exp
		\left(
		(n-1)\int_0^{t_0}\ell(\cdot,\sigma)\,d\sigma
		\right)d\mu
		\leq
		(1+\delta) |U|.
		\end{equation}

		Using again Remark~\ref{rem:initial-assumptions}\textup{(a)}, the uniform
		volume-growth estimate gives
		$
		|U|
		\leq
		C_0s^n,
		$
		for some constant
		$
		C_0=C_0(n,v_0,D).
		$
		Hence
		\begin{equation}\label{eq:spatial-exponential-excess-final}
		\int_U
		h \, d\mu
		\leq
		C_0\delta s^n, \qquad \text{with} \quad h(z)
		:=
		\exp
		\left(
		(n-1)\int_0^{t_0}\ell(z,\sigma)\,d\sigma
		\right)
		-1.
		\end{equation}

		We now transfer this spatial estimate to the family of curves by 
		applying \textup{\textbf{(A3)}} to the
		nonnegative Borel function $h$. As all the curves are contained in \(U\), equations
		\eqref{eq:controlled-curves-ball} and
		\eqref{eq:spatial-exponential-excess-final} yield
		\[
		\begin{aligned}
		&\int_{\Omega_{x,y}}\int_0^1
		\left[
		\exp
		\left(
		(n-1)\int_0^{t_0}
		\ell(\gamma_\omega(q),\sigma)\,d\sigma
		\right)-1
		\right]
		dq\,d\nu_{x,y}(\omega)
		\leq
		\Upsilon s^{-n}
		\int_U h\,d\mu
		\leq
		C_0\Upsilon\delta.
		\end{aligned}
		\label{eq:curve-exponential-average}
		\]
		Since \(\nu_{x,y}\) is a probability measure, there exists
		\(\omega_0\in\Omega_{x,y}\) such that
		\[
		\int_0^1
		\exp
		\left(
		(n-1)\int_0^{t_0}
		\ell(\gamma_{\omega_0}(q),\sigma)\,d\sigma
		\right)dq
		\leq
		1+C_0\Upsilon\delta.
		\]
		Combining this with
		\eqref{eq:length-ratio-controlled-curve}, we obtain
		$
		L_{t_0}(\gamma_{\omega_0})
		\leq
		\bigl(1+C_0\Upsilon\delta\bigr)Ds.
		$
		
		Set
		$
		\bar x:=\gamma_{\omega_0}(0),
		\,
		\bar y:=\gamma_{\omega_0}(1).
		$
		By \textup{\textbf{(A3)}},
		$
		d(x,\bar x)<\alpha s,
		\,
		d_0(y,\bar y)<\alpha s.
		$
		Recalling that
		$
		t_0=\alpha^2t\leq(\alpha s)^2,
		$
		the bootstrap distance estimate \textup{\textbf{(B4)}}, applied at scale
		\(\alpha s\), gives
		\[
		\begin{aligned}
		d_{t_0}(x,y)
		&\leq
		d_{t_0}(x,\bar x)
		+
		L_{t_0}(\gamma_{\omega_0})
		+
		d_{t_0}(\bar y,y)
		\leq
		\left[
		4  \alpha
		+
		\bigl(1+C_0\Upsilon\delta\bigr)
		\right] D \, s \leq
		\frac{5}{4} D \, s.
		\end{aligned}
		\label{eq:distance-at-intermediate-time}
		\]
		The latter follows by choosing first $\alpha \leq 1/32$ 
		and then
		$
		\delta=\delta(n,v_0, D, \Upsilon)>0
		$
		with $
		C_0\Upsilon\delta
		\leq
		\frac{1}{8}.
		$

		We next pass from \(t_0\) to \(t\). 
		By \eqref{eq:Ricci-lower-from-defect} combined with \textup{\textbf{(B2)}}, for every fixed rectifiable curve \(\xi\),
		\[
		\frac{d}{d\sigma}\log L_\sigma(\xi)
		\le \frac{(n-1)\eta}{\sigma}.
		\]
		Integrating from $t_0$ to $t$ and taking the infimum over curves
		joining $x$ to $y$, we obtain
		\[
		d_t(x,y)
		\le
		\left(\frac{t}{t_0}\right)^{(n-1)\eta}
		d_{t_0}(x,y)
		\le
		\alpha^{-2(n-1)\eta}\frac54Ds
		\le
		\frac32Ds,
		\]
		after decreasing $\eta$ if necessary,  since
		$\alpha=\alpha(D)$ was fixed in \eqref{order-constants}.
		
		It remains to obtain the estimate at every scale allowed in
		\textup{\textbf{(B4)}}. If $t\le s^2$, the estimate just proved gives
		$
		d_t(x,y)\le\frac32Ds\le\frac32Dr.
		$
		If $s^2<t$, then $s^2<\tau$, and the same estimate at time $s^2$
		gives
		$
		d_{s^2}(x,y)\le\frac32Ds.
		$
		Integrating the length evolution inequality following from
		\eqref{eq:Ricci-lower-from-defect} and \textup{\textbf{(B2)}} from $s^2$ to $t$, we obtain
		\[
		\begin{aligned}
		d_t(x,y)
		&\le
		\left(\frac{t}{s^2}\right)^{(n-1)\eta}
		d_{s^2}(x,y)\le
		\frac32Ds
		\left(\frac{r}{s}\right)^{2(n-1)\eta}
		\le
		\frac32Dr,
		\end{aligned}
		\]
		where in the last inequality we have also required
		$2(n-1)\eta\le1$. 
		
	\end{proof}
	
	\subsection{Extension of the integral curvature control (B5)}
	We next improve the spacetime Morrey bound \textup{\textbf{(B5)}}. Using
	\textup{\textbf{(B6)}} and Fubini's theorem, we express the integral in terms
	of the initial data. The volume and Morrey bounds \textup{\textbf{(A1)--(A2)}}
	control the nearby contribution, while Gaussian decay controls the
	sum over distant annuli.
	
	\begin{prop} \label{B5_improve}	If
		$\mathscr A \geq \underline{\mathscr A}(n,v_0,\mathscr B,\mathscr C)$,
		then the Ricci flow admits an extension beyond $\tau$ on which
		\textup{\textbf{(B5)}} remains valid.
	\end{prop}
	
	\begin{proof}
		Fix \(x\in M\) and \(0<r\leq1\), and set $\theta:=\min\{r^2,\tau\}$. Let us first use a priori assumption \textup{\textbf{(B6)}} to bound
		\begin{align} \label{B5_first}
		\mc J &:= \int_0^\theta \int_{B(x,r)} \ell^{1+\beta} d\mu dt
		\leq \mathscr B^{1+\beta} \int_0^\theta \int_{B(x,r)} \widehat{\ell}^{\, 1+\beta}_{\sqrt{t}} \,  d\mu \,  dt \nonumber
		\\ & = \mathscr B^{1+\beta} \int_{M}
		\ell^{1+\beta}(z,0)
		\int_0^\theta
		t^{-n/2}
		\int_{B(x,r)}
		e^{
			-\frac{d^2(y,z)}{\mathscr Ct}}
		d\mu(y)\,dt\,d\mu(z) = \mathscr B^{1+\beta}(\mathcal J_1 + \mathcal J_2),
		\end{align}
		where we have used the definition of \(\widehat\ell_{\sqrt t}\) and 
		Fubini's theorem. Here $\mathcal J_1$ and $\mathcal J_2$ correspond to splitting the outer integral into
		$
		B(x,2r)$ and 
		$
		M\setminus B(x,2r)
		$, respectively.

		For the first region, the volume growth estimate from Remark \ref{rem:initial-assumptions} (a) applied on suitable annuli exhausting $M$ gives 
		\[
		t^{-n/2}
		\int_M
		\exp\left(
		-\frac{d^2(y,z)}{\mathscr Ct}
		\right)d\mu(y)
		\leq
		C_0,
		\]
		where
		$
		C_0=C_0(n,v_0, \mathscr C)<\infty.
		$ Hence
		\[
		\begin{aligned}
		&\mathcal J_1
		\leq
		C_0\theta
		\int_{B(x,2r)}
		\ell^{1+\beta}(\cdot,0)\,d\mu.
		\end{aligned}
		\]
		By the uniform two-sided volume bounds \textup{\textbf{(A1)}}, a standard maximal-disjoint-ball argument shows that \(B(x,2r)\) can be covered by at most \(N=N(n,v_0)\) balls of radius \(r\). Assumption
		\textup{\textbf{(A2)}} therefore gives
		\[
		\int_{B(x,2r)}
		\ell^{1+\beta}(\cdot,0)\,d\mu
		\leq
		C_0 N \varepsilon r^{n-2(1+\beta)}.
		\]
		Since \(\theta\leq r^2\), the contribution of \(B(x,2r)\) is bounded
		by $C_0 N \varepsilon r^{n-2\beta}$.
		
		Now notice that the complementary region $M \setminus B(x, 2r)$ can be covered by the  annuli
		\[
		A_k
		:=
		B(x,2^{k+1}r)\setminus B(x,2^kr),
		\qquad
		k\geq1.
		\]

		Hereafter we denote by $C$ any generic constant depending on $n, \, v_{0}$, $\mathscr B$ and $\mathscr C. $ If \(z\in A_k\), \(k\geq1\), and \(y\in B(x,r)\), then
		$
		d(y,z)
		\geq
		d(x,z)-d(x,y)
		\geq
		(2^k-1)r
		\geq
		2^{k-1}r.
		$
		Consequently, using \textup{\textbf{(A1)}}, we obtain
		\[
		t^{-n/2}
		\int_{B(x,r)}
		\exp\left(
		-\frac{d^2(y,z)}{\mathscr Ct}
		\right)d\mu(y)
		\leq
		Cv_0^{-1}r^nt^{-n/2}
		\exp\left(
		-\frac{4^{k-1}r^2}{\mathscr Ct}
		\right).
		\label{eq:B5-far-Gaussian}
		\]
		Set $p:=1+\beta$. Applying the maximal-separated-set argument
		from the proof of Proposition~\ref{prop:lower-curvature-improvement} at scale $r$, and using the
		volume estimates in Remark~\ref{rem:initial-assumptions} (a), we can cover $A_k$ by at most
		\[
		N_k\le C2^{kn}\exp(C2^kr)
		\]
		balls of radius $r$. Since $r\le1$, assumption \textup{\textbf{(A2)}} applies
		on each of these balls, and hence
		\[
		\int_{A_k}\ell^p(\cdot,0)\,d\mu
		\le
		C\varepsilon 2^{kn}\exp(C2^kr)r^{n-2p}.
		\]
		Since $\theta\le r^2$, the change of variables $u=r^2/t$ gives
		\[
		\int_0^\theta
		t^{-n/2}
		\exp\left(-\frac{4^{k-1}r^2}{\mathscr Ct}\right)\,dt
		\le
		r^{2-n}\int_1^\infty
		u^{n/2-2}
		\exp\left(-\frac{4^{k-1}}{\mathscr C}u\right)\,du
		\le
		Cr^{2-n}e^{-c4^k}.
		\]
		Thus the contribution of $A_k$ to $\mathcal J_2$ is bounded by
		$
		C\varepsilon r^{n-2\beta}
		2^{kn}\exp\left(C2^kr-c4^k\right).
		$
		Since $r\le1$,
		we conclude that
		$
		\mathcal J_2\le C\varepsilon r^{n-2\beta}.
		$
		
		Finally, as \(\varepsilon\leq1\), we deduce that $\mc J \leq C \, r^{n - 2\beta}$. The statement follows for $\mathscr A \geq 2 C =: \underline{\mathscr A}(n,  v_{0}, \mathscr B, \mathscr C)$.
	\end{proof}

	\subsection{Extension of Gaussian bound for the curvature defect (B6)}
	\label{subsec:Gaussian-curvature-improvement}
	We next improve the Gaussian bound \textup{\textbf{(B6)}}. The proof combines
	the Gaussian heat-kernel bound with a reproduction estimate for the
	weights defining $\widehat\ell_r$. We first establish this
	reproduction estimate.

	\begin{lemma}[Gaussian reproduction formula] \label{lem_complicated_reproduction_formula}
		Under the standing assumptions of this section, for any
		\[ \mathscr E > 0, \qquad \text{and} \qquad 
		\eta \leq \overline\eta(n, v_{0}, \beta, \mathscr A),
		\]
		there is a constant $C = C(n, v_{0}, \mathscr E) < \infty$ such that the following holds:
		
		For $t \in (0,  \tau]$, $s \in (0, t/8]$ and $x, z \in M$ we have
		\begin{equation}
		\frac1{s^{ n/2}} \int_M \exp\bigg({ - \frac{d^2 (x, \cdot)}{\frac14 \mathscr E \cdot t}  - \frac{d^2(\cdot,z)}{\mathscr E s} }\bigg)  d\mu_s \leq C \exp \bigg({ - \frac{d^2 (x,z)}{\mathscr E t} }\bigg).   \label{eq_rep_formula_Z2}
		\end{equation}
	\end{lemma}
	
	\begin{proof}
		Set $\rho := \sqrt{s}$ and denote by $\mc I$ the left-hand side of \eqref{eq_rep_formula_Z2}.
		By assumption \textup{\textbf{(A1)}},  we get the bound
		\begin{equation} \label{eq_reprod_y_prime_ball} \mc I \leq 
		v_{0}^{-1} \rho^{-2n} \int_M \int_{B(y, \rho)}  \exp\bigg({ - \frac{d^2 (x, y)}{\frac14 \mathscr E \cdot t}  - \frac{d^2(y,z)}{\mathscr E s} }\bigg)  d\mu(q) d\mu_s(y).
		\end{equation}
		Next, note that if $d(y, q) < \rho$, then the triangle inequality and $(a-b)^2 \geq  \frac1{2} a^2 -b^2$ yield
		\[
		\begin{aligned}
		\frac{d^2(x, y)}{\frac14 \mathscr E \cdot t}  + \frac{d^2(y,z)}{\mathscr E s} & \geq \frac{d^2(x, q)}{\frac12 \mathscr E \cdot t}  + \frac{d^2(q,z)}{2\mathscr E s} - \frac{\rho^2}{\frac14 \mathscr E \cdot t} - \frac{\rho^2}{\mathscr E s}  \\
		&\geq  \frac{d^2(x, q) + d^2(q,z)}{\frac12 \mathscr E \cdot t} + \frac{d^2 (q,z)}{4\mathscr E s} - \frac{2}{\mathscr E} 
		\geq \frac{d^2(x, z)}{\mathscr E t} + \frac{d^2(q,z)}{4\mathscr E s} - \frac{2}{\mathscr E},
		\end{aligned}
		\]
		where the second inequality follows from $\rho^2 = s \leq t/8$, while for the latter we applied the triangle inequality and $(a + b)^2 \leq 2(a^2 + b^2)$.
		
		Combining this with \eqref{eq_reprod_y_prime_ball} implies that $\ds \mc I  \leq  v_{0}^{-1} e^{2/\mathscr E}  \exp\bigg({ -\frac{d^2(x, z)}{\mathscr E t}   }\bigg) \rho^{-2n} \mc J$, with
		\[
		\begin{aligned}
		\mc J & =  \int_M \int_{B(y, \rho)}  \!\!\! \exp\bigg({ - \frac{d^2 (q,z)}{4\mathscr E s}  }\bigg)  d\mu(q) d\mu_s(y)  =  \int_M \int_{B(q, \rho)}  \!\!\!\exp\bigg({ - \frac{d^2(q,z)}{4\mathscr E s}  }\bigg)  d\mu_s(y) d\mu(q) \\
		& =  \int_M |B(\cdot, \rho)|_s \cdot  \exp\bigg({ - \frac{d^2 (\cdot,z)}{4\mathscr E s}  }\bigg)   d\mu.
		\end{aligned}
		\]
		In order to estimate $|B(\cdot, \rho)|_s$ we proceed as follows: from \eqref{eq:Ricci-lower-from-defect}, we get $\scal_{g(t)} \geq -n (n-1) \ell$ and hence $\partial_t(\log d \mu_t )  \leq n(n-1)\ell$. Accordingly, we have
		\[ |B(q, \rho)|_s :=  \int_{B(q, \rho)} d\mu_s \leq \int_{B(q, \rho)} \exp \bigg(n(n-1) \int_0^s \ell(\cdot,\sigma)\,d\sigma \bigg) d\mu \]
		and next we can bound the right hand side by using Lemma \ref{lem:exponential-integrability}. In fact, as $\rho^2  \leq \tau$ a priori assumption \textup{\textbf{(B5)}} and \textup{\textbf{(A1)}} imply
		\[\int_0^{\rho^2}\fint_{B(q, \rho)} \ell^{1+\beta}(\cdot,\sigma)\,d\mu\,d\sigma \leq \mathscr A v_{0}^{-1} \rho^{-2 \beta},\]
		where $\fint_U F := \frac1{|U|}\int_U F$. 
		Now using Lemma \ref{lem:exponential-integrability} for $A = \max\{\mathscr A \cdot v_{ 0}^{-1}, n(n-1)\}$ we can find a constant 
		$\overline\eta = \overline\eta(n, v_0, \beta, \mathscr A)$ so that, if a priori assumption \textup{\textbf{(B2)}} holds for $\eta \leq \overline\eta$, then
		we get 
		\[ |B (q, \rho)|_s    \leq 2 |B(q, \rho)|\leq 2 v_{ 0}^{-1} \rho^n, \]
		where the latter follows from assumption \textup{\textbf{(A1)}}. So the left-hand side of \eqref{eq_rep_formula_Z2} can be bounded as follows
		\[ \mc I \leq 2 v_{0}^{-2} e^{2/\mathscr E}  \exp\bigg({ -\frac{d^2 (x, z)}{\mathscr E t}   }\bigg) \cdot  \int_M  \rho^{-n} \exp\bigg({ - \frac{d^2 (\cdot,z)}{4\mathscr E s}  }\bigg)   d\mu. \]
		Notice that the remaining Gaussian integral is uniformly bounded. Indeed,
		decomposing \(M\) into the annuli
		$
		B(z,(k+1)\rho)\setminus B(z,k\rho)
		$
		and using the volume-growth estimate from
		Remark~\ref{rem:initial-assumptions}\textnormal{(a)}, we obtain
		\[
		\begin{aligned}
		&\rho^{-n}
		\int_M
		\exp\left(
		-\frac{d^2(\cdot,z)}{4\mathscr E\rho^2}
		\right)d\mu
		\leq
		C(n,v_0)
		\sum_{k=0}^{\infty}
		(k+1)^n
		\exp\left(
		C(n,v_0)(k+1)\rho-\frac{k^2}{4\mathscr E}
		\right).
		\end{aligned}
		\]
		Since \(\rho=\sqrt s\leq1\), the series is bounded by a constant
		depending only on \(n,v_0,\mathscr E\).
	\end{proof}
	
	Now we are in a position to prove the remaining extension property.
	
	\begin{prop} \label{improved_Gaussian}
		If
		\[\eta \leq \min\{1, \overline\eta(n, v_{0}, \beta, \mathscr A)\}, \,
		\mathscr C \geq \underline{\mathscr C}(n, v_{0}, \mathscr K), \,
		\mathscr B \geq \underline{\mathscr B} (n, v_{0}, \mathscr K) \, \text{and} \quad
		\eps \leq \overline{\eps} (n, \beta, v_{0}, \mathscr K, \mathscr B, \mathscr C), \]
		then the Ricci flow admits an extension beyond $\tau$ on which
		\textup{\textbf{(B6)}} remains valid.
	\end{prop}
	
	\begin{proof}
		If $(n-1) (n-2)\eta \leq 1$, then the evolution inequality from Proposition \ref{prop:ell-barrier} and \textup{\textbf{(B2)}}
		imply that
		\begin{equation} \label{eq_evol_ineq_ell_1t_term}
		\partial_t \ell \leq \triangle \ell + \scal  \ell + (n-1)(n-2) \ell^2 \leq \triangle \ell + \scal \ell + \frac{\ell}{t}  
		\end{equation}
		and hence
		\begin{equation} \label{din_l1}
		\partial_t \big( t
		^{-1} \ell \big) \leq \triangle \big( t^{-1} \ell \big) + \scal \big( t^{-1} \ell \big). 
		\end{equation}
		
		Let \(G(x,t;y,s)\) be the heat kernel introduced in \eqref{heat_ker_def}. Now, for fixed \((x,t)\), set
		\[ F(s) := \int_M G(x,t;y,s)\frac{\ell(y,s)}{s}\,d\mu_s(y), \qquad 0<s<t. \]
		Integration by parts and \eqref{din_l1} lead to  $F'(s)\leq0$.  Hence, for every  $\sigma \in [\frac{t}{8},t)$, we have $F(\sigma)\leq F\left(\frac{t}{8}\right).$ Letting \(\sigma\nearrow t\) and using $ G(x,t;\cdot,\sigma)\longrightarrow\delta_x, $ we deduce
		\begin{equation} \label{eq_ell_bound_to_alpha_t0}
		\ell (x,t) \leq 8 \int_M G(x,t; \cdot,\tfrac{t}{8}) \, \ell (\cdot, \tfrac{t}{8}) d\mu_{t/8}. 
		\end{equation}
		
		Next, plugging $(n-1) (n-2)\eta^{1-\beta} \leq 1$ into the evolution inequality of $\ell$ and using again \textup{\textbf{(B2)}}, we reach
		\[ \partial_t \ell \leq \triangle \ell + \scal  \ell + \frac{\ell^{1+\beta}}{t^{1-\beta}}. \]
		By integration of this inequality against the heat kernel $G$ in space, we deduce
		\[ \frac{d}{d s} \int_M  G(x,t; \cdot, s) \, \ell (\cdot, s) d\mu_{s} \leq \frac1{s^{ 1-\beta}}  \int_M  G(x,t; \cdot, s) \ell^{1+\beta} (\cdot, s) d\mu_s. \]
		Integrating this inequality over the time-interval $[0, t/8]$ and combining it with \eqref{eq_ell_bound_to_alpha_t0} yields
		\begin{equation} \label{eq_ell_bound_2_integrals}
		\ell (x, t) \leq 8 \int_M G(x, t; \cdot, 0) \,\ell (\cdot, 0) d\mu 
		+ 8 \int_0^{\frac{t}{8}} \int_M G(x, t; \cdot,s) \frac{1}{s^{1-\beta}} \ell^{1+\beta}(\cdot,s) d\mu_{s} ds. 
		\end{equation}
		
		From the Gaussian estimate on $G$  (cf.~Proposition \ref{thm:heat}) and H\"older's inequality, we can find a constant $\mathscr C_\ast = \mathscr C_\ast (n, v_{0}, \mathscr K)<\infty $ such that
		\begin{align}  \label{eq:first_integral_bound}
		\int_M G(x, t; \cdot, 0) \ell (\cdot, 0) d\mu &\leq 
		\int_M \frac{\mathscr C_\ast}{t^{n/2}} \exp \bigg({ - \frac{d^2(x,\cdot)}{\mathscr C_\ast t}}\bigg) \ell (\cdot,0) d\mu  \nonumber \\
		& \leq   \bigg[ \int_M \tfrac{\mathscr C_\ast}{t^{n/2}} e^{-\frac{d^2 (x,\cdot)}{\mathscr C_\ast t}} \ell^{1 + \beta}(\cdot, 0) d\mu  \bigg]^{\frac1{1+\beta}}  \cdot
		\bigg[\int_M \tfrac{\mathscr C_\ast}{t^{n/2}}  e^{-\frac{d^2(x,\cdot)}{\mathscr C_\ast t}}  d\mu  \bigg]^{\frac{\beta}{1+\beta}}  \nonumber \\
		& \leq C \widehat{\ell}_{\sqrt{t}} (x). 
		\end{align}
		for some $C = C(n, v_0, \mathscr K) < \infty$. Here the second factor is uniformly bounded by the volume-growth estimate in Remark~\ref{rem:initial-assumptions}\textup{(a)}, while the first factor is controlled by \(\widehat\ell_{\sqrt t}(x)\) after choosing the Gaussian constant \(\mathscr C\) in its definition so that $\mathscr C\geq \mathscr C_\ast(n, v_0, \mathscr K)$.

		Let us now estimate the double integral $\mc I$ in  \eqref{eq_ell_bound_2_integrals}.
		First, using a priori assumption \textup{\textbf{(B6)}} and the Gaussian heat kernel bound from Proposition  \ref{thm:heat}, we find that
		\[
		\begin{aligned}
		\mc I 
		& \leq \mathscr B^{1+\beta} \int_0^{ t/8} \int_M \frac{\mathscr C_\ast}{(t - s)^{n/2}} \exp\bigg({ - \frac{d^2_{s} (x, \cdot)}{\mathscr C_\ast (t - s)} }\bigg) \cdot  \frac{1}{s^{1-\beta}} \widehat{\ell}^{1+\beta}_{\sqrt{s}} \, d\mu_{s} \, ds   \\
		& \leq \frac{2^{n/2} \mathscr B^{1+\beta}\mathscr C_\ast}{t^{n/2}} \int_0^{t/8} \int_M \exp\bigg({ - \frac{d^2_{s} (x, \cdot)}{ \mathscr C_\ast t} }\bigg) \cdot  \frac{1}{s^{1-\beta}} \, \widehat{\ell}^{1+\beta}_{\sqrt{s}} \,  d\mu_{s} ds,
		\end{aligned}
		\]
		where the latter follows because, for \(0<s\leq t/8\), we have $\frac{t}{2}\leq t-s\leq t.$

		A priori assumption \textup{\textbf{(B1)}} and Hamilton's  distance distortion bound \eqref{net_timet}
		ensure that for some $C_1 = C_1 (n, v_0, \mathscr K, \mathscr B)$
		we have
		\[
		\begin{aligned}
		\frac{\mc I}{\mathscr B} & \leq \frac{C_1}{t^{n/2}} \int_0^{ t/8} \int_M \exp\bigg({ - \frac{d^2 (x, \cdot)}{2 \mathscr C_\ast t} }\bigg) \cdot  \frac{1}{s^{1-\beta}} \widehat{\ell}^{\, 1+\beta}_{\sqrt{s}}  d\mu_{s} \, ds \\
		& = \frac{C_1}{t^{n/2}} \!\!\int_0^{ t/8} \!\!\!\int_{M \times M} \!\!\!\exp\bigg({ - \frac{d^2(x, y)}{2\mathscr C_\ast t} - \frac{d^2(y,z)}{\mathscr C s}}\bigg) \frac{1}{s^{1-\beta}}  \frac1{s^{ n/2}} \ell^{1+\beta} (z,0) d\mu(z) d\mu_{s} (y) ds.
		\end{aligned}
		\]
		Here the equality follows from the definition of \(\widehat\ell_{\sqrt s}\) and Tonelli's theorem.
		
		Now assume furthermore that $\mathscr C \geq 8 \mathscr C_\ast$ and $\eta \leq \overline\eta(n, v_{0}, \beta, \mathscr A)$, where $\overline\eta$ is the constant from Lemma~\ref{lem_complicated_reproduction_formula}. Then the latter can be applied with $\mathscr E = \mathscr C$ to guarantee that  we can find 
		a constant $C_2 = C_2 (n, v_{0}, \mathscr K, \mathscr B, \mathscr C)$ for which we get the estimate
		\[
		\begin{aligned}
		\frac{\mc I}{\mathscr B} & \leq  \frac{C_2}{t^{n/2}} \!\! \int_0^{ t/8} s^{\beta-1}\, ds  \int_M   \exp \bigg({ - \frac{d^2(x,\cdot)}{\mathscr C t} }\bigg) \ell^{1+\beta} (\cdot ,0) d\mu   = \frac{C_2}{\beta} \cdot \bigg( \frac{t}{8}\bigg)^\beta \cdot \widehat{\ell}^{\, 1+\beta}_{\sqrt{t}} (x)
		\\ & \leq 8^{-\beta} \, C_3 \, \eps^{\frac{\beta}{2}}  \cdot \widehat{\ell}_{\sqrt{t}} (x),
		\end{aligned}
		\]
		where the last inequality follows from \eqref{eq_lem_B2_l_hat_bound} and  $C_3 = C_3(n, \beta, v_{0}, \mathscr K, \mathscr B, \mathscr C)$.

		So combining the latter with  \eqref{eq_ell_bound_2_integrals} and \eqref{eq:first_integral_bound} gives us 
		\begin{align*}
		\ell (\cdot , t)   
		\leq 8 \bigg( \frac{C}{\mathscr B} + C_3 \eps^{\frac{\beta}{2}}  \bigg) \mathscr B \cdot \widehat{\ell}_{\sqrt{t}} \leq \tfrac12 \mathscr B \, \widehat{\ell}_{\sqrt{t}}
		\end{align*}
		provided that we first take $\mathscr B \geq \underline{\mathscr B}(n, v_{0}, \mathscr K)$, after which we choose
		$\eps \leq \overline{\eps}(n, \beta, v_{0}, \mathscr K, \mathscr B, \mathscr C)$. 
		
		By the preceding propositions, the flow extends to $[0,\tau']$
		with \textup{\textbf{(B1)--(B5)}} still valid. After decreasing $\tau'$,
		we may assume $\tau<\tau'<\min\{8\tau,\tau_0\}$.
		In estimating the second integral in
		\eqref{eq_ell_bound_2_integrals}, we used \textup{\textbf{(B6)}} only at
		times $s\leq\frac{t}{8}$. Thus the same argument applies for
		$\tau<t\leq\tau'$, since these earlier times lie below $\tau$,
		and proves \textup{\textbf{(B6)}} on the extension.
	\end{proof}

	\subsection{Closure of the bootstrap and continuation}
	
	We now combine the preceding improvements in an open--closed argument.
	After establishing \textup{\textbf{(B1)--(B6)}} on a short initial interval,
	we use these improvements to extend the flow with the same bounds
	up to the fixed time $\tau_0$. The following theorem is a more detailed
	version of Theorem~\ref{Lp-RF-existence_short}, including the orbifold
	case and the additional geometric and integral estimates.
	
	\begin{thm}\label{Lp-RF-existence}
		Given $2\leq n\in\mathbb N$, $v_0,D,\Upsilon>0$ and
		$\beta\in(0,\frac12)$, there are positive constants
		$v_1,\tau_0,\mathscr K,\mathscr A,\mathscr B,\mathscr C$ depending
		only on $(n,v_0,D)$, a constant $\alpha=\alpha(n,D)\in(0,1]$, and
		$\eps_0,\eta\in(0,1]$ depending only on $(n,v_0,D,\beta,\Upsilon)$
		such that the following holds.
		
		Let $0<\varepsilon\leq\varepsilon_0$, and let $(M^n,g)$ be a complete smooth effective Riemannian orbifold
		with bounded curvature satisfying \textup{\textbf{(A1)--(A3)}}, with this value of
		$\varepsilon$ in \textnormal{\bf (A2)}. Then there is
		a unique complete Ricci flow with bounded curvature and initial
		metric $g$ on $[0,\tau_0]$, satisfying \textup{\textbf{(B1)--(B6)}} with
		$\tau=\tau_0$. All metric and integral conditions are understood
		with respect to the orbifold distance and Riemannian measure.
		The same statement holds when $\ell$ denotes the negative part of
		the lowest eigenvalue of the curvature operator.
	\end{thm}
	
	\begin{proof}
		First suppose that $M$ is a manifold. Choose the constants in
		the order specified in  \eqref{order-constants}, with
		$\varepsilon_0$ satisfying all the smallness requirements in the
		preceding propositions, and with $v_1<v_0$ and $\mathscr B$ sufficiently
		large for the estimate below. Since these requirements are monotone
		in the defect parameter, the same choices apply to every
		$0<\varepsilon\leq\varepsilon_0$.
		
		 Bounded-curvature short-time existence gives
		a complete Ricci flow $(M,g(t))_{t\in[0,\tau]}$ with $g(0)=g$.
		If $\tau>0$ is chosen sufficiently small, then
		$|\Rm_{g(t)}|\leq2K_0$, where $K_0:=1+\sup_M|\Rm_g|$, and
		\textup{\textbf{(B1)--(B5)}} hold by \textup{\textbf{(A1)}} and the uniform short-time
		control of curvature, metrics and volume forms.
		
		To establish \textup{\textbf{(B6)}}, we use the same heat-kernel comparison
		and H\"older estimate as in the proof of
		Proposition~\ref{improved_Gaussian}. In the first inequality of
		\eqref{eq_evol_ineq_ell_1t_term}, the bound
		$\ell\leq C_nK_0$ allows us to bound the quadratic term by
		$C_nK_0\ell$. Comparison with the heat kernel and
		\eqref{eq:first_integral_bound} therefore give
		\[
		\ell(x,t)
		\leq e^{C_nK_0t}\int_M G(x,t;y,0)\ell(y,0)\,d\mu(y)
		\leq C e^{C_nK_0t}\widehat\ell_{\sqrt t}(x).
		\]
		Here the heat-kernel bound follows from \textup{\textbf{(B1)}} and
		\textup{\textbf{(B3)}}, and $C$ is independent of $K_0$.
		Taking $\mathscr B\geq2C$ and choosing $\tau$ sufficiently small
		therefore guarantees \textup{\textbf{(B6)}} as well.
		
		We now choose $\tau\leq\tau_0$ maximal such that the flow exists on
		$[0,\tau]$ and satisfies all a priori assumptions \textup{\textbf{(B1)--(B6)}}.
		Such a maximal endpoint is attained because these conditions are
		closed and \textup{\textbf{(B1)}} ensures bounded-curvature continuation at
		any positive finite endpoint. If $\tau<\tau_0$,
		Propositions~\ref{prop:upper-curvature-improvement}--\ref{improved_Gaussian}
		give an extension to $[0,\tau']$ with $\tau<\tau'<\tau_0$ on which
		all a priori assumptions remain valid. This contradicts
		maximality, so $\tau=\tau_0$. The constants are independent of $K_0$;
		only the initial choice of $\tau$ depends on the initial curvature
		bound.
		
		For the curvature-operator version, $\ell$ satisfies
		\eqref{eq:Ricci-lower-from-defect} and the first evolution inequality
		in \eqref{eq_evol_ineq_ell_1t_term} in the barrier sense, with a
		possibly different dimensional coefficient of $\ell^2$; see
		\cite[Proposition~2.2]{BamlerCabezasRivasWilking}.
		The blow-up limits used above again have nonnegative curvature
		operator, so the same proof applies after adjusting dimensional
		constants.
		
		The same argument applies to orbifolds. The short-time theory
		carries over directly, and the local curvature and parabolic
		estimates are applied equivariantly in uniformizing charts, while
		the integral and segment estimates
		use the orbifold Riemannian measure. Moreover, \textup{\textbf{(A1)}} bounds
		the order of every local isotropy group $\Gamma$ by
		$|\Gamma|\leq\frac{\omega_n}{v_0}$, so these arguments remain uniform.
		
		For the blow-up argument in
		Proposition~\ref{prop:upper-curvature-improvement}, orbifold
		Ricci-flow compactness gives a nonflat ancient noncollapsed limit
		with nonnegative curvature operator
		\cite[Proposition~5.5]{KleinerLottOrbifolds}. If the limit is
		noncompact, its positive asymptotic volume ratio contradicts
		\cite[Lemma~6.9]{KleinerLottOrbifolds}; a compact limit is excluded
		by the lower volume bound at arbitrarily large radii. The packing
		and volume-convergence argument for \textup{\textbf{(B3)}} applies to the
		underlying Alexandrov spaces. Thus all the preceding estimates
		and the open--closed argument hold in the orbifold case with the
		same dependence of the constants.
		
		Finally, taking $\tau:=\tau_0$, \textup{\textbf{(B1)}}, \textup{\textbf{(B6)}} and
		\eqref{eq_lem_B2_l_hat_bound} give the two curvature estimates in
		Theorem~\ref{Lp-RF-existence_short}, after increasing
		$C=C(n,v_0,D)$. This proves that theorem as well.
	\end{proof}
	
	\begin{proof}[Proof of Corollary \ref{cor:integral-curvature-gap}]
		Otherwise, choose counterexamples with integral defect tending to
		zero. At $t_*=\frac{1}{2}\min\{\tau,1\}$, the uniform Ricci flow
		estimates give curvature, noncollapsing and injectivity-radius bounds.
		Subdividing initial
		minimizing geodesics into segments of length at most one, the
		distance estimate \eqref{eq:B4} also gives a uniform diameter bound.
		Shi's estimates and smooth compactness yield a subsequential smooth
		limit on a closed manifold. Its curvature lies in the required
		cone because the pointwise defect tends to zero. Pulling back this
		metric by the convergence diffeomorphisms contradicts the choice
		of the sequence.
	\end{proof}

	\section{Construction of the approximators for Petrunin's Conjecture}
	\label{sec:approximators}
	
	In this section we prove Theorem~\ref{thm:smoothing}. We first
	construct smooth orbifold metrics approximating the polyhedral
	space in the Gromov--Hausdorff sense, by inductively smoothing the
	normal links and using their Ricci flows to fill in the successive
	strata. We then apply Theorem~\ref{Lp-RF-existence} on a common time
	interval and pass to a limiting Ricci flow. Its positive-time slices
	provide the desired smoothing.
	
	Let $(P,d_P)$ be a compact Euclidean polyhedral space of dimension
	$n\geq2$, without boundary and with nonnegative Alexandrov curvature. Fix a finite
	polyhedral decomposition of $P$, denote its $k$-skeleton by $P^{(k)}$,
	and set $P^{(-1)}=\emptyset$.
	
	We first choose auxiliary normal charts and compatible scale functions.
	\begin{lemma}\label{lem:face-scales-and-charts}
		For every $k$-face $F$, $0\leq k\leq n-2$, there are a function
		$\varphi_F\in C^0(F)\cap C^\infty(\operatorname{Int}F)$, positive
		on $\operatorname{Int}F$ and zero on $\partial F$, and a normal
		chart
		\begin{equation}\label{eq:face-collar}
		\Psi_F:D_F
		:=
		\{
		(p,r,\theta):
		p\in\operatorname{Int}F,\ 
		0<r<\varphi_F(p),\
		\theta\in L_F
		\}
		\longrightarrow U_F\subset P,
		\end{equation}
		where $L_F$ is the compact spherical polyhedral normal link.
		The map $\Psi_F$ is an open topological embedding into
		$P\setminus P^{(k)}$ and a local isometry for the product of the
		Euclidean metric on $F$ and the cone metric on $C(L_F)$. It extends
		across $r=0$, with the usual collapse of $L_F$, to a homeomorphism
		onto a neighborhood of $\operatorname{Int}F$ in
		$P\setminus P^{(k-1)}$. The charts are induced by the Euclidean
		normal projections, and their images are disjoint for distinct
		$k$-faces.
		
		For every proper face $F'\subset F$, let
		$\Lambda_{F,F'}\subset L_{F'}$ denote the closed spherical normal
		face of $F$ along $F'$. In the chart $\Psi_{F'}$, the face $F$
		corresponds to $\theta\in\Lambda_{F,F'}$, and on the common chart
		domain
		\begin{equation}\label{eq:face-scale-factorization}
		\varphi_F\circ\Psi_{F'}(p,r,\theta)
		=
		r\varphi'_{F,F'}(\theta),
		\end{equation}
		where $\varphi'_{F,F'}$ is smooth and positive on
		$\operatorname{Int}\Lambda_{F,F'}$, continuous on
		$\Lambda_{F,F'}$, and vanishes on its boundary.
		
		If $\partial F\neq\varnothing$, then $\varphi_F$ may be chosen so
		that, with $d_{\partial F}=\operatorname{dist}_F(\,\cdot\,,\partial F)$
		and suitable constants $0<c_\varphi\leq C_\varphi<\infty$,
		\begin{equation}\label{eq:face-scale-distance}
		c_\varphi d_{\partial F}
		\leq
		\varphi_F
		\leq
		C_\varphi d_{\partial F}
		\qquad\text{on }\operatorname{Int}F.
		\end{equation}
	\end{lemma}
	
	\begin{proof}
		For each face $F$, let
		\[
		\Psi'_F:D'_F\longrightarrow U'_F\subset P
		\]
		be the normal exponential chart on a neighborhood of the cone axis
		in $\operatorname{Int}F\times C(L_F)$. We construct the functions
		$\varphi_F$ by descending induction on the dimension of $F$, and
		then define $\Psi_F$ as the restriction of $\Psi'_F$ to
		$D_F=\{0<r<\varphi_F\}$.
		
		Fix a face $F$ and suppose that the functions associated with all
		higher-dimensional faces have already been constructed. We define
		$\varphi_F$ near $\partial F$ by increasing induction on the
		dimension of its proper faces. For each proper face
		$F'\subset\partial F$, choose a positive angular function
		$\varphi'_{F,F'}$ on
		$\operatorname{Int}\Lambda_{F,F'}$, compatible with the functions
		already prescribed near the smaller incident faces. Such functions
		are obtained by extending the boundary prescriptions smoothly and
		using angular cutoffs which leave them unchanged. In the normal
		coordinates associated with $F'$, set
		\[
		\varphi_F=r\varphi'_{F,F'}(\theta).
		\]
		The compatibility of the angular functions makes these homogeneous
		definitions agree on overlaps.
		
		After all proper faces have been treated, extend $\varphi_F$
		positively over the remaining interior of $F$ using a cutoff
		supported away from the boundary, and set $\varphi_F=0$ on
		$\partial F$. For a vertex, choose $\varphi_F$ to be a positive
		constant.
		
		At each stage, multiply $\varphi_F$ by a sufficiently small positive
		constant so that $D_F\subset D'_F$ and all previously constructed
		factorizations hold throughout $D_F$. The compatible homogeneous
		descriptions near the boundary and compactness away from it ensure
		that this is possible. Thus
		\eqref{eq:face-scale-factorization} holds throughout every final
		chart. Since the face lattice is finite, the same homogeneous descriptions
		give \eqref{eq:face-scale-distance} with uniform constants
		$c_\varphi$ and $C_\varphi$. They also imply that, for every $j\geq1$,
		\[
		\bigl|(\nabla^F)^j\varphi_F\bigr|_{g_F}
		\leq
		C_j d_{\partial F}^{\,1-j}
		\qquad\text{on }\operatorname{Int}F,
		\]
		with constants uniform over the finitely many faces.

		Finally, multiplying all functions $\varphi_F$ by a sufficiently
		small common constant makes $U_F$ and $U_E$ disjoint whenever neither
		face contains the other. This preserves
		\eqref{eq:face-scale-factorization} and
		\eqref{eq:face-scale-distance}, and in particular gives the required
		disjointness for distinct faces of the same dimension.
	\end{proof}
	
	Our smoothing construction depends on parameters
	\[
	\mathbf a=(a_{n-2},\ldots,a_0),
	\qquad
	0<a_0\le\cdots\le a_{n-2}\le1.
	\]
	The parameter $a_k$ governs
	the scale at which we smooth the $k$-skeleton and is chosen
	sufficiently small after $a_{n-2},\ldots,a_{k+1}$ have been fixed.
	We start with the flat metric $g_{n-1}$ on
	$M_{n-1}:=P\setminus P^{(n-2)}$, whose metric completion is
	$(P,d_P)$. We then smooth the remaining strata in decreasing
	order of dimension. For $k=n-2,\ldots,0$, we construct a smooth
	structure and a Riemannian metric
	\[
	g_k=g_{(a_{n-2},\ldots,a_k),k}
	\quad\text{on}\quad M_k:=P\setminus P^{(k-1)}.
	\]
	Hereafter, smooth structures, metrics, and maps are understood in
	the effective orbifold sense when necessary.
	
	We suppress the parameter dependence below. The smooth structure on
	$M_k$ restricts to the previously constructed one on $M_{k+1}$, and
	the underlying inclusion
	\[
	M_{k+1}\hookrightarrow M_k,
	\]
	viewed as a map between subsets of $P$, is independent of the
	parameters. Before smoothing the $k$-faces, we assume the following properties
	of $(M_{k+1},g_{k+1})$.
	\begin{itemize}
		\item The natural inclusion $M_{k+1}\hookrightarrow P$ extends to a
		homeomorphism from the metric completion of $(M_{k+1},g_{k+1})$
		onto $P$. We denote the resulting distance on $P$ by
		$d_P^{k+1}$. In particular, $d_P^{k+1}$ induces the original topology of $P$.
		\item For every $d$-face $F$ with $d\leq k$, a neighborhood of
		$\operatorname{Int}F$ in $(P,d_P^{k+1})$ is locally isometric to
		a Cartesian product of $\mathbb R^d$ with a cone.
		More precisely, the map $\Psi_F$ is a local isometry from
		\[
		D_F\subset (\operatorname{Int}F,g_F)\times C(L_F,d_F^{k+1})
		\]
		onto the corresponding punctured neighborhood of
		$\operatorname{Int}F$. Here $g_F$ is the Euclidean face metric
		and $d_F^{k+1}$ is a distance on $L_F$.
		
		\item If $d=k$, then $d_F^{k+1} =d_{\widehat g_F}$
		for a smooth Riemannian orbifold metric $\widehat g_F$ on $L_F$, and
		\begin{equation}\label{eq:face-metric-form}
		\gamma_F:=\Psi_F^*g_{k+1}
		=g_F+dr^2+r^2\hat g_F.
		\end{equation}
	\end{itemize}
	The initial flat metric $(M_{n-1},g_{n-1})$ has these properties.
	Now let $0\leq k\leq n-2$ and suppose that $(M_{k+1},g_{k+1})$
	has been constructed. For each $k$-face $F$, we will modify $\gamma_F$
	on $\{r<a_k\varphi_F\}$ to a metric $\widetilde\gamma_F$ 
	such that
	$(\Psi_F^{-1})^*\widetilde\gamma_F$ extends smoothly over
	$\operatorname{Int}F$.

	If $k=n-2$, the normal link metric is
	$\hat g_F=\kappa_F^2d\theta^2$, so the cone in
	\eqref{eq:face-metric-form} is taken over a circle of length
	$2\pi\kappa_F$, where $0<\kappa_F\leq1$ by the nonnegative
	Alexandrov curvature of $P$ (see e.g. \cite[Theorem 4.2.14]{BBI}).
	We smooth it by replacing a small neighborhood of the vertex in
	each normal cone with a cap at scale $a_k\varphi_F$.
	In terms of the metric, this takes the following form.
	Choose a smooth nondecreasing
	function $A_F$ equal to $\kappa_F^2$ on $[0,\frac{1}{4}]$ and to $1$ on
	$[\frac{3}{4},\infty)$, and set
	\begin{equation}\label{eq:codimension-two-cap}
	\widetilde\gamma_F
	=g_F+dr^2+\kappa_F^2r^2d\theta^2
	+(A-1)\frac{a_k^2(\varphi_F\,dr-r\,d\varphi_F)^2}{a_k^2\varphi_F^2+r^2},
	\qquad A=A_F\left(\frac{r}{a_k\varphi_F}\right).
	\end{equation}
	This agrees with $\gamma_F$ for $r\geq\frac{3}{4}a_k\varphi_F$ and
	extends smoothly over $F$ at $r=0$. Since $|d\varphi_F|_{g_F}$ is
	uniformly bounded, choosing $a_k$ sufficiently small gives
	$a_k|d\varphi_F|_{g_F}\leq1$, which ensures positive definiteness. Indeed, a direct computation of the radial--face block shows that
	its determinant is positive under this condition. 
	The factorization in Lemma~\ref{lem:face-scales-and-charts} makes
	this correction purely angular in the normal cone at every
	incident lower face, as we will verify below.
	
	If $k<n-2$, consider the maximal Ricci flow $\hat g_F(t)$,
	$t\in[0,T_F)$, on $L_F$, with $\hat g_F(0)=\hat g_F$.
	Suppose that this flow develops
	a round singularity at its finite final time. That is, $L_F$ is a
	quotient of $S^{n-k-1}$ by a finite group and the flow converges,
	after rescaling, to a round metric as $t\nearrow T_{F}$.
	Otherwise, we stop and say that the construction has failed for
	this choice of parameters. We use this flow to construct a smooth
	family of metrics $h_F(s)$, $s\in[0,1]$, on $L_F$ with the
	following properties:
	\begin{itemize}
		\item If $s\in[0,0.1]$, then $h_F(s)=h_F(0)$ has constant
		sectional curvature $1$.
		\item For $s\in[0,0.3]$, the metric $h_F(s)$ satisfies
		\[
		\Rm_{h_F(s)}-\mathscr I_{h_F(s)}
		\in\mathcal C_{\mathrm{geom}\geq0},
		\]
		and is $a_{k+1}$-close to $h_F(0)$ in the $C^\infty$ topology.
		\item If $s\in[0.2,1]$, then
		$h_F(s)=c_F(s)\hat g_F(t_F(s))$, where $c_F(s)>0$ and
		$t_F(s)\in[0,T_F)$ are smooth.
		\item If $s\in[0.4,1]$, then $h_F(s)=\hat g_F$.
	\end{itemize}
	Here closeness is measured using a fixed distance inducing the
	$C^\infty$ topology, after identifying the round endpoint with
	a fixed unit-round metric. All these choices depend only on
	$a_{n-2},\ldots,a_{k+1}$ and are made before choosing $a_k$.
	We take $c_F(s)=1$ whenever $t_F(s)\leq\tau_F$, where
	$\tau_F>0$ is independent of the smoothing parameters.
	Whenever $c_F(s)\neq1$, we require the same shifted curvature
	condition as above. We normalize near extinction so that the metrics become unit round
	in the limit; the portions with $t_F(s)\geq\tau_F$ are chosen in a
	smoothly precompact family, up to pullback. On each compact parameter
	family, we choose these pullbacks continuously and use the resulting
	fixed identifications as the gauges for the paths $h_F$. The
	finite-order bounds used below include derivatives in $s$ in these
	gauges. The uniform choices will be justified in Lemma~\ref{lem:approximator-properties}.
	We now set
	\begin{equation}\label{eq:logarithmic-link-interpolation}
	\widetilde\gamma_F=g_F+dr^2+r^2
	\begin{cases}
	h_F\bigl((\frac{r}{a_k\varphi_F})^{a_k}\bigr),
	&0<r\leq\frac{1}{2}a_k\varphi_F,\\
	\hat g_F,&\frac{1}{2}a_k\varphi_F<r<a_k\varphi_F.
	\end{cases}
	\end{equation}
	The two formulas agree near $r=\frac{1}{2} a_k \varphi_F$, and
	$\widetilde\gamma_F=\gamma_F$ outside $\{ r< \frac{1}{2} a_k \varphi_F\}$.
	On $\{r<a_k(\frac{1}{10})^{\frac{1}{a_k}}\varphi_F\}$,
	the angular metric is fixed and unit round, so
	$\widetilde\gamma_F$ extends smoothly over $F$, possibly as an orbifold metric.
	
	In either case, use $\widetilde\gamma_F$ in the disjoint face tubes,
	retain $g_{k+1}$ elsewhere, and extend over the cone axes. This
	constructs $(M_k,g_k)$.
	For fixed parameters, the new and old metrics are uniformly
	comparable on $M_{k+1}$: in the second case this follows from the
	compact smooth family $h_F(s)$. Since the filled faces have
	codimension at least two, curves can be approximated by curves
	avoiding them. Induction therefore identifies the metric completion
	with $P$, defining $d_P^k$. The comparison constants here may depend
	on the parameters.
	
	The modifications preserve the exact product-cone structure at every
	remaining face. Indeed, let $\mathcal V$ denote the radial dilation
	field on the normal cone of an incident lower-dimensional face. By
	\eqref{eq:face-scale-factorization}, both $r$ and $\varphi_F$ are
	homogeneous of degree one with respect to $\mathcal V$. Hence
	\[
	(\varphi_F\,dr-r\,d\varphi_F)(\mathcal V)=0,
	\qquad
	\mathcal V\left(\frac{r}{\varphi_F}\right)=0.
	\]
	Thus the correction in \eqref{eq:codimension-two-cap} is angular and
	homogeneous of degree two. The same holds for
	\eqref{eq:logarithmic-link-interpolation}, since its link parameter
	depends only on $r/\varphi_F$ and its data are independent of the
	point on the lower face. Consequently, after each stage, every
	unfilled face still has a neighborhood which is an exact product of
	the face with a metric cone.
	
	We next quantify the freezing of the angular metric. Fix a stage
	$k<n-2$, a $k$-face $F$, and the preceding parameter
	tuple $\mathbf b=(a_{n-2},\ldots,a_{k+1})$. Write $a=a_k$ and set
	\begin{equation} \label{def_s}
	s=\left(\frac{r}{a\varphi_F}\right)^a
	\end{equation}
	on the modified part of the normal tube. Let $\nabla^F$ denote the
	Levi--Civita connection of the face metric $g_F$. In the present cone
	coordinates, $\mathcal V=r\partial_r$, and direct differentiation
	gives
	\[
	\mathcal V s=as,
	\qquad
	\mathcal V^2s=a^2s,
	\qquad
	r\nabla^Fs
	=
	-as\,\frac{r}{\varphi_F}\nabla^F\varphi_F,
	\]
	and
	\[
	r^2(\nabla^F)^2s
	=
	a(a+1)s\frac{r^2}{\varphi_F^2}
	d\varphi_F\otimes d\varphi_F
	-
	as\frac{r^2}{\varphi_F}(\nabla^F)^2\varphi_F.
	\]
	The higher derivatives have the same structure: every derivative
	falling on $s$ produces a factor $a$, while the remaining terms are
	scale-invariant combinations of covariant derivatives of
	$\varphi_F$ with respect to $g_F$.
	
Let $Q$ be a product box in the modified tube, centered at
$(p_0,r_0,\theta_0)$, whose face projection has diameter at most
$Cr_0$, whose radial coordinate lies in $[cr_0,Cr_0]$, and whose
angular component lies in a fixed controlled link chart. If $s_0$ is the value of $s$ at the center
	of $Q$, set
	\[
	\gamma_{F,s_0}^{\mathrm{fr}}
	=
	g_F+dr^2+r^2h_F(s_0).
	\]
	In the following estimates, $\nabla$ and the tensor norms are taken
	with respect to $\gamma_{F,s_0}^{\mathrm{fr}}$. For every $N\geq2$ there is a constant
	$C_N(\mathbf b)<\infty$ such that
	\begin{equation}\label{eq:freezing-metric}
	\sum_{j=0}^{N}
	r_0^j
	\left|
	\nabla^j
	\bigl(
	\widetilde\gamma_F-\gamma_{F,s_0}^{\mathrm{fr}}
	\bigr)
	\right|
	\leq
	C_N(\mathbf b)a,
	\end{equation}
	and
	\begin{equation}\label{eq:freezing-curvature}
	\sum_{j=0}^{N-2}
	r_0^{j+2}
	\left|
	\nabla^j
	\bigl(
	\Rm_{\widetilde\gamma_F}
	-
	\Rm_{\gamma_{F,s_0}^{\mathrm{fr}}}
	\bigr)
	\right|
	\leq
	C_N(\mathbf b)a.
	\end{equation}
	
	After rescaling the metric by $r_0^{-2}$, the boxes $Q$ have
	uniformly controlled size. Since $Q$ lies in the modified region and
	$r$ is comparable to $r_0$, one has
	$
	r_0\leq Ca\varphi_F
$
	throughout $Q$. The homogeneous description in Lemma~\ref{lem:face-scales-and-charts}, together
	with \eqref{eq:face-scale-distance} and the derivative bounds above,
	therefore gives, for every fixed $j\geq1$,
	\[
	\frac{r_0}{\varphi_F}
	+
	\frac{r_0^j
		|(\nabla^F)^j\varphi_F|_{g_F}}{\varphi_F}
	+
	r_0^j
	|(\nabla^F)^j\log\varphi_F|_{g_F}
	\leq
	C_j.
	\]
	Thus all scale-invariant combinations of derivatives of
	$\varphi_F$ arising from the chain rule are uniformly controlled on
	$Q$. Moreover, for 
	$
	s
	$ as in \eqref{def_s} every derivative 
	contains a factor $a$. Together with the uniform finite-order bounds
	for the paths $h_F$ in the chosen gauges, the chain rule and the
	curvature formulas for families of cone metrics give
	\eqref{eq:freezing-metric}--\eqref{eq:freezing-curvature}, uniformly
	over the faces at stage $k$.

	We finally choose the smoothing parameters quantitatively. Fix
	functions $\omega_k(u)\searrow0$ and $N_k(u)\nearrow\infty$ as
	$u\searrow0$. For fixed $u>0$, the admissible preceding tuples
	$\mathbf b=(a_{n-2},\ldots,a_{k+1})$ with $a_{k+1}=u$ range in a
	relatively compact parameter set. The corresponding link metrics,
	interpolation paths, and constants in
	\eqref{eq:freezing-metric}--\eqref{eq:freezing-curvature} are
	therefore uniformly bounded. Hence there is a positive threshold
	$\tau_k(u)$ such that $a_k\leq\tau_k(u)$ makes the freezing errors
	at most $\omega_k(u)$ through order $N_k(u)$, simultaneously for
	all faces at stage $k$.
	
	Replacing $\tau_k$ by a positive continuous minorant, we obtain a
	continuous function
	$\overline a_k:(0,\overline a_{n-2}]\to(0,\infty)$, with
	$\overline a_k(u)\leq u$, such that
	\begin{equation}\label{eq:freezing-parameter-choice}
	0<a_k\leq\overline a_k(a_{k+1})
	\end{equation}
	has the preceding property. These functions depend only on $P$ and
	the fixed auxiliary choices and are independent of the exponent used
	later in the Morrey estimate.
	
	The same compact-parameter argument is used to choose the paths
	$h_F$ on the fixed link structures. For fixed $a_{k+1}=u>0$, the
	admissible preceding tuples lie in a compact family of successful
	constructions. Normalized-flow stability near the round limit and a
	small rescaling provide the required interpolation to a continuously
	chosen unit-round endpoint, with a strict margin in the shifted cone
	condition. Thus the paths depend continuously on the preceding
	parameters in the chosen gauges, with uniform finite-order bounds,
	including derivatives in $s$, as used in
	\eqref{eq:freezing-metric}--\eqref{eq:freezing-curvature}. No
	canonical choice of the round endpoint is required.

	If all required link flows exist and become round, iteration
	completes the construction and gives
	\[
	M=M_0=P,
	\qquad
	g_{\mathbf a}=g_0,
	\qquad
	d_{g_{\mathbf a}}=d_{g_{\mathbf a}}=d_P^0.
	\]
	Here $M=M_0=P$ denotes an identification of the underlying
	topological spaces; $M$ carries the resulting smooth effective
	orbifold structure. All objects constructed above depend on
	$\mathbf a=(a_{n-2},\ldots,a_0)$. We use the choices just described
	throughout the remainder of the section.
	
	\begin{lemma}\label{lem:approximator-properties}
		Fix $(P,d_P)$ and consider the metrics
		$g_k=g_{(a_{n-2},\ldots,a_k),k}$ constructed above.
		There are a constant $\overline a_{n-2}>0$ and continuous
		functions $\overline a_k:(0,\overline a_{n-2}]\to(0,\infty)$,
		$0\leq k<n-2$, with $\overline a_k(a)\leq a$, depending only
		on $P$ and the fixed auxiliary choices, such that the
		construction of $g_{\mathbf a,k}$ does not fail whenever
		$\mathbf a=(a_{n-2},\ldots,a_k)$ satisfies
		\begin{equation}\label{eq:admissible-smoothing-parameters}
		0<a_{n-2}\leq\overline a_{n-2},
		\qquad
		0<a_j\leq\overline a_j(a_{j+1})
		\quad(k\leq j<n-2).
		\end{equation}
		Moreover, let $\mathbf a_i=(a_{n-2,i},\ldots,a_{k,i})\to0$
		satisfy \eqref{eq:admissible-smoothing-parameters}, put
		$g_i=g_{\mathbf a_i,k}$, and choose $x_i\in M_k$ and
		$0<r_i\leq1$ with $r_i\to r_\infty\in[0,1]$.
		After passing to a subsequence, the metric completions of
		$(M_k,r_i^{-2}g_i,r_i^{-n}d\mu_{g_i},x_i)$ converge in the
		pointed measured Gromov--Hausdorff sense, and the following holds.
		
		\begin{enumerate}[label=\textnormal{(\roman*)}]
			
			\item\label{item:approximator-models}
			If $r_\infty>0$, the limit is
			$(P,r_\infty^{-1}d_P,r_\infty^{-n}\mathcal H_P^n,x_\infty)$
			for some $x_\infty\in P$. If $r_\infty=0$, the following models can occur:
			
			\begin{enumerate}[
				label=\textnormal{(\alph*)},
				ref=\textnormal{(\roman{enumi})(\alph*)}
				]
				
				\item\label{item:blowup-caps}
				A tangent cone of $P$, or a Euclidean factor times a complete
				two-dimensional manifold of nonnegative curvature, asymptotic
				to the normal tangent cone of a codimension-two face of $P$.
				
				\item\label{item:blowup-cones}
				A Euclidean factor times a cone over a smooth limiting link
				metric,
				\begin{equation}\label{eq:frozen-link-model}
				\mathbb R^{\dim F}\times C(L_F,\hat g_\infty) \qquad \text{with} \quad k\leq\dim F<n-2.
				\end{equation}
				Here  $(L_F,\hat g_\infty)$ is a smooth
				pointed Cheeger--Gromov limit, modulo parabolic rescaling, of
				positive-time slices of the Ricci flows starting from the link
				metrics $\hat g_{\mathbf a_i,F}$. The cone metric $g_{\mathbb R^{\dim F}}+dr^2+r^2\hat g_\infty$ has
				nonnegative cosectional curvature.
				
				\item\label{item:blowup-inherited}
				A positive-time slice of a Ricci flow emerging from a model in
				{\rm(a)} or {\rm(b)} (in the sense of
				Definition~\ref{def:flow-from-metric-space}) and having
				nonnegative cosectional curvature.
				
				\item\label{item:blowup-flat}
				A flat manifold or flat orbifold quotient arising when the
				basepoint in one of the conical models escapes to infinite
				radial distance.
				
			\end{enumerate}
			
			\noindent Convergence is locally smooth in the Cheeger--Gromov sense away
			from the limiting strata inherited from $P^{(n-2)}$ in
			\ref{item:blowup-caps}, away from the cone axis
			$\mathbb R^{\dim F}\times\{o\}$ in
			\ref{item:blowup-cones}, and everywhere in
			\ref{item:blowup-inherited} and \ref{item:blowup-flat}, where
			$o$ denotes the cone vertex. If $r_\infty>0$, convergence is
			locally smooth on $P\setminus P^{(n-2)}$.
			
			\item\label{item:approximator-estimates}
			There is a constant $v_0=v_0(P)>0$ such that, for all sufficiently
			large $i$, if
			$B_{g_i}(x_i,2r_i)\subset M_k$, with balls taken in the metric
			completion, then \textnormal{\bf (A1)} holds for $x=x_i$ and $r=r_i$
			with this same constant $v_0$. For every fixed
			$\alpha\in(0,1]$, the generalized
			$(\alpha,D,\Upsilon)$-segment inequality holds on $(M_k,g_i)$, for
			all pairs at distance at most one, with constants $D<\infty$ and
			$\Upsilon=\Upsilon(\alpha)<\infty$ independent of $i$. For every
			fixed $\beta\in(0,\frac12)$, \textnormal{\bf (A2)} also holds with these
			choices and with $\varepsilon_i=\varepsilon_i(\beta)\to0$.
			
			\item\label{item:approximator-flows}
			If $k=0$, the flows starting from $g_{\mathbf a_i,0}$ exist on a uniform time
			interval and subconverge smoothly for $t\in(0,T)$ to a Ricci flow
			emerging from $P$.
			
		\end{enumerate}
	\end{lemma}
	
	\begin{proof}
		We argue by induction on the dimension and, within each dimension,
		by decreasing induction on $k$.
		
		Suppose first that $k=n-2$. If $n=2$, the functions $\varphi_F$
		are constant and the caps have nonnegative curvature. Assume
		$n\geq3$, fix a codimension-two face $F$, and set
		$a:=a_{n-2}$ and $h:=a\varphi_F$. At
		$y_0\in\operatorname{Int}F$, compare the cap metric with the affine
		model obtained by replacing $h$ by its first-order Taylor polynomial
		at $y_0$. If $b:=|dh(y_0)|>0$, then, after rotating the face
		coordinates, the nontrivial factor of this model is the cone over
		\[
		B(u)\,du^2+\kappa_F^2\sin^2u\,d\theta^2,
		\qquad
		B(u)=1-\bigl(1-A_F(s)\bigr)
		\frac{1+b^2s^2}{1+s^2},
		\qquad
		s=\frac{\tan u}{b}.
		\]
		As $b\leq1$, one has
		$\kappa_F^2\leq B\leq1$ and $B_u\geq0$, and hence the Gaussian
		curvature of the link is
		\[
		K=\frac1B+\frac{\cot u\,B_u}{2B^2}\geq1.
		\]
		Thus the affine model has nonnegative cosectional curvature; when
		$b=0$, it is a product with the ordinary nonnegatively curved
		surface cap.
		
		In smooth normal Cartesian coordinates, comparison with this model
		and the derivative bounds at the end of Lemma~\ref{lem:face-scales-and-charts} give, throughout
		the modified region,
		\[
		\bigl|
		\Rm_{\widetilde\gamma_F}
		-
		\Rm_{\gamma_F^{\mathrm{lin}}}
		\bigr|
		\leq
		C\left(
		\frac{|D^2h|+|dh|\,|D^2h|}{h}
		+|D^2h|^2+|D^3h|
		\right)
		\leq
		Cd_{\partial F}^{-2}.
		\]
		The comparison remains smooth at the cone axis because
		$A_F=\kappa_F^2$ near zero, while outside the cap the metric is
		unchanged and flat. Consequently,
		\begin{equation}\label{eq:cap-defect}
		\ell_{g_{n-2}}
		\leq
		Cd_{\partial F}^{-2}
		\mathbf 1_{\{r\leq Ca_{n-2}d_{\partial F}\}}.
		\end{equation}

		Estimate~\eqref{eq:cap-defect} gives a uniform Morrey bound on balls
		of every radius. Indeed, let $d=d_{\partial F}$ at the center of a ball of
		radius $\rho$. If $\rho\leq a_{n-2}d$, then
		\[
		\rho^{2q-n}
		\int_{B(x,\rho)}
		\ell_{g_{n-2}}^q\,d\mu_{g_{n-2}}
		\leq
		C\left(\frac{\rho}{d}\right)^{2q}.
		\]
		If $a_{n-2}d\leq\rho\leq cd$, then the transverse area of the
		support gives
		\[
		\rho^{2q-n}
		\int_{B(x,\rho)}
		\ell_{g_{n-2}}^q\,d\mu_{g_{n-2}}
		\leq
		Ca_{n-2}^2
		\left(\frac{\rho}{d}\right)^{2q-2}.
		\]
		Near $\partial F$, the homogeneous face coordinates reduce the
		estimate to
		$
		Ca_{n-2}^2\int_0^{C\rho}s^{2-2q}\,ds,
		$
		together with the Euclidean factors tangent to the incident face.
		Thus, for every $1<q<\frac32$,
		\begin{equation}\label{eq:initial-cap-morrey}
		\sup_{x,\,0<\rho\leq1}
		\rho^{2q-n}
		\int_{B(x,\rho)}
		\ell_{g_{n-2}}^q\,d\mu_{g_{n-2}}
		\longrightarrow0
		\qquad\text{as }a_{n-2}\longrightarrow0.
		\end{equation}
		Here $q<\frac32$ is precisely the integrability condition at
		$\partial F$.
		
		We also verify the base case of Assertion~\ref{item:approximator-models}.
		Comparing $r_i$ with the distance to $\partial F$ and with the cap
		scale $a_{n-2,i}\varphi_F$, and using the homogeneity in
		Lemma~\ref{lem:face-scales-and-charts}, we obtain either a tangent
		cone of $P$ or the product of a Euclidean factor with a complete
		nonnegatively curved surface asymptotic to the normal cone of a
		codimension-two face. Away from $\partial F$, according as
		$a_{n-2,i}\varphi_F/r_i$ tends to zero, to a positive finite limit,
		or to infinity, the normal factor converges to the original flat
		cone, to a rescaled cap, or to its smooth tangent plane. If the
		rescaled distance to $\partial F$ remains bounded, the same argument
		passes to an incident lower-dimensional face and terminates after
		finitely many steps.
		
		Convergence is smooth away from the inherited strata. Since normal
		fibers at radius $r$ have diameter at most $Cr$ and measure at most
		$Cr^m$, it extends to pointed measured convergence of the metric
		completions. The limiting models have uniform two-sided volume bounds
		and satisfy the usual segment inequality. Together with
		\eqref{eq:initial-cap-morrey}, this proves
		Assertion~\ref{item:approximator-estimates} at the initial stage.
		If $n=2$, then $k=0$. Applying
		Theorem~\ref{Lp-RF-existence} to the capped metrics and using the
		preceding pointed measured convergence proves
		Assertion~\ref{item:approximator-flows} directly.
		
		Assume now that $k<n-2$ and that the assertions hold at the
		preceding stages. We retain the previously chosen bounds for
		$a_{k+1},\ldots,a_{n-2}$ and choose $a_k$ according to
		\eqref{eq:freezing-parameter-choice}. Thus the errors in
		\eqref{eq:freezing-metric}--\eqref{eq:freezing-curvature} tend to
		zero through an order tending to infinity.
		
		We first transfer the inductive estimates to the normal links.
		Consider an exact product-cone supplied by the preceding stage,
		$\mathbb R^d\times C(L_F^m,h)$, where $n=d+m+1$. Its defect is
		\begin{equation}\label{eq:ambient-link-defect}
		\ell_{\mathbb R^d\times C(L_F,h)}
		=
		r^{-2}\ell_h^{\mathrm{sph}},
		\qquad
		\ell_h^{\mathrm{sph}}
		:=
		\inf\left\{
		b\geq0:
		\Rm_h-\mathscr I_h+b\mathscr I_h
		\in\mathcal C_{\mathrm{geom}\geq0}
		\right\}.
		\end{equation}
		Indeed, the angular curvature operator is
		$r^{-2}(\Rm_h-\mathscr I_h)$ and all mixed curvatures vanish.
		Moreover, $\ell_h\leq\ell_h^{\mathrm{sph}}$, since
		$\mathscr I_h\in\mathcal C_{\mathrm{geom}\geq0}$.
		
		Fix a uniform scale $\sigma_0>0$ sufficiently small that the
		preceding product-cone charts contain the required interior shells
		at every link scale $\sigma\leq\sigma_0$. Applying the ambient
		estimates on such a shell and integrating over its radial and
		Euclidean factors gives uniform two-sided volume bounds for $h$ and
		\begin{equation}\label{eq:link-initial-morrey}
		\int_{B_h(z,\sigma)}
		(\ell_h^{\mathrm{sph}})^p\,d\mu_h
		\leq
		C\varepsilon\sigma^{m-2p},
		\qquad
		0<\sigma\leq\sigma_0.
		\end{equation}
		The generalized segment inequality transfers on the same scales: working in a fixed interior shell of radial size comparable to one,
		the ambient curves remain in a shell and face box whose radial and
		Euclidean widths are $O(\sigma)$, and integration in these
		$d+1$ directions converts the ambient density bound into
		$C\sigma^{-m}$ on the link. Endpoint separation controls the
		constant-speed reparametrization.
		
		We apply Theorem~\ref{Lp-RF-existence} to the fixed rescaling
		$\sigma_0^{-2}h$. Its scales at most one correspond precisely to the
		scales $\sigma\leq\sigma_0$ considered above. Thus the rescaled link
		metrics satisfy the hypotheses of that theorem uniformly; returning
		to $h$ changes the resulting constants and existence time only by
		fixed factors.

		We now establish a compactness statement for the link flows.
		
		\begin{claimA}
			Let $F$ be a face with $k\leq\dim F<n-2$, and let
			$\mathbf a_i=(a_{n-2,i},\ldots,a_{k+1,i})\to0$ satisfy
			\eqref{eq:admissible-smoothing-parameters}. Let $\lambda_i\geq1$
			with $\lambda_i\to\lambda_\infty\in[1,\infty]$, and consider the
			rescaled link flows
			\begin{equation}\label{eq:rescaled-link-flow}
			g'_i(t)
			=
			\lambda_i^2
			\hat g_{\mathbf a_i,F}
			\left(\frac{t}{\lambda_i^2}\right),
			\qquad
			0\leq t<\lambda_i^2T_{\mathbf a_i,F}.
			\end{equation}
			Suppose that
			$(L_F,g'_i(0),z_i)$ converges in the pointed
			Gromov--Hausdorff sense to $(X,d_X,z_\infty)$. If
			$\lambda_\infty<\infty$, then $X$ is a spherical polyhedral normal
			link of $P$, with its distance multiplied by $\lambda_\infty$; if
			$\lambda_\infty=\infty$, then $X$ is a Euclidean polyhedral space,
			a Euclidean factor times a cone over a smooth Riemannian orbifold,
			or a smooth Riemannian orbifold. The flows $g'_i(t)$ subconverge
			smoothly on compact positive-time subintervals to a maximal Ricci
			flow $g'_\infty(t)$ emerging from $X$ in the sense of
			Definition~\ref{def:flow-from-metric-space}.
			
			For any $t_i\searrow0$, set $t'_i=\lambda_i^2t_i$. If
			$t'_i\to0$, then $(L_F,g'_i(t'_i),z_i)$ converges to
			$(X,d_X,z_\infty)$ in the pointed measured Gromov--Hausdorff sense
			and locally smoothly wherever the initial convergence is smooth. If
			$t'_i\to t'_\infty\in(0,\infty)$, the slices converge smoothly to
			$g'_\infty(t'_\infty)$; if $t'_i\to\infty$, they subconverge
			smoothly to a flat limit. Finally, if $\lambda_\infty=1$, then
			\begin{equation}\label{eq:limiting-link-spherical-bound}
			\Rm_{g'_\infty(t)}-\mathscr I_{g'_\infty(t)}
			\in\mathcal C_{\mathrm{geom}\geq0},
			\end{equation}
			and $g'_\infty(t)$ develops a spherical singularity.
		\end{claimA}
		
		\begin{proof}[Proof of Claim A]
			The uniform link estimates allow us to apply
			Theorem~\ref{Lp-RF-existence}, yielding a common short-time interval,
			uniform curvature and noncollapsing bounds, and, by
			\textup{\textbf{(B6)}}, vanishing defect at positive times. If
			$\lambda_i\to\infty$, this interval becomes arbitrarily long after
			rescaling, and compactness gives convergence on every compact
			positive-time interval. If $(\lambda_i)$ remains bounded, compactness
			first gives convergence on a common short interval; continuous
			dependence and bounded-curvature uniqueness, applied successively
			from positive-time slices, extend it to every compact subinterval of
			the maximal limiting interval.
			
			Suppose that $t'_i\to0$. On compact subsets where the initial
			metrics converge smoothly, pseudolocality
			\cite[Theorem~10.3]{P1} and local derivative estimates give local
			smooth convergence at times $t'_i$. To obtain metric convergence,
			fix $0<\delta<1$ and choose a finite $\frac{\delta}{2}$-net of
			regular points in a sufficiently large bounded ball of $X$.
			Distances between regular points can be approximated by regular
			paths. Local smooth convergence along finitely many such paths
			controls expansion between the net points, while
			\eqref{net_timet} controls contraction by $C\sqrt{t'_i}$.
			The same contraction estimate places every point of a bounded
			time-$t'_i$ ball in a slightly larger initial ball. Using
			\textup{\textbf{(B4)}} and then letting $i\to\infty$ and
			$\delta\searrow0$ proves pointed Gromov--Hausdorff convergence to
			$X$.

			For measured convergence, write $m=\dim X$. We first note that
		$
			|B_{g_i'(t)}(z,r)|\leq Cr^m
			$
			uniformly on bounded regions. For $r\leq\sqrt t$, this follows from
			the Ricci lower bound supplied by \textnormal{\bf (B1)} and Bishop
			comparison. For $r\geq\sqrt t$, distance contraction places the
			evolving ball in an initial ball of radius $Cr$; covering the latter
			by $C(r/\sqrt t)^m$ initial balls of radius $\sqrt t$, using
			\textnormal{\bf (B4)} to place each in an evolving ball of radius
			$C\sqrt t$, and applying the preceding estimate gives the claim.
			The omitted strata have zero $m$-dimensional Hausdorff measure.
			Covering them by balls with arbitrarily small sums of the $m$-th
			powers of their radii gives neighborhoods of uniformly small
			measure in the approximating slices. Smooth convergence on the
			complement identifies the limiting measure.
			
			If $t'_i\to t'_\infty\in(0,\infty)$, smooth convergence follows
			from flow compactness. If $t'_i\to\infty$, then the short-time curvature estimate for the
			unrescaled link flow gives
			\[
			\left|\Rm_{g'_i(t'_i)}\right|
			=
			\lambda_i^{-2}
			\left|\Rm_{\hat g_{\mathbf a_i,F}(t_i)}\right|
			\leq
			\frac{C}{\lambda_i^2t_i}
			=
			\frac{C}{t'_i}
			\longrightarrow0.
			\]
			Together with noncollapsing, this gives a smooth flat subsequential
			limit.
			
			It remains to prove
			\eqref{eq:limiting-link-spherical-bound}. Let $e$ be the defect of
			$\Rm_{g'_\infty(t)}-b(t)\mathscr I_{g'_\infty(t)}$, where
			$b(t)=(1-2(m-1)t)^{-1}$, for small $t>0$.
			Nonnegative cosectional curvature gives $0\leq e\leq b$.
			The supporting-hyperplane argument, using \eqref{eq:Q-shift},
			Hamilton invariance and the nonnegative Ricci tensor of the shifted
			operator, gives
			\[
			(\partial_t-\Delta)e
			\leq
			b'-2(m-1)(b-e)^2
			\leq
			4(m-1)be
			\]
			in the barrier sense. Since scalar curvature is nonnegative,
			\[
			\frac{d}{dt}
			\int e\,d\mu_{g'_\infty(t)}
			\leq
			4(m-1)b(t)
			\int e\,d\mu_{g'_\infty(t)}.
			\]
			Smooth convergence on the regular part and measured convergence
			give $\int e\,d\mu_{g'_\infty(t)}\to0$ as $t\searrow0$.  
			Gronwall's inequality gives $e\equiv0$, and hence $
			\Rm_{g'_\infty(t)}-b(t) \mathscr I_{g'_\infty(t)}
			\in\mathcal C_{\mathrm{geom}\geq0}.$
			Since $b(t)>1$ for $t>0$, the shifted condition
			\eqref{eq:limiting-link-spherical-bound} is strict at every
			positive time. The rounding theorems \cite{BW2,HamSurfaces} imply that the
			limiting flow develops a spherical singularity.
		\end{proof}
		
		Claim A shows that the construction does not fail. Otherwise, after
		decreasing $\overline a_{n-2}$, there would be an admissible
		sequence of preceding parameter tuples tending to zero for which a
		link flow does not become round. Applying Claim A with $\lambda_i=1$ gives a limiting flow satisfying
		the strict shifted condition above, since $b(t)>1$ for every
		$t>0$. Smooth convergence therefore gives the same strict condition
		for the approximating link flows at a fixed positive time. The rounding theorems \cite{BW2,HamSurfaces} give a
		contradiction.
		
		B\"ohm--Wilking's theorem applies to orbifolds in dimensions at
		least three. The two-dimensional links are smooth surfaces, since
		their cone points were removed at the codimension two stage, so
		Hamilton's theorem applies. Smooth dependence on this compact
		positive-time family gives the normalized paths $h_F$ used in the
		construction.
		
		We now prove Assertion~\ref{item:approximator-models}. At fixed
		positive scales, the supports of the modifications shrink to
		$P^{(n-2)}$, giving the asserted limit when $r_\infty>0$. Assume that $r_\infty=0$. By
		\eqref{eq:freezing-metric}--\eqref{eq:freezing-curvature}, every
		rescaled filling annulus is asymptotic, through an order tending to
		infinity, to a cone with frozen link metric. All models outside the
		new filling regions are covered by the preceding induction stage.
		It therefore remains to consider pointed limits of
		$\mathbb R^{\dim F}\times C(L_F,h_i)$, where
		$h_i=h_{\mathbf a_i,F}(s_i)$.
		
		If $s_i\leq0.3$ along a subsequence, the link metrics converge
		smoothly to a unit-round metric. If the cone basepoints remain at
		bounded radial distance, the corresponding cones give a model in
		\ref{item:blowup-cones}; if they escape to infinity, the pointed
		limit is flat and belongs to \ref{item:blowup-flat}.
		
		Otherwise, $h_i=c_i\hat g_{\mathbf a_i,F}(t_i)$. If $t_i$ stays
		bounded away from zero, Claim A and the normalized round regime give
		smooth pointed limits of the link metrics. The corresponding cone
		limits belong to \ref{item:blowup-cones} when the radial basepoints
		remain bounded and to \ref{item:blowup-flat} when they escape to
		infinity.
		
		Suppose that $t_i\to0$. If the cone basepoints remain at bounded
		distance from the vertex, Claim A with $\lambda_i=1$ recovers an
		original spherical link, giving \ref{item:blowup-caps}. If the
		basepoints escape, set $\lambda_i=d(o,x'_i)\to\infty$ and
		$t'_i=\lambda_i^2t_i$. If $t'_i\to0$, the limit is covered by the
		preceding induction stage. If
		$t'_i\to t'_\infty\in(0,\infty)$, Claim A gives a model in
		\ref{item:blowup-inherited}. Restarting at such a slice only
		advances the same bounded-curvature flow, by standard short-time
		bounded-curvature uniqueness. If $t'_i\to\infty$, the limit is
		flat and belongs to \ref{item:blowup-flat}.
		
		Smooth annular convergence extends to the metric completions because
		the normal fibers at radius $r$ have diameter at most $Cr$ and
		measure at most $Cr^m$. The same argument applies across face
		boundaries by the exact conical compatibility of the construction.
		Thus no further models arise, proving
		Assertion~\ref{item:approximator-models}.
		
		We now prove the volume and segment parts of
		Assertion~\ref{item:approximator-estimates}. The volume bounds follow from the pointed measured limits in
		Assertion~\ref{item:approximator-models}. Since there are only
		finitely many stages and face charts, taking the minimum of the
		resulting lower-volume constants gives a single
		$v_0=v_0(P)>0$, valid for all sufficiently large $i$.
		
		 Each limiting model satisfies the usual segment inequality, and its
		intrinsic regular part is convex. The additional strata excluded
		from smooth convergence have zero $n$-dimensional Hausdorff measure,
		so almost every segment in the endpoint families avoids them
		(and, in the flat-cone case, also avoids the cone axis). Shrinking
		the endpoint balls from radius $\alpha$ to $\alpha/4$, we may
		therefore retain a positive-probability family of segments contained
		in a compact subset of the smooth-convergence region. Smooth
		convergence preserves these segments and keeps their endpoints in
		the original $\alpha$-balls. After constant-speed
		reparametrization, they give the generalized segment inequality,
		with $D$ uniform and $\Upsilon$ depending only on $\alpha$.
		
		It remains to prove \textup{\textbf{(A2)}}. The next claim provides the
		stronger-exponent estimate needed to transfer the link-flow bounds
		back to all balls in the filled metric.
		\begin{claimB}
			Fix $1<p<\frac32$. In the setting of Claim A, let $t_i\to0$, put
			$t'_i=\lambda_i^2t_i$ and $h_i=g'_i(t'_i)$, and write
			$m=\dim L_F$. Then the shifted defects
			$
			\ell'_i
			=
			\lambda_i^{-2}
			\ell_{\hat g_{\mathbf a_i,F}(t_i)}^{\mathrm{sph}}
			$
			satisfy
			\vspace*{-0.1cm}
			\begin{equation}\label{eq:link-p-morrey}
			\sup_{z,\,0<\sigma\leq1}
			\sigma^{2p-m}
			\int_{B_{h_i}(z,\sigma)}
			(\ell'_i)^p\,d\mu_{h_i}
			\longrightarrow0.
			\end{equation}
		\end{claimB}
		
		\begin{proof}[Proof of Claim B]
			Choose $p<q<\frac32$. The initial estimate
			\eqref{eq:link-initial-morrey}, applied with exponent $q$, and the
			evolving-volume estimate in the proof of
			Lemma~\ref{lem_complicated_reproduction_formula} give
			\begin{equation}\label{eq:link-flow-spatial-morrey}
			\int_{B_{h_i}(z,\sigma)}
			(\ell'_i)^q\,d\mu_{h_i}
			\leq
			C\left(
			\varepsilon_i\sigma^{m-2q}
			+
			\lambda_i^{-2q}\sigma^m
			\right),
			\qquad
			0<\sigma\leq1,
			\end{equation}
			where $\varepsilon_i=\varepsilon_i(q)\to0$. Indeed, we apply the analytic estimates with exponent
			$\beta=q-1$ and raise \textnormal{\bf (B6)} to the power $q$, so that
			the initial quantity being integrated is the $q$-th power of the
			defect. For $\sigma\leq\sqrt{t'_i}$, the estimate follows from the
			pointwise defect and volume bounds. If
			$\sigma\geq\sqrt{t'_i}$, distance contraction places the evolving
			ball in an enlarged initial ball. Fubini's theorem and the
			evolving-volume estimate from Lemma~\ref{lem_complicated_reproduction_formula} control the integration over
			its initial $\sqrt{t'_i}$-balls, while the contribution from outside
			the enlarged ball is summed over Gaussian annuli, as in the proof of
			Proposition~\ref{B5_improve}. The initial $q$-Morrey estimate then gives the first
			term in \eqref{eq:link-flow-spatial-morrey}, and the spherical shift
			gives the second.
			
			H\"older's inequality and the volume bound yield
			\begin{equation}\label{eq:link-holder}
			\sigma^{2p-m}
			\int_{B_{h_i}(z,\sigma)}
			(\ell'_i)^p\,d\mu_{h_i}
			\leq
			C\left(
			\varepsilon_i+
			\lambda_i^{-2q}\sigma^{2q}
			\right)^{p/q}.
			\end{equation}
			If $\lambda_i\to\infty$, this proves
			\eqref{eq:link-p-morrey}. Otherwise, for
			$\sigma\leq\sigma_0$ the limit superior in
			\eqref{eq:link-holder} is at most $C\sigma_0^{2p}$. For
			$\sigma\geq\sigma_0$, \eqref{eq:link-flow-spatial-morrey} gives a
			uniform total $L^q$ bound. Claim A gives measured convergence and
			smooth convergence away from the limiting strata, where
			$\ell'_i\to0$. Since these strata have zero $m$-dimensional measure,
			H\"older's inequality shows that
			$
			\int_{L_F}(\ell'_i)^p\,d\mu_{h_i}\longrightarrow0.
			$
			Letting $\sigma_0\searrow0$ proves
			\eqref{eq:link-p-morrey}.
		\end{proof}

		Fix now $\beta\in(0,\frac12)$ and put $p=1+\beta$. Choose
		$p<q<\frac32$. Claim B applies when the corresponding link-flow
		times tend to zero. At times bounded away from zero, the same
		conclusion follows from Claim A and smooth convergence, while for
		$s\leq0.3$ or $c_F(s)\neq1$ the spherical defect vanishes by
		construction. Thus, uniformly in $F$, $s$ and $z$,
		\begin{equation}\label{eq:interpolated-link-morrey}
		\int_{B_h(z,\sigma)}
		(\ell_h^{\mathrm{sph}})^p\,d\mu_h
		\leq
		\delta_i\sigma^{m-2p},
		\qquad
		0<\sigma\leq1,
		\qquad
		\delta_i\longrightarrow0.
		\end{equation}

		By \eqref{eq:freezing-curvature}, on every filling chart,
		\begin{equation}\label{eq:filling-defect}
		\ell_{\widetilde\gamma_F}
		\leq
		Cr^{-2}
		\left(
		\ell_{h_F(s)}^{\mathrm{sph}}+\eta_i
		\right),
		\qquad
		\eta_i\longrightarrow0.
		\end{equation}
		The estimate is uniform in the face point and in $s$, and the
		hierarchy defining $\eta_i$ is independent of $\beta$.
		
		Let $B$ be an ambient ball of radius $\rho$ whose center has radial
		distance $R$ from the cone axis. If $\rho\leq cR$, then $B$ lies in
		a uniformly controlled product box and its projection to the link
		has radius comparable to $\rho/R$. The face and radial factors
		contribute $C\rho^{d+1}$, while
		\eqref{eq:interpolated-link-morrey} and
		\eqref{eq:filling-defect} contribute
		\[
		R^{m-2p}
		\left(\frac{\rho}{R}\right)^{m-2p}
		=
		\rho^{m-2p}
		\]
		in the angular directions. Since $n=d+m+1$, we obtain
		\begin{equation}\label{eq:filling-ball-morrey}
		\int_B
		\ell_{\widetilde\gamma_F}^{\,p}\,d\mu
		\leq
		C(\delta_i+\eta_i^p)\rho^{n-2p}.
		\end{equation}

		If $\rho>cR$, decompose the normal tube intersecting $B$ into the
		annuli $2^{-j-1}\rho<r<2^{-j}\rho$. Their contributions are bounded
		by
		$
		C(\delta_i+\eta_i^p)
		2^{-j(m+1-2p)}\rho^{n-2p},
		$
		whose sum is finite because $m+1\geq3>2p$. Thus \eqref{eq:filling-ball-morrey} also holds for balls meeting the
		axis. It remains to consider balls crossing a face untreated at the
		current stage. On each incident filling tube, the chart of that face
		restricts the incident face-footpoints to a patch of volume at most
		$C\rho^d$ and its radial coordinate to $r\leq C\rho$. Hence its
		contribution is bounded by
		\[
		C(\delta_i+\eta_i^p)\rho^d
		\int_0^{C\rho}r^{m-2p}\,dr
		\leq
		C(\delta_i+\eta_i^p)\rho^{n-2p},
		\qquad
		n=d+m+1.
		\]
		The portions at radial distance larger than $C\rho$ are controlled
		by the product-box estimate. Since the incident faces and compatible
		charts have uniformly bounded multiplicity, while the genuinely
		unchanged region is covered by the preceding induction stage,
		\eqref{eq:filling-ball-morrey} follows for every ambient ball.
		Together with \eqref{eq:initial-cap-morrey} for the codimension-two
		caps, this completes Assertion~\ref{item:approximator-estimates}.

		Finally, if $k=0$, Assertion~\ref{item:approximator-estimates} and
		Theorem~\ref{Lp-RF-existence} give flows on a common time interval
		whose positive-time limit has nonnegative cosectional curvature.
		Assertion~\ref{item:approximator-models} and the distance estimates
		identify $(P,d_P)$ as its metric initial condition, proving
		Assertion~\ref{item:approximator-flows}.
	\end{proof}

	We are now in a position to put all the pieces together and prove the Smoothing Conjecture.
	\begin{proof}[Proof of Theorems~\ref{thm:smoothing} and~\ref{thm:smoothing-flow}]
		Choose an admissible sequence $\mathbf a_i\to0$. By
		Lemma~\ref{lem:approximator-properties} (ii) the corresponding orbifold metrics satisfy
		\textup{\textbf{(A1)--(A3)}} uniformly, with curvature defect tending to
		zero. Then, by Theorem~\ref{Lp-RF-existence}, the orbifold Ricci flows starting from
		$g_{\mathbf a_i}$ exist on a common interval $(0,\tau]$ satisfying
		\textup{\textbf{(B1)--(B6)}}.
		
		By Lemma~\ref{lem:approximator-properties} (iii), after passing to a subsequence, the flows converge
		smoothly on compact positive-time intervals to a Ricci flow
		$(\mathcal O,g_\infty(t))$ emerging from $(P,d_P)$ in the pointed
		sense. This convergence is in fact global; indeed, Assertion~{\rm(i)}
		gives
		$
		(M_i,g_i)\xrightarrow{GH}(P,d_P),
		$
		and hence the initial diameters are uniformly bounded. Subdividing an
		initial curve into segments of length at most one and applying
		\textup{\textbf{(B4)}} to each segment gives
		\[
		\diam_{g_i(t)}M_i
		\le
		2D\bigl(\diam_{g_i(0)}M_i+1\bigr),
		\qquad 0<t\le \tau.
		\]
		Thus the positive-time limit is compact. The global smooth orbifold
		convergence identifies its underlying space with that of $M_i$,
		which is homeomorphic to $P$. Now smooth convergence and the vanishing defect imply
		$
		\Rm_{g_\infty(t)}\in \mathcal C_{\mathrm{geom}\ge0}$ for 
		$t>0$, 
		while the definition of emergence from $P$ gives
		\[
		(\mathcal O,g_\infty(t))
		\xrightarrow[t\searrow0]{GH}
		(P,d_P).
		\]
		Thus any sequence $t_j\searrow0$ provides the required smooth
		orbifold approximations.
	\end{proof}

\end{document}